\documentclass[reqno]{amsart}
\newif\ifdraft\draftfalse
\ifdraft\usepackage[notcite, notref]{showkeys}\fi

\usepackage{amsfonts,amssymb, amsmath, mathtools, xypic}
\usepackage{oldgerm}
\usepackage{amscd}
\usepackage{bbm}

\newcommand\N[1]{N(#1)}

\newtheorem{proposition}[equation]{Proposition}
\newtheorem{theorem}[equation]{Theorem}

\newtheorem{lemma}[equation]{Lemma}

\newtheorem{corollary}[equation]{Corollary}

\newtheorem{warning}[equation]{Warning}

\theoremstyle{definition}

\newtheorem{defn}[equation]{Definition}

\theoremstyle{remark}
\newtheorem{remark}[equation]{Remark}

\newtheorem{ntt}[equation]{}

\newcommand{\nocontentsline}[3]{}
\newcommand{\tocless}[2]{\bgroup\let\addcontentsline=\nocontentsline#1{#2}\egroup}

\numberwithin{equation}{section}
\newcommand{\Rep}{\mathop{\mathrm{Rep}}\nolimits}
\newcommand{\Nil}{\mathop{\mathrm{Nilp}}\nolimits}

\newcommand{\Vect}{\mathop{\mathrm{Vect}}\nolimits}

\newcommand{\Spec}{\mathop{\mathrm{Spec}}\nolimits}
\newcommand{\GL}{\mathop{\mathrm{GL}}\nolimits}
\newcommand{\gl}{\mathop{\mathfrak{gl}}\nolimits}
\newcommand{\Split}{\mathop{\mathrm{Split}}\nolimits}
\newcommand{\Sym}{\mathop{\mathrm{Sym}}\nolimits}
\newcommand{\ClZar}{\mathop{\mathrm{ClZar}}\nolimits}
\newcommand{\nilp}{\mathop{\mathrm{nilp}}\nolimits}

\newcommand\der{\mathrm{der}}
\newcommand{\Gal}{\mathop{ \mathrm{Gal}}\nolimits}

\DeclareMathOperator{\id}{id}

\DeclareMathOperator{\spec}{sp}
\DeclareMathOperator{\Cent}{Cent}
\DeclareMathOperator{\Const}{Const}
\DeclareMathOperator{\const}{const}
\DeclareMathOperator{\diag}{diag}
\DeclareMathOperator{\spl}{split}
\DeclareMathOperator{\Standard}{\mathbf{Std}}

\def\define{\def}

\define\M{{\mathcal{M}}}
\define\V{{\mathcal{V}}}

\define\H{{\mathcal{H}}}
\define\ZZ{{\mathcal{Z}}}

\define\C{\mathbb{C}}
\define\R{\mathbb{R}}

\define\Z{\mathbb{Z}}

\define\gg{\mathfrak g}      
\define\half{\frac{1}{2}}
\define\hph{\hphantom}
\define\Res{\operatornamewithlimits{Res}}
\define\Im{\operatorname{Im}}

\define\Ext{\operatorname{Ext}}
\define\Hom{\operatorname{Hom}}

\define\reg{\operatorname{reg}}

\def\ad{\mathrm{ad\,}}

\newcommand\HS{\mathrm{HS}}
\newcommand\MHS{\mathrm{MHS}}
\newcommand\VMHS{\mathrm{VMHS}}

\newcommand\NF{\mathrm{NF}}

\define\a{\alpha}
\define\b{\beta}

\define\l{\lambda}

\newcommand\thh{\theta}

\define\beq{\begin{equation}}
\define\eeq{\end{equation}}

\newcommand\Ad{\mathop{\mathrm{Ad}}\nolimits}
\newcommand\Gr{\mathop{\mathrm{Gr}}\nolimits}
\newcommand\SL{\mathop{\mathrm{SL}}\nolimits}
\newcommand\ssl{\mathop{\mathrm{sl}}\nolimits}
\newcommand\End{\mathop{\mathrm{End}}\nolimits}
\newcommand\Aut{\mathop{\mathrm{Aut}}\nolimits}

\newcommand\CC{\mathbb C}

\newcommand\RR{\mathbb R}

\begin{document}

\title[Zero locus]{On the algebraicity of the zero locus of an
admissible normal function}

\author{Patrick Brosnan}
\address{Department of Mathematics\\
 University of Maryland\\
 College Park, MD 20742, USA}
\email{pbrosnan@umd.edu}
\author{Gregory Pearlstein}
\address{Department of Mathematics\\
 Michigan State University\\
 East Lansing, MI 48824, USA}
\email{gpearl@math.msu.edu}
\thanks{P.B. and G.~P. were respectively supported by NSF grants
DMS-1103269 and  DMS-1002625.}
%\subjclass{}

% \thanks{The authors would like to thank Phillip Griffiths who
%  generously shared his ideas on normal functions
%  with the authors during their stay at the Institute for Advanced
%  Study in 2004--2005, and Pierre Deligne for sharing his theory of 
%  the $\ssl_2$-splitting.  We would also like to thank Christian 
%  Schnell for helpful discussions and reviewing our proof of the two variable 
% version of Conjecture $(1.1)$ in June 2008. Finally, we thank Kazuya Kato, 
% Chikara Nakayama and Sampei Usui for their encouragement and helpful 
% discussions about the several variable $\SL_2$-orbit theorem.}

%\dedicatory{}

\begin{abstract} We show that the zero locus of an admissible normal function
on a smooth complex algebraic variety is algebraic.  

In Part 2 of the paper, which is an appendix, we compute the Tannakian
Galois group of the category of one-variable admissible real nilpotent 
orbits with split limit.    We then use the answer to recover an unpublished
theorem of Deligne, which characterizes the $\ssl_2$-splitting of a real mixed
Hodge structure.
\end{abstract}

\maketitle
% \tableofcontents

\tocless\section{Introduction}\label{s.Intro}

Let $\bar S$ be a complex manifold and $\mathcal H$ be an integral
variation of pure Hodge structure of weight $w<0$ which is defined on a
Zariski open subset $S$ of $\bar S$.  M.~Saito~\cite{SaitoANF}
defines an {\it admissible normal function on $S$ with respect to $\bar{S}$} to be an extension class 
 \begin{equation}
    \label{e.Variation}
    0\to \H\to \V\to \mathbb{Z}(0)\to 0
  \end{equation}
in the category $\VMHS(S)^{\ad}_{\bar{S}}$ of variations of mixed
Hodge structure on $S$ which are admissible relative to $\bar{S}$ (see
\emph{loc.\ cit.}).  Via Carlson's formula~\cite{Carlson}, an admissible 
normal function corresponds to a holomorphic section $\nu:S\to J(\mathcal H)$ 
of the associated family of generalized intermediate Jacobians 
$J(\mathcal H)\to S$ which satisfies a version of Griffiths horizontality 
and has controlled asymptotic behavior near the boundary of $S$ in $\bar S$.  
The purpose of this paper is to prove the following.

\begin{theorem}\label{theorem:main}
  With the above conventions, let $\nu:S\to J(\mathcal H)$ be an
  admissible normal function with respect to $\bar S$.  Let
  $\ZZ=\ZZ(\nu)=\{s\in S: \nu(s)=0\}$ denote the zero locus of $\nu$.
  Then the topological closure of $\ZZ$ in $\bar {S}$ is a closed
  analytic complex subspace of $\bar{S}$.
\end{theorem}

In this paper we use the notation $\NF(S,\H)^{\ad}_{\bar{S}}$ to
denote the group of normal functions $\nu:S\to J(\mathcal
H)$ which are admissible with respect to $\bar S$ as above.  If $S$ has
has an algebraic structure, then, by Nagata and Hironaka, there exists
some smooth algebraic compactification $\bar{S}$.  However, for $S$
algebraic the notion of admissibility is independent of the choice of
smooth compactification $\bar S$ of $S$~\cite[Remark 1.6
(i)]{SaitoANF}.  Therefore we follow~\cite{SaitoANF} and write
$\NF(S,\H)^{\ad}$ instead of $\NF(S,\H)^{\ad}_{\bar{S}}$ for the group
of admissible normal functions on $S$.  The following corollary is
then immediate from GAGA.

\begin{corollary}\label{corollary:gg} If $S$ is algebraic then
the zero locus of an admissible normal function $\nu:S\to J(\mathcal H)$
is an algebraic subvariety of $S$.
\end{corollary}

% After some reductions, we prove 
% Theorem~\ref{theorem:main} by writing down rather explicit local 
% equations for the closure $\bar{\ZZ}$ of $\ZZ$.  These equations are 
% given at the end of Section~\ref{s.Analyticity}.
For $w=-1$ the assertion of Corollary \eqref{corollary:gg} was a
conjecture of Phillip Griffiths and Mark Green.  At least in the case
that $w=-1$, Theorem~\eqref{theorem:main} has also been proved by
Christian Schnell~\cite{SchnellNeron}.  In Schnell's work it is a
consequence of the existence of a ``N\'{e}ron model'' extending the
family $J(\H)\to S$ over $\bar{S}$.  The full theorem is used in our
joint work~\cite{BPS} with Christian Schnell where we prove the
generalization of the theorem of Cattani, Deligne and Kaplan to
admissible variations of mixed Hodge structure.  The
paper~\cite{KNUModLog} by K.~Kato, C.~Nakayama and S.~Usui
indicates a proof of the theorem using the log classifying spaces
developed by those authors.

To prove Theorem~\ref{theorem:main}, it suffices
(Theorem~\ref{t.Reduction}) to consider the case that $D = \bar S
\setminus S$ is a normal crossing divisor.  Working locally on $S$, we
are then reduced to proving the following theorem.

\begin{theorem}\label{theorem:zero-locus-closure} 
   In the context of Theorem~\ref{theorem:main}, suppose
   $D:=\bar{S}\setminus S$ is a normal crossing divisor and that 
   $p\in D$ is an accumulation point of $\ZZ$.  Then there exists an open
  polydisk $P\subset\bar S$ containing $p$ and an analytic subvariety
  $A$ of $P$ such that $A\cap S = \ZZ\cap P$.
\end{theorem}

\tableofcontents

\subsection*{Overview}

The first part of this paper is devoted to the proof of 
Theorem~\ref{theorem:zero-locus-closure}. 
In section~\ref{s.Polydisks} we review and extend the theory of variations 
of real mixed Hodge structure in $\VMHS(\Delta^{*r})^{\ad}_{\Delta^r}$.  These
have a very concrete local normal form which can be expressed in terms of 
matrix-valued holomorphic functions.  We state several results concerning the 
asymptotics of these variations in section~\ref{s.Polydisks}, in particular, 
Theorem~\ref{theorem:conj-1} which is a strong boundedness result.  Some of
the proofs, in particular the proof of our main boundedness result 
Theorem~\ref{theorem:conj-1}, are deferred to Section~\ref{s.proof}.

In section~\ref{s.Reduct}, we reduce the proof of 
Theorem~\ref{theorem:zero-locus-closure} to the analogous statement on a 
polydisk.  In section~\ref{s.Analyticity} we use the results stated 
in section~\ref{s.Polydisks} to give an explicit system of equations for $\ZZ$ 
on a punctured polydisk $\Delta^{*r}$. 

The remainder of the first part of the paper is devoted to the proofs
of the results concerning variations of real mixed Hodge structure
stated in section~\ref{s.Polydisks}.

The second part of the paper is an appendix, which proves results of Deligne on the $\ssl_2$-splitting
stated in an unpublished letter to Cattani and Kaplan~\cite{D}.  We
interpret Deligne's results as the computation of the Tannakian Galois
group of the category $\Split_1$ consisting of admissible nilpotent
orbits (in one variable) whose limits are split over $\mathbb{R}$.
The group in question is very similar to Deligne's group $\mathcal{M}$
whose category of representations is the category of real mixed Hodge
structures~\cite{DeligneReal}.  Part of this work consists in
explaining what a one variable $\SL_2$-orbit is in terms of the
representation a real reductive group of rank $3$, which we call the
\emph{Schmid group}.  This is probably well-known to the experts and
implicit in Schmid's~\cite{Schmid}, but to the best of our knowledge
it has not yet been written down explicitly.

\subsection*{Notation:}  
\begin{itemize}
\item[{---}] $\Delta^r$ is a polydisk with holomorphic coordinates
 $s=(s_1,\dots,s_r)$;
\item[{---}] $\Delta^{*r}\subset\Delta^r$ is the set of points where
$s_1\cdots s_r\neq 0$;
\item[{---}] $U^r\subset\mathbb C^r$ is the product of upper half-planes;
\item[{---}] Points $z=(z_1,\dots,z_r)\in \mathbb C^r$ are written 
$z = x+iy$ where $x=(x_1,\dots,x_r)$, and $y=(y_1,\dots,y_r)\in\mathbb R^r$;
\item[{---}] $\pi:U^r\to \Delta^{*r}$ is the covering map given by
$s_j = e^{2\pi iz_j}$ for $j=1,\dots,r$.
\item[{---}] The underlying vector space of a filtration or mixed Hodge
structure is denoted $V$.  If $V$ is defined over a subring $R$ of $\mathbb C$
the associated $R$-module is denoted $V_R$;
\item[{---}] Elements of $\GL(V)$ act linearly on filtrations of $V$,
e.g.~ $(g.F)^p = g(F^p)$.  Elements of $\GL(V)$ act on endomorphisms
of $V$ via the adjoint action, i.e.~$g.Y = gYg^{-1}$.
\end{itemize}
% We try to maintain the convention that $z$ denotes the coordinate on the
% upper-half plane and $s$ denotes the coordinate on the open unit disk.

\subsection*{Acknowledgments} 
The authors would like to thank Phillip Griffiths who generously
shared his ideas on normal functions with the authors during their
stay at the Institute for Advanced Study in 2004--2005, and Pierre
Deligne for sharing his theory of the $\ssl_2$-splitting.  We would
also like to thank Christian Schnell for helpful discussions and
reviewing our proof of the two variable version of
Theorem~\eqref{theorem:main} in June 2008 and for several other
helpful comments.  (In particular, for telling us about Hironaka's
paper~\cite{Hironaka}.)  We would like to thank Claire Voisin for
correcting us on a point related to Lemma~\ref{l.SplitLemmaWeight} and
Kazuya Kato, Chikara Nakayama and Sampei Usui for their encouragement
and helpful discussions about the several variable $\SL_2$-orbit
theorem.  Finally, we would like to express our gratitude to the referee, who 
read the paper carefully and patiently and helped us considerably to improve it.

\part{Zero Loci of Admissible Normal Functions}\label{p.1}

\section{Admissible variations on the punctured polydisk}\label{s.Polydisks}

Here we collect results about the structure of variations of mixed
Hodge structure on the punctured polydisk $\Delta^{*r}$ which are admissible
relative to the polydisk $\Delta^r$.  In this section, all variations of mixed 
Hodge structure will be variations with real coefficients and the emphasis
will be on asymptotics.

\subsection*{Deligne Gradings} Given an increasing filtration 
$W$ of a finite dimensional vector space $V$ over a field of
characteristic zero, a grading of $V$ is a semisimple endomorphism $Y$
of $V$ such that, for each index $k\in\mathbb{Z}$, $W_k$ is the direct sum of
$W_{k-1}$ and the $k$-eigenspace $E_k(Y)$ for each index $k$.  By a
theorem of Deligne \cite{HodgeII}, a mixed Hodge structure $(F,W)$
induces a unique, functorial decomposition 
\beq V_{\C} = 
\bigoplus_{r,s}\, I^{r,s}_{(F,W)}\text{ or simply } \bigoplus_{r,s} I^{r,s}\label{eq:Deligne-bigrading} 
\eeq of the
underlying complex vector space $V_{\C}$ such that
\begin{itemize}
\item[{(a)}] $F^p = \oplus_{r\geq p}\, I^{r,s}$; 
\item[{(b)}] $W_k = \oplus_{r+s\leq k}\, I^{r,s}$;
\item[{(c)}] $\bar I^{p,q} = I^{q,p} \mod \oplus_{r<q,s<p}\, I^{r,s}$.
\end{itemize}
Here $F$ is the Hodge filtration and $W$ is the weight filtration.

In particular, a mixed Hodge structure $(F,W)$ induces a grading 
$Y_{(F,W)}$ of $V_{\C}$ by the requirement that $Y_{(F,W)}$ acts
as multiplication by $p+q$ on $I^{p,q}$.  
We sometimes write $V^{r,s}$ instead of $I^{r,s}$ when this is the only 
bigrading of $V$ in sight.

\subsection*{Local normal form.} Let $\V$ be an 
admissible variation of graded-polarized mixed Hodge structure over a 
punctured polydisk $\Delta^{*r}$ with unipotent monodromy $T_j=e^{N_j}$ about 
$s_j=0$, with weight filtration $W$.  Let $V$ be any fiber of $\V$ and define 
$\mathfrak g$ to be the Lie subalgebra of $\gl(V)$ consisting of all 
elements which preserve $W$ and act by infinitesimal isometries on
$\Gr^W V$.   
Then, the limit mixed Hodge structure $(F_{\infty},M)$ of 
$\V$ induce a mixed Hodge structure on $\mathfrak g$.  We write $\gg^{r,s}$ the corresponding
decomposition of $\mathfrak g_{\CC}$ by the $I^{p,q}$'s.  There is then a
a distinguished vector space decomposition
\beq
      {\mathfrak g}_{\C} = \mathfrak q\oplus{\mathfrak g}_{\C}^{F_{\infty}},
      \qquad  \mathfrak q = \bigoplus_{r<0,s}\, \gg^{r,s}_{(F_{\infty},M)}. 
      \label{eq:decomp}
\eeq
where ${\mathfrak g}_{\C}^{F_{\infty}}$ is the stabilizer of the limit
Hodge filtration.  Relative to this decomposition, we can then write 
(cf. $(6.11)$ \cite{P2}) the period map
$$  
       F:U^r\to\M
$$
of the pullback $\V$ to the universal cover $\pi:U^r\to\Delta^{*r}$ as
\beq
        F(z_1,\dots,z_r) 
        = e^{\sum_j\, z_j N_j}e^{\Gamma(s_1,\dots,s_r)}.F_{\infty}
        \label{eq:lnf}  
\eeq
where $\Gamma(s)$ is a $\mathfrak q$-valued holomorphic function which 
vanishes at the origin.  This is called the \emph{local normal form 
of the variation $\V$}.   
% See 

% \begin{remark} For variations over $\Delta^{*r}\times\Delta^b$, one has an
% local normal form 
% \beq
%         F(z,w) = e^{\sum_j\, z_j N_j}e^{\Gamma(s,w)}.F_{\infty}
%                                                    \label{eq:lnf-variant}
% \eeq
% However, this can be viewed as a degenerate case of \eqref{eq:lnf} for a
% variation on $\Delta^{*(r+b)}$ such that $N_{r+1},\dots,N_{r+b}=0$.
% \end{remark}

\subsection*{Nilpotent orbits} If $\V\to\Delta^{*r}$ is an admissible 
variation of mixed Hodge structure over the punctured polydisk $\Delta^{*r}$ 
with local normal form \eqref{eq:lnf} then the associated map
$$
      \theta(z_1,\dots,z_r) = e^{\sum_j\, z_j N_j}.F_{\infty}
$$
from $\C^r$ into the \lq\lq compact dual\rq\rq{} of $\mathcal M$ is 
called the nilpotent orbit of $\V$.  (See~\cite{P2} for the notion of compact 
dual.)

In general, given a classifying space $\mathcal M$ with 
\lq\lq compact dual\rq\rq{} $\check{\mathcal M}$, a nilpotent
orbit with values in $\M$ is a holomorphic, horizontal map
\begin{equation}\label{e.Nilp}
      F_{\nilp}(z_1,\dots,z_r) = e^{\sum_j\, z_j N_j}.F:
      \C^r\to\check{\mathcal M}
\end{equation}
such that 
\begin{itemize}
% \item[{(a)}] $F\in\check{\mathcal M}$;
\item[{(a)}] $N_1,\dots,N_r$ are nilpotent, mutually commuting elements of
the Lie algebra $\mathfrak{g}_{\R} := \mathfrak{g}_{\C}\cap \gl(V_{\R})$;
\item[{(b)}] There exists a constant $K>0$ such that
$$
     F_{\nilp}(z_1,\dots,z_r)\in\mathcal M
$$
for all $z\in\C^r$ with $\Im(z_j)>K$.
\end{itemize}
In particular, $F_{\nilp}$ defines a variation of mixed Hodge structure 
$\V_{\nilp}$ on 
the set of points in $\Delta^{*r}$ where $|s_j|<e^{-2\pi K}$.  Accordingly, we 
say that $F_{\nilp}$ is admissible if $\V_{\nilp}$ is admissible.   

Via the formula~\eqref{e.Nilp}, a real nilpotent orbit is completely
determined by the data $(N_1,\ldots, N_r;F,W)$ 
consisting of nilpotent operators $N_1,\ldots, N_r$ on a real vector space $V$ 
together with a an decreasing filtration $F$ of $V_{\mathbb{C}}$ and a
decreasing filtration
$W$ of $V$.

\begin{remark} For the remainder of this paper, we will suppress the
data of the graded-polarization when discussing nilpotent orbits.
\end{remark}

Given an admissible nilpotent orbit $(N_1,\ldots,N_r;F,W)$, define
$$
       N(z) = \sum_j\, z_j N_j.
$$

In~\cite{Kashiwara} Kashiwara proved the following results concerning
relative weight filtrations associated to admissible nilpotent orbits.  
(We refer to \emph{loc.\ cit.} for the notion of a relative weight filtration.)

\begin{theorem}
Suppose $(N_1,\ldots, N_r;F,W)$ is an admissible nilpotent orbit.  Then 
\begin{itemize}
\item[{(a)}] The relative weight filtration $M(N(v),W)$ exists for every
vector $v\in\mathbb\R^r_{\geq 0}$;
\item[{(b)}] For each subset $I\subset\{1,\ldots, r\}$,  let $C(I)$ denote the 
monodromy cone $\sum_{i\in I}\,\RR_{>0}N_i$ in $\mathfrak g_{\RR}$.  Then 
the filtration $M(C(I),W)=M(N,W)$ is constant for $N\in C(I)$;
\item[{(c)}] Let $I=\{1,\dots,r\}$ and $M = M(C(I),W)$.  Then, $(F,M)$ is a
mixed Hodge structure with respect to which each $N_i$ is a 
$(-1,-1)$-morphism.  More generally, if $I$ is a subset of 
$\{1,\dots,r\}$ with complement $I'$ then 
\beq
        \theta_{I} = (\exp(\sum_{j\in I'}\, z_j N_j).F,M(C(I),W))
                                              \label{eq:partial-nilp}
\eeq
is an admissible nilpotent orbit, and each $N_i\in I$ is a $(-1,-1)$
morphism of the mixed Hodge structure on the right hand side
of \eqref{eq:partial-nilp};
\item[{(d)}] If $I$ and $J$ are subsets of $\{1,\dots,r\}$ then 
$$
      M(C(I),M(C(J),W)) = M(C(I\cup J),W)
$$
\end{itemize}
\end{theorem}

\subsection*{Deligne Systems} 
% \label{sub.DeligneSystems}
Let $(N_1,\dots,N_r;F,W)$ define an admissible nilpotent orbit
and let $W^0,\dots,W^r$ be the sequence of increasing filtrations defined
by the requirement that $W^0 = W$ and $W^j = M(N_j,W^{j-1})$.  Then,
by a theorem of Deligne \cite{D, Schwarz, BP1}, the data 
$(N_1,\dots,N_r,Y_{(F,W^r)})$ defines a sequence of mutually commuting
gradings (in the notation of equation $(3.3)$ of \cite{BP1})
\beq
       Y^{r} = Y_{(F,W^r)},\quad  Y^{r-1} = Y(N_r,Y^r),\quad\dots\  .
       \label{eq:Deligne-system}
\eeq
such that $Y^k$ grades $W^k$. Furthermore, if $(F,W^r)$ is split
over $\R$ this construction gives the corresponding gradings
of the $\SL_2$-orbit theorem~\cite{CKS, KNU}.  More precisely, 
let $(\hat F,W^r)$ denote the $\ssl_2$-splitting of $(F,W^r)$, and 
$\{\hat Y^j\}$ be the corresponding system of gradings.   Let 
$\hat H_j = \hat Y^j - \hat Y^{j-1}$ and $\hat N_j$ denote the component of 
$N_j$ with eigenvalue zero with respect to $\ad\hat Y^{j-1}$ for $j=1,\dots,r$.
Then, each pair $(\hat N_j,\hat H_j)$ is an $\ssl_2$-pair which commutes
with $(\hat N_k,\hat H_k)$.

Our main interest in Deligne systems will be in the grading 
\beq
Y^0=Y(N_1, Y(N_2, \ldots, Y_{(F,W^r)}))   \label{eq:Deligne-system-iterate}
\eeq
obtained by applying the construction in \eqref{eq:Deligne-system} 
recursively.

Note that the proof of Deligne's theorem is pure linear algebra.  In
particular, it applies to situations that do not necessarily arise
from Hodge theory.  In the terminology of~\cite[Definition
2]{Schwarz}, a finite dimensional vector space $V$ equipped with a
finite increasing filtration $W^0$ and $r$ commuting nilpotent
operators $N_1,\ldots, N_r$ and an operator $Y^r$ preserving $W^0$ is
called a \emph{Deligne system} if 
\begin{enumerate}
\item $W^{j+1} = M(N_{j+1}, W^j)$ exists for $j\geq 0$;
\item $W^{j+1}|_{W^i_{\ell}}= M(N_{j+1}, W^j|_{W^i_{\ell}})$, for each $j$ and
  $\ell$ and each $i<j$;
\item $N_i\in W^j_{-2}\End(V)$ for $i\leq j$ and $N_i\in W^j_0\End(V)$
for $i\geq j$;
\item $Y^r$ splits $W^r$ and preserves each $W^i$.  Moreover
  $[Y^r,N_i]=-2N_i$ for all $i$.  
\end{enumerate}
Deligne's theorem~\cite[Theorem 2]{Schwarz}, shows that the data
$(N_1\ldots, N_r; W^0)$ of a Deligne system gives rise to a system of
splittings $Y^k$ of the $W^k$ as above.

In particular, if we start with an admissible orbit $(N_1,\ldots, N_r,
F,W)$ and choose vectors $v_1,\ldots, v_d\in\mathbb{R}^r$ such that the 
$\N{v_i}$ all lie in the closure of the monodromy cone $C(\{1,\dots,r\})$ and 
$\sum_{i=1}^r v_i\in\RR_{>0}^r$, then the data 
\beq
      (\N{v_1},\ldots, \N{v_d};F,W) \label{eq:new-nilp}
\eeq 
again defines an admissible nilpotent orbit.  
Thus we obtain a system of gradings as above and, in particular, a grading 
$Y^0=Y(\N{v_1},Y(\N{v_d}, \ldots, Y_{(F,M)}))$ of $W$.

\begin{remark} If $(F,W)$ is a mixed Hodge structure then the data
$(N=0,F,W)$ determines an admissible nilpotent orbit for which the
associated grading \eqref{eq:Deligne-system-iterate} is just $Y_{(F,W)}$.
\end{remark}

\subsection*{Splittings} Let $(F,W)$ be an $\R$-mixed Hodge structure
with underlying vector space $V$, and 
$\gl(V) = \bigoplus_{a,b}\, \gl(V)^{a,b}$ be the bigrading 
\eqref{eq:Deligne-bigrading} for the mixed Hodge structure
induced by $(F,W)$ on the Lie algebra $\gl(V)$.  Define 
\beq
   \Lambda^{-1,-1}_{(F,W)} = \bigoplus_{a,b<0}\, \gl(V)^{a,b}
\eeq
Then, on account of the defining properties $(a)$--$(c)$ of the bigrading
\eqref{eq:Deligne-bigrading} it follows that
$$
          I^{p,q}_{(g.F,W)} = g.I^{p,q}_{(F,W)}
$$
for all $g\in\exp(\Lambda^{-1,-1}_{(F,W)})$.

\begin{theorem}[Deligne, Prop. $(2.20)$ \cite{CKS}] Given an 
$\R$-mixed Hodge structure $(F,W)$ with underlying vector space
$V$, there exists a unique, functorial element 
$$
        \delta\in \gl(V_{\R})\cap\Lambda^{-1,-1}_{(F,W)}
$$ 
such that $(e^{-i\delta}.F,W)$ is split over $\R$.  Every morphism of 
$(F,W)$ commutes with $\delta$; thus the morphisms of $(F,W)$ are
exactly the morphisms of $(e^{-i\delta}.F,W)$ which commute with this
element.
\end{theorem}

\par The proof of Lemma $(6.60)$ in~\cite{CKS} contains the implicit
construction of another functorial splitting operation (cf.~equation
$(3.30)$ in~\cite{CKS})
\beq
          (F,W)\mapsto (e^{-\xi}.F,W)          \label{eq:sl2-splitting-op}
\eeq
on the category of $\R$-mixed Hodge structures which is optimal for the 
study of nilpotent orbits.  More precisely, if for any mixed Hodge
structure $(F,W)$ we define
\beq
          \hat Y_{(F,W)} = Y_{(e^{-\xi}.F,W)}   \label{eq:sl2-grading}
\eeq
then one of the major components of the $\SL_2$-orbit theorem 
of~\cite{KNU} can be stated as follows:

\begin{theorem}\label{theorem:knu-main}\cite{KNU} Let $(N_1,\dots,N_r;F,W)$ 
generate an admissible nilpotent orbit and let $y(m)\in\R^r$ be a sequence of 
positive real numbers such that the ratios 
$y_{j+1}(m)/y_j(m)$ tend to $0$ for $j=1,\dots,r$ upon formally setting
$y_{r+1}(m)=1$. Then,
\beq
       \lim_{m\to\infty}\, \hat Y_{(e^{iN(y(m))}.F,W)}
        = Y(N_1,Y(N_2,\dots,\hat Y_{(F,W^r)})).       \label{eq:knu-main}
\eeq
\end{theorem}

\par In this paper, we call the splitting operation
\eqref{eq:sl2-splitting-op} the {\it
  $\ssl_2$-splitting}. In~\cite{KNU}, \eqref{eq:sl2-splitting-op} is
called the {\it canonical splitting} operation and $\xi$ is denoted as
$\epsilon(F,W)$.  The proof of Lemma $(6.60)$ in~\cite{CKS} gives a
recursive formula for $\xi$ in terms of the Hodge components of
Deligne $\delta$-splitting for $(F,W)$.  In
Corollary~\ref{c.SL2_Splitting}, we prove the following result due
essentially to Deligne.

\begin{theorem}\cite{D}\label{theorem:SL2_Splitting} The $\ssl_2$-splitting is the
unique, functorial splitting of $\R$-MHS which is given by universal
Lie polynomials in the Hodge components of Deligne's $\delta$-splitting
such that if $(e^{zN}.F,W)$ is a admissible nilpotent orbit with limit
mixed Hodge structure $(F,M)$ which is split over $\R$ then the Deligne
grading of the splitting of $(e^{iN}.F,W)$ is a morphism of type $(0,0)$
for $(F,M)$.
\end{theorem}

\begin{remark}
  It follows from \cite[Lemma 3.12]{CKS} that $(e^{iN}.F,W)$ is a
  mixed Hodge structure whenever $(e^{zN}.F,W)$ is an admissible
  nilpotent orbit with limit split over $\R$, because, in that case,
  the graded quotients $\Gr^W$ are $\SL_2$-orbits.
\end{remark}

The notion of Deligne system does not appear in~\cite{KNU}. To
extract  \eqref{eq:knu-main} from \cite{KNU}, we need 
the following lemma from an unpublished letter of Deligne to 
Cattani and Kaplan.   The reader can find a proof in Remark~\ref{r.ds-2}
of Part~\ref{p.2}.

\begin{lemma}\label{lemma:ds-2}\cite{D,BP1} 
Let $(N;F,W)$ be an admissible nilpotent orbit (in one variable) with
limit mixed Hodge structure $(F,M)$ split over $\R$.   Then 
\beq
\hat Y_{(e^{iN}.F,W)} = Y(N,Y_{(F,M)})   \label{eq:filt-preserve}
\eeq
\end{lemma}

In particular 
(cf.~\cite{D, KP, BP}), it follows from equation 
\eqref{eq:filt-preserve} and Theorem \eqref{theorem:SL2_Splitting} 
that $Y(N,Y_{(F,M)})$ is a morphism of type $(0,0)$ for $(F,M)$.  
We will use the following extension of this result.

\begin{lemma}\label{l.FiltPreserveZero} Let $(N_1,\dots,N_r;F,W)$ 
define an admissible nilpotent orbit.  Then,
$$
      Y=Y(N_1,Y(N_2,\dots,Y(N_r,Y_{(F,W^r)})))
$$
preserves the Deligne $I^{p,q}$'s of $(F,W^r)$.  Furthermore, if 
$(\hat F,W^r) = (e^{-\xi}.F,W^r)$ is the $\ssl_2$-splitting of 
$(F,W^r)$ then 
$$
       Y(N_1,Y(N_2,\dots,Y(N_r,Y_{(\hat F,M)})))
       = e^{-\xi}.Y(N_1,Y(N_2,\dots,Y(N_r,Y_{(F,M)})))
$$
\end{lemma}
\begin{proof}
  See Lemma~\ref{l.xiy}
\end{proof}

\begin{remark} Implicit in the second part of Lemma~\eqref{l.FiltPreserveZero}
is the statement that $\xi$ preserves each weight filtration $W^j$.
This is explained in the proof of Lemma~\ref{l.xiy}. 
\end{remark}

\par For future use, we record the following:

\begin{lemma}\label{lemma:morphisms-and-splittings} Let $\a$ be a 
morphism of type $(-1,-1)$ for the mixed Hodge structure $(F,W)$.
Then, $(\widehat{e^{i\alpha}.F},W) = (\hat F,W)$ for both the Deligne 
and $\ssl_2$-splitting operation.
\end{lemma}
\begin{proof} By Proposition $(2.20)$ of~\cite{CKS}, 
$$
        \bar Y_{(F',W)} = e^{-2i\delta_{(F',W)}}.Y_{(F',W)}
$$
for any mixed Hodge structure $(F',W)$.  Using this formula and the
fact that Deligne's splitting commutes with morphisms of mixed Hodge
structure it easily follows that 
$$
        \delta_{(e^{i\alpha}.F,W)} = \delta_{(F,W)} + i\alpha
$$
since $\alpha$ is a $(-1,-1)$-morphism of both $(e^{i\alpha}.F,W)$
and $(F,W)$.  To obtain the analogous assertion for the $\ssl_2$-splitting
one uses the fact that $\epsilon$ is given by universal Lie polynomials
in the Hodge components of $\delta$ and the fact that 
$[\delta^{p,q},\alpha]=0$ since $[\delta,\alpha] = 0$ and $\alpha$ is 
of type $(-1,-1)$.
\end{proof}

\subsection*{Limiting Gradings}
The (standard) vertical strip in $U^r$ is the set of points 
\beq
       I = \{\, z = x + iy \mid x_j\in [0,1],\hph{a} y_j\in
       [1,\infty)\hph{a}\forall\, j\,\}     \label{eq:vertical-strip}
\eeq
For a point $z=x+iy\in U^r$ we define $t_j = y_{j+1}/y_j$, where
we formally set $y_{r+1}=1$.  Let $S_r$ denote the group of
permutations of $\{1,\dots,r\}$ and let $\sigma\in S^r$ act on $\C^r$
by permuting coordinates.  Then,
\beq
       I = \cup_{\sigma\in S_r}\, \sigma(I')  \label{eq:finite-translates}
\eeq
where $I' = \{\, z\in I \mid t_j\in (0,1]\hph{a}\forall j\,\}$ is the set of points where $y_1\geq y_2\geq \cdots \geq y_r \geq y_{r+1}=1$.  

\begin{defn}\label{defn:sl2-sequence} A sequence of points 
$z(m) = x(m) + iy(m)$ in $I'$ is said to be \emph{tame} if 
$\lim_{m\to\infty}\, x_j(m)$ and $\lim_{m\to\infty}\, t_j(m)$
exist for each index $j$.  A tame sequence is said to be 
an \emph{$\ssl_2$-sequence} if there exists 
\begin{itemize}
\item[{(a)}] A linear transformation $T:\R^d\to\R^r$;
\item[{(b)}] A sequence $v(m)\in \R_{>0}^d$;
\item[{(c)}] A convergent sequence $b(m)\in\R^r$;
\end{itemize}
such that
$$
            y(m) = T(v(m)) + b(m)
$$
and $\lim_{m\to\infty}\, v_{j+1}(m)/v_j(m) = 0$ for $j=1,\dots,d$ (with
$v_{d+1}(m)=1$ as usual).  An $\ssl_2$-sequence is said to be 
strict if $d=r$, $b(m)=0$ and $T$ is the identity. 
\end{defn}

% \begin{remark} For brevity, we often shorten $\ssl_2$-convergent sequence 
% to $\ssl_2$-sequence.
% \end{remark}

\par In particular, since \eqref{eq:finite-translates} is a finite
union, we have:

\begin{lemma}\label{lemma:finite-translates} Let $z(m) = x(m) + iy(m)\in I$ 
be a sequence of points.  Then, there exists an element $\sigma\in S_r$
and a subsequence $z(m')$ of $z(m)$ such that $\sigma(z(m'))$ is an
$\ssl_2$-sequence.  
\end{lemma}

\par Given an $\ssl_2$-sequence with associated linear
transformation $T$ as above, let $\{e_1,\ldots, e_d\}$ denote the standard 
basis of $\RR^d$ and define $\thh^i = T(e_i)$.  Then, while neither $d$ nor 
the transformation $T:\RR^d\to \RR^r$ is uniquely determined by 
the $\ssl_2$-sequence $y(m)$, the associated flag defined by
the increasing sequence of subspaces
\begin{equation}
\label{e.Flag}
       \Theta^j = \sum_{i\leq j}\, \RR\theta^i
\end{equation}
depends only on the sequence $y(m)$.   Moreover, since $y(m)\in U^r$, it 
can be arranged that $T$ is injective and 
$\thh^i\in \RR^r_{\geq 0}$, $i=1,\ldots, d$.

\par Accordingly, we henceforth assume 
$\theta^1,\dots,\theta^d\in\mathbb R^r_{\geq 0}$. 
With this assumption in place, it then follows from~\eqref{eq:new-nilp} that
$$
       (N(\theta^1),\dots,N(\theta^d);F,W)            
$$   
is also an admissible nilpotent orbit. 

% The following elementary proposition is implicit in~\cite{CDK}.
% 
% \begin{proposition}
%   Any sequence $z(m)\in U^r$ with bounded real part admits an $\ssl_2$
%   subsequence.
% \end{proposition}
% \begin{proof}
%   We leave the details to the reader but indicate that, by passage to a
%   subsequence, we can assume that $y(n)/|y(n)|$ converges to a vector
%   $\thh^1$ in the sphere $S^{r-1}$.   After this, the assertion can be
%   proved by induction on $r$.  
% \end{proof}

\subsection*{} We are now ready to state our main theorem concerning the 
asymptotic behavior of variations of mixed Hodge structure.  

\begin{theorem}\label{theorem:conj-1} 
  Let $\V$ be an admissible variation of mixed Hodge structure on
  a poly-punctured disk $\Delta^{*r}$ with unipotent monodromy and
  associated local normal form \eqref{eq:lnf}.
  Let $z(m)=x(m)+iy(m)$ be an
  $\ssl_2$-sequence with corresponding flag \eqref{e.Flag}.
  Then,
\beq
       \lim_{m\to\infty}\, e^{-\N{x(m)}}.\hat{Y}_{(F(z(m)),W)} =
               Y(\N{\thh^1},Y(\N{\thh^2},\cdots,Y_{(\hat F_{\infty},W^r)})).
               \label{eq:main-limit}
\eeq
\end{theorem}

\begin{remark} The statement of equation \eqref{eq:main-limit} implicitly
assumes that $y_j\to\infty$ for each $j$.  In the case
where only $y_1,\dots,y_{\ell}$ diverge the limit Hodge filtration
$F_{\infty}$ should be replaced by 
$$
        \lim_{m\to\infty}\, e^{\sum_{j\leq\ell}\, -z_j(m) N_j}.F(z(m))
$$
and the weight filtration $W^r$ should be replaced by $W^{\ell}$. In the
extreme case where $z(m)$ is bounded, $\ell=0$ in the previous
equation and Theorem~\eqref{theorem:conj-1} is just the continuity
of the $\ssl_2$-splitting.
\end{remark}

The proof of Theorem~\eqref{theorem:conj-1} will take up most of this
paper.  However, we would like to bring to the reader's attention the
obvious fact that, as the left side of the equation in the theorem
depends only on the choice of flag \eqref{e.Flag}  the right-hand-side 
must also depend only on the choice of this flag.  It is an elementary
exercise using Corollary~\eqref{corollary:composition-series} to show
that right-hand-side of \eqref{eq:main-limit} depends only on 
the flag $\Theta$
% Indeed it an elementary exercise using the properties of Deligne systems
% to show that the right-hand-side of \eqref{eq:main-limit} depends only on 
% the flag $\Theta$.  

% \par One of the basic ingredients in the proof of 
% Theorem~\eqref{theorem:conj-1} is a relative compactness statement
% (Lemma~\eqref{lemma:rel-compact}) which \emph{a priori} only gives convergence 
% along a subsequence.  To overcome this restriction, given a flag $\Theta$
% as in \eqref{e.Flag} we define a $\Theta$-sequence to be an 
% $\ssl_2$-sequence with associated flag $\Theta$.  
% 
% 
% \begin{lemma}\label{lemma:Theta-limits} Let $X$ be a connected metric
% space, $f:U^r\to X$ be a continuous map and $p\in X$.
% \begin{itemize}
% \item[{(a)}] Suppose that for every $\ssl_2$-sequence $z(m)$ there exists a 
% subsequence $z'(m)$ such that $f(z'(m))\to p$. Then, $f(z(m))\to p$  
% for every $\ssl_2$-sequence;
% \item[{(b)}] Suppose that for every $\Theta$-sequence $z(m)$ there exists 
% a subsequence $z'(m)$ such that $f(z'(m))\to p$.  Then, 
% $f(z(m))\to p$ for every $\Theta$-sequence.
% \end{itemize}
% \end{lemma}

\begin{remark} Theorem~\eqref{theorem:conj-1}  has also been obtained 
independently by Kato, Nakayama and Usui \cite{KNU3} in their study of 
classifying spaces of degenerations of mixed Hodge structure.  In particular, 
as part of their study of log intermediate Jacobians \cite{KNU}, they 
obtain an independent proof of Conjecture \eqref{theorem:main}.
\end{remark}

\subsection*{\bf Finiteness.}
One immediate corollary of Theorem~\eqref{theorem:conj-1} is the
boundedness of the function $z\mapsto \hat Y_{(F(z),W)}$.  

\begin{corollary}
\label{c.Boundedness} Let $F(z)$ be the period map of an admissible
variation of mixed Hodge structure over $\Delta^{*r}$.
Let $I$ denote the standard vertical strip
\eqref{eq:vertical-strip} for $U^r$.  Then the function 
$z\mapsto \hat{Y}_{(F(z),W)}$ is bounded on $I$.
\end{corollary}
\begin{proof} Otherwise, we can find a sequence of points $z(m)\in I$ 
on which $\hat{Y}_{(F(z(m),W)}$ is unbounded. Passing to an 
$\ssl_2$-subsequence, we then obtain a contradiction to 
Theorem~\eqref{theorem:conj-1}. 
\end{proof}

This boundedness gives rise to finiteness which will be important for
integral variations.

\begin{corollary}
\label{corollary:finiteness}
  Let $\V$ be an admissible variation of integral mixed Hodge
  structure over $\Delta^{*r}$ with unipotent monodromy.   Then, with the
  notation as in Theorem~\eqref{c.Boundedness}, the set $\mathcal Y$
  of integral gradings in the image of the map $z\mapsto
  Y_{(F(z),W)}$ as $z$ runs over the vertical strip $I$ is finite.
\end{corollary}
\begin{proof}
  If $Y_{(F(z),W)}$ is an integral grading, then (clearly) it is a
  real
grading.  It follows that $\hat{Y}_{(F(z),W)}=Y_{(F(z),W)}$.  (To see
this note that, in the terminology of~\cite[\S 1.2]{KNU},
$\delta(F,W)=0$ and this implies that $\epsilon(F,W)=0$.)
Therefore, the set of integral gradings of the form  $Y_{(F(z),W)}$ 
as $z$ ranges over $I$ is bounded and discrete.  Thus it is finite.
\end{proof}

\section{Reductions}\label{s.Reduct}

Our next job is to reduce Theorem~\ref{theorem:main} to the case that
$\bar{S}$ is a polydisk and $S$ is a punctured polydisk.  We begin
with some review about germs.

\begin{ntt}
  Let $X$ be a topological space.  Write $\mathcal{S}$ for the
  presheaf (in fact, a sheaf) associating to every open $U\subset X$
  the set of all subsets of $U$.  (If $V\subset U$, the map
  $\mathcal{S}(U)\to \mathcal{S}(V)$ is given by $Z\mapsto Z\cap V$).
  The set of subset germs at a point $x\in X$ is the stalk
  $\mathcal{S}_x$.  If $Z\subset X$ then the germ of $Z$ at $x$ is
  the image of $Z$ in $\mathcal{S}_x$. 
\end{ntt}

\begin{ntt}
If $X$ is a complex analytic space we write $\mathcal{A}$ for the
presheaf (also a sheaf) associating to every open subset $U$ of $X$
the set of all complex analytic subspaces $Z$ of $U$.   The set of
germs of complex analytic subspaces at a point $x\in X$ is the stalk
$\mathcal{A}_x$.    We write $\mathcal{A}^r$ for the subsheaf of
reduced subspaces.   There is an obvious inclusion
$\mathcal{A}^r\to\mathcal{S}$.   We say that a germ
$Z\in\mathcal{S}_x$ is analytic if it is in the image of this inclusion.
\end{ntt}

\begin{ntt}
  Let $X$ be a reduced complex analytic space, let $x\in X$ and let
  $Z\in\mathcal{S}_x$ be a subset germ.  We say that
  $f\in\mathcal{O}_{X,x}$ vanishes on $Z$ if there is an open
  neighborhood $U$ of $x$ such that $f$ is regular on $U$ and vanishes
  on a subset $Z_U$ of $U$ whose germ is $Z$.  We write
  $\mathcal{I}_{Z,x}$ for the set of all $f\in\mathcal{O}_{X,x}$ which
  vanish on $Z$.  Clearly, $\mathcal{I}_{Z,x}$ is an ideal in
  $\mathcal{O}_{X,x}$.  By the Noetherian property of
  $\mathcal{O}_{X,x}$ it is, therefore, finitely generated.  We define
  \emph{the Zariski closure $\ClZar_xZ$ of $Z$ at $x$} to be the
  analytic subspace germ corresponding to $\mathcal{I}_{Z,x}$.  We say
  that $Z$ is \emph{Zariski dense} at $x\in X$ if $\ClZar_x Z$ is the
  germ associated to $X$.  Then $Z$ is
  Zariski dense at $x$ if, any $f\in\mathcal{O}_{X,x}$ regular on a
  neighborhood $U$ of $x$ and vanishing on $U\cap Z$, vanishes identically
on a neighborhood of $x$.   A subset $Z$ of $X$ is Zariski dense in $X$ if it is
Zariski dense at every point $x\in X$. 
\end{ntt}

\begin{ntt}
  Let $\bar{S}$ be a complex analytic space with Zariski dense regular
  locus $\bar{S}_{\reg}$ (for example, any reduced complex analytic
  space).  Let $S\subset\bar{S}_{\reg}$ be a Zariski open subset.  A
  variation of mixed Hodge structure $\V$ on $S$ is \emph{admissible
    relative to $\bar{S}$} if, for any resolution of
  singularities $\pi:\bar{T}\to \bar{S}$ with $T:=\pi^{-1}S$
  biholomorphic to $S$, $\V_T$ is admissible relative to $\bar{T}$.
  Note that if the above property holds for one resolution of
  singularities $\bar{T}\to\bar{S}$ it holds for all.    This defines
  a category $\VMHS(S)^{\ad}_{\bar{S}}$ which is, in fact, equivalent
  to the category $\VMHS(T)^{\ad}_{\bar{T}}$ for any resolution
  $\bar{T}\to \bar{S}$.    If $\H$ is a variation of Hodge
  structure of negative weight on $S$, we define
  $\NF(S,\H)^{\ad}_{\bar{S}}
    =\Ext^1_{\VMHS(S)^{\ad}_{\bar{S}}}(\Z,\H)=\NF(T,\H_T)^{\ad}_{\bar{T}}$
where $\pi:\bar{T}\to T$ is any resolution with $\pi:\pi^{-1}(S)\to S$
an isomorphism.
\end{ntt}

\begin{theorem}
\label{t.Reduction}
Let $r$ be a non-negative integer, then the following are equivalent.
\begin{enumerate}

\item[(a)] Let $S=\Delta^{*r}$, $\bar{S}=\Delta^{r}$; let $\H$ be a polarized
  variation of pure Hodge structure of negative weight, with
  $\H_{\Z}$-torsion-free and with unipotent monodromy on $S$; and
  let $\nu\in\NF(S,\H)^{\ad}_{\bar{S}}$.  Let $\bar{\ZZ}(\nu)$ denote
  the closure of the zero locus $\ZZ(\nu)$ in the analytic topology of
  $\bar{S}$.  Assume that the germ of
  $\ZZ(\nu)$ at $0$ is Zariski dense at $0$.  
Then, the germ of $\bar{\ZZ}(\nu)$ at $0$ coincides with the germ of
$\bar{S}$ at $0$. 

\item[(b)] Let $S,\bar{S}$ and $\H$ be as in (a), but drop the
  assumption that the germ of $\ZZ(\nu)$ at $0$ is Zariski dense.  Then
  the germ of $\bar{\ZZ}(\nu)$ at $0$ is analytic.

\item[(c)] The same statement as in (a) holds without the assumption that
$\H_{\Z}$ is torsion-free.

\item[(d)] The same statement as in (b) holds without the assumption 
that $\H$ has unipotent monodromy.

\item[(e)] Let $a$ and $b$ be non-negative integers with $a+b=r$,  let
  $S=\Delta^{*a}\times\Delta^{b}$ and $\bar{S}=\Delta^{a+b}$.  Let
  $\H$ be a variation of pure Hodge structure of negative weight on $S$ and let
  $\nu\in\NF(S,\H)^{\ad}_{\bar{S}}$.  Then the germ of $\bar{\ZZ}(\nu)$ at $0$ is 
analytic. 

\item[(f)] Theorem~\ref{theorem:main} holds in the case that $S$ has
  dimension $r$ and 
  $\bar{S}\setminus S$ is a normal crossing divisor.

\item[(g)] Let $\bar{S}$ be a complex analytic space of
  dimension $r$ and let $S$ be a Zariski open subset of
  $\bar{S}_{\reg}$.  Let $\H$ be a variation of Hodge structure of
  negative weight on $S$ and let $\nu\in\NF(S,\H)^{\ad}_{\bar{S}}$.
  Then the topological closure $\bar{\ZZ(\nu)}$ is the underlying
  space of a closed complex subspace of $\bar{S}$. 

\end{enumerate}
\end{theorem}

\begin{proof}
We prove the entire theorem by induction on $r$.  The equivalence 
is obvious for $r=0$ (since all of the individual statements hold
unconditionally).  

  (a) $\Rightarrow$ (b): Let $Z$ denote the Zariski closure of
  $\ZZ(\nu)$ at $0$.  Shrinking the polydisk $\bar{S}$ if necessary,
  we can assume that $Z$ is an analytic subspace of $\bar{S}$ and that
  $Z$ contains $\bar{\ZZ}(\nu)$.  We can also assume that $\dim Z<r$.
  Let $\nu_Z$ denote the restriction
  of $\nu$ to the regular locus of $Z$.   By induction (g) applies
  to show that  $\bar{\ZZ}(\nu_Z)$ is a closed complex analytic
  subspace of $Z$.  Since $\ZZ(\nu_Z)=\ZZ(\nu)\cap Z_{reg}$ is Zariski
  dense in $Z$, this implies that $\bar{\ZZ}(\nu)=Z$.

  (b) $\Rightarrow$ (c): Let $\H_{\mathrm{tors}}$ denote the torsion
  part of $\H$ and $\H_{\mathrm{free}}:=\H/\H_{\mathrm{tors}}$
  denote the torsion-free part with $\pi:\H\to\H_{\mathrm{free}}$
  the projection map.  Then, for $\nu\in\NF(S,\H)^{\ad}_{\bar{S}}$ we
  have $\ZZ(\nu)=\ZZ(\pi(\nu))$ (because $\Ext^1_{\MHS}(\Z,H)=0$ for $H$ a 
  torsion mixed Hodge structure).   

  (c) $\Rightarrow$ (d): By Borel's theorem, the monodromy of $\H$ is
  quasi-unipotent.  Therefore we can find a positive integer $d$ such
  that the pull-back of $\H$ to $\Delta^{*r}$ via the map $f:S\to S$
  given by $(z_1,\ldots, z_r)\mapsto (z_1^d,\ldots, z_r^d)$ has
  unipotent monodromy.  By assumption, the germ of
  $\bar{\ZZ}(f^*(\nu))$ at $0$ is analytic.  Since $f$ is proper,
  the proper mapping theorem implies that the germ of $\bar{\ZZ}(\nu)$
  at $0$ coincides with the germ of $f(\bar{\ZZ}(\nu))$ at $0$ and is
  analytic.

  (d)$\Rightarrow$ (e): We induct on $r:=a+b$ starting with $a=b=0$
  where the statement is obvious.  For $i\in\{1,\ldots r\}$ set
  $S_i:=\{z\in S:z_i=0\}$, and set $S_0=\Delta^{*r}$.  For
  $i\in\{0,\ldots, r\}$ let $\nu_i$ denote the restriction of $\nu$ to
  $S_i$.  Then $\ZZ(\nu)=\cup_{i=0}^r \ZZ(\nu_i)$, so
  $\bar{\ZZ}(\nu)=\cup_{i=0}^r \bar{\ZZ}(\nu_i)$.  By hypothesis, the
  germ of $\bar{\ZZ}(\nu_0)$ at $0$ is analytic and, by induction, for
  each $i>0$ the germ of $\bar{\ZZ}(\nu_i)$ at $0$ is analytic.  It
  follows that the germ of $\bar{\ZZ}(\nu)$ at $0$ is analytic.

  (e) $\Rightarrow$ (f): To prove that $\bar{\ZZ}(\nu)$ is analytic,
  it suffices to prove that its germ is analytic at each point $s\in
  \bar{S}$.  Since $\bar{S}\setminus S$ is a normal crossing divisor
  this follows from (e) and from the obvious fact that the germ 
  of $\bar{\ZZ}(\nu)$ is analytic at every point $s\in S$.  

  (f) $\Rightarrow$ (g): Set $C:=\bar{S}\setminus S$.  By
  Hironaka~\cite{Hironaka}, we can find a proper morphism
  $\pi:\bar{T}\to \bar{S}$ where $\bar{T}$ is smooth, $D=\pi^{-1} (C)$ is
  a normal crossing divisor and $\pi:\bar{T}\setminus D\to S$ is
  an isomorphism.  Then, setting $T:=\bar{T}\setminus D$, $\pi$
  induces an isomorphism $\pi^*\NF(S,\H)^{\ad}_{\bar{S}}\cong
  \NF(T,\H)^{\ad}_{\bar{T}}$~\cite[Remark 1.6 (i)]{SaitoANF}.  Let
  $\nu\in\NF(S,\H)^{\ad}_{\bar{S}}$ be a normal function, and let
  $\ZZ_T=\{s\in \bar{T}:\pi^*(\nu)=0\}$.  Suppose $\bar{\ZZ}_T\subset
  \bar{T}$ is complex analytic.  Since $\pi$ is proper, the proper
  mapping theorem shows that $\pi(\bar{\ZZ}_T)$ is a analytic and
  $\bar{\ZZ}=\pi(\bar{\ZZ}_T)$.

  (g) $\Rightarrow$ (a): obvious.
\end{proof}

\begin{lemma}
\label{l.SplitLemmaWeight}
  Let $H$ be a pure Hodge structure of weight $w<0$ and with
  $H_{\Z}$-torsion free.  Let $\nu\in\Ext^1_{\MHS}(\Z,H)$ be
  represented by the short exact sequence
\begin{equation}
\label{e.Pre1}
0\to H\to V\to \Z\to 0
\end{equation}
with $V=(V_{\Z}, F,W)$.  Then 
$\nu=0\Leftrightarrow Y_{(F,W)}\in
w\End(V_{\Z})$.
\end{lemma}

\begin{proof}
 $\Rightarrow$:  If $\nu=0$, we have $V=\Z\oplus H$.  So, every
$v\in V_{\Z}$ can be written as $v=r+h$ with $r\in\Z$ and $h\in H$.  Clearly,
 $Y_{(F,W)}(v)=wh\in wV_{\Z}$.

 $\Leftarrow$: Suppose $Y_{(F,W)}\in w\End(V_{\Z})$.  Then the map
 $\frac{1}{w}Y_{(F,W)}$ is a morphism of mixed Hodge structure from
 $V$ to $H$ inducing a retraction of the sequence~\eqref{e.Pre1}.  
\end{proof}

\begin{corollary}\label{c.CorRed1}
Let $S,\bar{S},\H$ and $\nu$ be as in Theorem~\ref{t.Reduction} (b), and let 
$\nu$ be given by an extension
$$
0\to \H\to \V\to \Z\to 0
$$
of variations of mixed Hodge structures on $S=(\Delta^*)^n$.    Let $(F(z),W)$ be 
the local normal form of $\V$ on $U^n$ with $F(z)=e^{N(z)}e^{\Gamma(s)}.F_{\infty}$.
Then 
$$
\ZZ(\nu)=\{s\in S: s=e^{2\pi i z}, Y_{(F(z),W)}\in w\End V_{\Z} \}.
$$
\end{corollary}

\section{Analyticity of the zero locus}
\label{s.Analyticity}

We now prove Theorem~\eqref{theorem:main} 
assuming Theorem~\eqref{theorem:conj-1} and the results on the Deligne
systems and the $\ssl_2$-splittings stated in
section~\ref{s.Polydisks}.   In fact, we will deduce the theorem as a
corollary of a more general result concerning admissible variations on
punctured polydisks.    

\begin{ntt}  
  Set $S=\Delta^{*r}$, $\bar{S}=\Delta^r$ and $\pi:U^r\to S$ be as in 
the discussion preceding~\eqref{eq:lnf}.   Let
  $\V\in\VMHS(S)^{\ad}_{\bar{S}}$ with $V$ and the
  local normal form of $\V$
$$
F(z)=e^{N(z)}e^{\Gamma(s)}.F_{\infty}
$$
as in~\eqref{eq:lnf}.  To fix the notation, we
remind the reader that $N(z)=\sum z_i N_i$ with $N_i\in\End
V_{\mathbb{Q}}$ and that  $V_{\mathbb{Q}}$ comes equipped with the weight
filtration $W$.  By definition, the limit mixed Hodge structure is 
$(F_{\infty},M)$ where $M=W^r=M(N_1+\ldots +N_r, W)$. 

For $z\in U^r$, set $Y(z)=Y_{(F(z),W)}$.  Let $I$ denote the vertical
strip~\ref{eq:vertical-strip}.   Then, for each integral $Y_{\mathbb{Z}}\in \End
V_{\mathbb{Z}}$, set $B(Y_{\mathbb{Z}}):=\{z\in I: Y(z)=Y_{\mathbb{Z}}\}$ and 
$C(Y)=\pi(B(Y_{\mathbb{Z}}))$.  
\end{ntt}

\newcommand\YZ{Y_{\mathbb{Z}}}
\newcommand\txi{\tilde{\xi}}
\newcommand\Yinf{Y_{\infty}}

\begin{theorem}\label{t.two}
  Suppose $C(\YZ)$ is Zariski dense at the origin in $\bar{S}$
  for some $\YZ\in \End V$.   Then $C(\YZ)=S$.  Moreover,
  $[N_i,\YZ]=0$ for all $i$. 
\end{theorem}

\begin{proof}
  By assumption, $0$ is a limit point of $C(\YZ)$ in the usual
  topology.  Therefore, we can find (possibly after permuting the
  coordinates) an $\ssl_2$-sequence $z(m)\in I$ such that
  $Y(z(m))=\YZ$ for all $m$ (and $z_i(m)$ is unbounded for each $i$).  

Write $z(m)=x(m)+iy(m)$ with $x,y$ real and set $\mu=\lim_{m\to
  \infty} x(m)$.  
Write $\xi$ for the $\ssl_2$-splitting of $(F_{\infty},M)$.   (In the
notation of~\cite{KNU}, $\xi=\epsilon(F_{\infty},M)$.)  Then, by
Theorem~\ref{theorem:conj-1} and Lemma~\ref{l.FiltPreserveZero}, 
\begin{align*}
\YZ&=e^{N(\mu)}.Y(N(\theta^1),Y(N(\theta^2),\ldots, Y_{(\hat
  F_{\infty},M)}))\\
&= e^{N(\mu)}e^{-\xi}.Y(N(\theta^1),Y(N(\theta^2),\ldots, Y_{(F_{\infty},M)}))
\end{align*}
To simplify the notation, we write $\Yinf=
Y(N(\theta^1),Y(N(\theta^2),\ldots, Y_{(F_{\infty},M)}))$ and
$\txi=\xi-N(\mu)$. 
Then, since $\xi$ commutes with the $N_i$ we have 
$$
\YZ=e^{-\txi}.\Yinf.
$$

Now, for any operator $A$ on $V$, write $A^{p,q}$ for the component of
$A$ in $\gl^{p,q}(V)_{(F_{\infty},M)}$.  By Lemma~\ref{l.FiltPreserveZero},
$\Yinf=\Yinf^{0,0}$.  In other words, $\Yinf$ preserves the
$I^{p,q}_{(F_{\infty},M)}$.  Likewise, $\ad\Yinf$ preserves the subalgebra
$\mathfrak q$.

\begin{lemma}\label{l.simple}
Suppose $z\in B(\YZ)$.  Then
\begin{equation}\label{e.simple}
e^{\Gamma(s)}.Y_{\infty}=e^{-N(z)}.\YZ.
\end{equation}
\end{lemma}

\begin{proof}[Proof of Lemma~\ref{l.simple}]   Recall that by
equation~\eqref{eq:decomp} we have
$\mathfrak g_{\C} = \mathfrak g_{\C}^{F_{\infty}}\oplus\mathfrak q$ where 
$\mathfrak g_{\mathbb C}^{F_{\infty}}$ stabilizes the limit Hodge filtration.  
Therefore, since 
$Y(z)$ and $e^{N(z)}e^{\Gamma(s)}.Y_{\infty}$ both preserve $F(z)$
it follows that 
$
       Y(z) - e^{N(z)}e^{\Gamma(s)}.Y_{\infty}
$
preserves $F(z)$ and hence
$$
       f(z) = e^{-\Gamma(s)}e^{-N(z)}.Y(z) - Y_{\infty}
$$
takes values in $\mathfrak g_{\C}^{F_{\infty}}$.    On the other hand,
since $z\in B(Y_{\Z})$ and $Y_{\Z}=e^{-\tilde\xi}.Y_{\infty}$ we also have
$$
      f(z) = e^{-\Gamma(s)}e^{-N(z)}e^{-\tilde\xi}.Y_{\infty} - Y_{\infty}
$$
In particular, since $\Gamma(s)$, $N(z)$ and $\tilde\xi$ are elements
of $\mathfrak q$ and $\ad Y_{\infty}$ preserves $\mathfrak q$ it follows
that $f(z)$ takes values in 
$$
         \mathfrak g_{\C}^{F_{\infty}}\cap\mathfrak q = 0 
$$
\end{proof}

\newcommand\Linf{L_{\infty}}

\begin{lemma}\label{l.infz}
  The two linear maps $L_{\infty}:\mathbb{C}^r\to \End V_{\mathbb{C}}$
  given by $z\mapsto [N(z),\Yinf]$  
  and $L_{\mathbb{Z}}:\mathbb{C}^r\to \End V_{\mathbb{C}}$ given by
  $z\mapsto [N(z),\YZ]$ have the same kernels.  
\end{lemma}

\begin{proof}[Proof of Lemma~\ref{l.infz}]  This follows directly from
  the fact that $\txi$ commutes with $N(z)$.  
\end{proof}

Since $\YZ$ is integral, we can find a subset
$\Omega\subset\{1,\ldots, r\}$ such that the $[N_j,\YZ]$ form a basis
of $L_{\mathbb Z}(\mathbb{Q}^r)$ as $j$ runs through $\Omega$.   Thus there
exists rational numbers $\beta_{ij}$
($i\in\{1,.., r\}, j\in\Omega$) such that  
$$
[N_i,\YZ]=\sum_{j\in\Omega} \beta_{ij}[N_j,\YZ].
$$ 
Clearly $\beta_{jj}=1$ for $j\in\Omega$. Likewise, applying 
$\Ad(e^{\tilde\xi})$ to the previous equation, it follows that
$$
     [N_i,Y_{\infty}]=\sum_{j\in\Omega} \beta_{ij}[N_j,Y_{\infty}].
$$
So, since $\tilde{\xi}$ commutes with $N_1,\ldots, N_r$, setting 
$L_i=[N_i,Y_{\infty}]$, we have  
$$
   L_{\infty}(z)=\sum_{j\in\Omega}\sum_{i\geq j} \beta_{ij}z_i L_j.
$$

Multiplying equation~\eqref{e.simple} by $e^{\txi}$ and using the fact that 
$\txi$ commutes with $N(z)$, we see that 
\begin{equation}\label{e.four}
z\in B(Y)\Rightarrow
e^{\txi}e^{\Gamma(s)}.Y_{\infty}=e^{-N(z)}.Y_{\infty}.
\end{equation}

Since $\Yinf=\Yinf^{0,0}$ with respect to $(F_{\infty},M)$ while 
each $N_i$ is a morphism of type $(-1,-1)$, if follows that 
$$
          (e^{N(z)}.\Yinf)^{-1,-1}=[N(z),\Yinf].
$$  
Set $\gamma(s)=(e^{\txi}e^{\Gamma(s)}.\Yinf)^{-1,-1}$.  This is a
holomorphic function on $\bar{S}$ which by\eqref{e.four} must be
equal to 
$$
(e^{-N(z)}.\Yinf)^{-1,-1}=\Linf(-z)=-\sum_{j\in\Omega}\sum_{i=1}^r
\beta_{ij}z_i L_j
$$
for $z\in B(\YZ)$. 
Since $C(\YZ)$ is Zariski dense at the origin and $\gamma(s)$ lies in
$\Linf(\CC^r)$ for $s\in C(\YZ)$, $\gamma(s)$ must lie in
$\Linf(\CC^r)$ for all $s\in\bar{S}$. Since the $L_j, j\in\Omega$
form a basis of $\Linf(\CC^r)$, we can write
$\gamma(s)=\sum_{j\in\Omega} \gamma_j(s)L_j$ for $s\in\bar{S}$.  

Therefore
$$
z\in B(\YZ)\Rightarrow \gamma_j(s)=\sum_{i=1}^r \beta_{ij} z_i,
\forall j\in\Omega.
$$
We can find a positive integer $b$ such that $b\beta_{ij}\in\mathbb{Z}$ for all $i,j$.  So we have
$$
z\in 
B(\YZ)\Rightarrow  b\gamma_j(s)=\sum_{i=1}^r b\beta_{ij} z_i, \forall j\in\Omega.
$$
Exponentiating both sides we find that $z\in B(\YZ)\Rightarrow$
\begin{equation}
\exp(2\pi ib\gamma_j(s))=\prod_{i=1}^r
s_i^{b\beta_{ij}}, \forall j\in\Omega.\label{e.2.8}
\end{equation} By our assumption that $C(\YZ)$ is Zariski dense at $0$, \eqref{e.2.8} must
hold identically identically on $\bar{S}$.  The
left-hand-side of the equation is non-vanishing and holomorphic on a
neighborhood of the origin, this forces $\beta_{ij}=0$ for all $i,j$.
Thus, since $\beta_{jj}=1$ for $j\in\Omega$,  we have $\Omega=\emptyset$ and
$[N_i,Y_{\infty}]=0$ for all $i$.  Since $N(z)$ commutes with
$Y_{\infty}$ and with $\txi$, it commutes with $\YZ$.  Thus we have
$$
z\in B(\YZ)\Leftrightarrow e^{\Gamma(s)}.Y_{\infty}=\YZ.
$$
As the above equation is a holomorphic equation for $C(\YZ)$, it must hold identically.  
Thus we have $C(\YZ)=S$.  
Moreover, since $\Gamma(0)=0$, we have $Y_{\infty}=\YZ$.   

This completes the proof of Theorem~\ref{t.two}.
\end{proof}

\begin{corollary}\label{c.ZD}  Suppose $\nu\in\NF(S,\H)^{\ad}_{\bar S}$
and suppose that $Z(\nu)$ is Zariski dense at the origin in $S$ with
$\H$ and $S$ as in  Theorem~\ref{t.Reduction}.  Then $\ZZ(\nu)=S$.  
\end{corollary}

\begin{proof}
We write 
$$
0\to \H\to \V\to \Z\to 0
$$ for the extension corresponding to $\nu$.  Since, by
Corollary~\ref{corollary:finiteness}, the density of $\ZZ(\nu)$ at the
origin implies that there is a $\YZ$ such that $C(\YZ)$ is also
Zariski dense.  Then use Theorem~\ref{t.two}.
\end{proof}

\begin{proof}[Proof of Theorem~\ref{theorem:main}]
Corollary~\ref{c.ZD} establishes (a) of Theorem~\ref{t.Reduction}.  Therefore
(g) holds as well.    This directly implies Theorem~\ref{theorem:main}.
\end{proof}

\section{Deligne systems I}
\par In the remainder of this paper we will work exclusively with
admissible variations of $\mathbb R$-mixed Hodge structure.   

\subsection*{} We now reduce the proof of Lemma~\ref{l.FiltPreserveZero} 
to a  corollary of the following sequence of lemmata:

\begin{lemma}\label{lemma:ds-1}\cite{D} If $(N; F,W)$ defines an admissible
nilpotent orbit with limit mixed Hodge structure split over $\R$ then, 
in the notation of \eqref{eq:Deligne-system},
the $\ssl_2$-splitting of the mixed Hodge structure
$(e^{iN}. F,W)$ is $(e^{i\hat N}.F,W)$.

\end{lemma}

\begin{proof}
We have $Y_{(e^{i\hat N}.F,W)}=Y(N,Y_{(F,M)})$ by the second
line in the proof of Theorem 2 of the appendix to~\cite{KP}.  
(In~\cite{KP}, the notation $N_0$ is used to denote the
$0$-eigencomponent of $N$ under the
operator $Y(N,Y_{(F,M)})$, which is denoted by  $\hat N$ in this paper.)
It follows from Lemma~\ref{lemma:ds-2} that $Y_{(e^{i\hat N}.F,W)}=\hat
Y_{(e^{iN}.F,M)}$.  Therefore the $\SL_2$-splitting of $(e^{iN}.F,W)$
is equal to $(e^{i\hat N}.F,W)$.  
\end{proof}

Suppose now that $(N_1,\dots,N_r;\hat F_r,W)$ defines an admissible 
nilpotent orbit with limit mixed Hodge structure split over $\R$.  Following 
the notation of \eqref{eq:Deligne-system}, let $W^0,\dots,W^r$ be the associated
system of weight filtrations.  Recall that, by \cite{CKS} and 
\cite{Kashiwara}, 
$$
       (z_1,\dots,z_{r-1})\mapsto (e^{\sum_{j\leq r-1}\, z_j N_j}e^{iN_r}.\hat F_r,W^0)
$$
is an admissible nilpotent orbit, and hence $(e^{iN_r}.\hat F_r,W^{r-1})$
is a mixed Hodge structure.  Accordingly,
$$
        (z_1,\dots,z_{r-1}) \mapsto 
                              (e^{\sum_{k\leq r-1}\,z_k N_k}.\hat F_{r-1},W^0)
$$
is an admissible nilpotent orbit with limit mixed Hodge structure split over 
$\R$, where
$
      (\hat F_{r-1},W^{r-1}) = (e^{i\hat N_r}.\hat F_r,W^{r-1})
$
is the $\ssl_2$-splitting of $(e^{iN_r}.\hat F_r,W^{r-1})$. Iterating this 
construction, we obtain a sequence of mixed Hodge structures
\beq
          (\hat F_{j-1},W^{j-1}) = (e^{i\hat N_j}.\hat F_j,W^{j-1})
          \label{eq:iter-1}
\eeq
and associated nilpotent orbits
\beq
          (z_1,\dots,z_j) \mapsto e^{\sum_{k\leq j}\, z_k N_k}.\hat F_j,
          \label{eq:iter-2}    
\eeq
where $\hat N_j$ is defined as in the paragraph after~\eqref{eq:Deligne-system}.

% \begin{remark} Editing Note: The next paragraph fits very well with section
%   2.9, and can be eliminated with minor editing.
% \end{remark}

\par In particular, given the data $(N_1,\dots,N_r;\hat F_r,W)$ of an
admissible nilpotent orbit with limit mixed Hodge structure split
over $\R$, the sequence of gradings $\hat Y^j$ constructed in 
\eqref{eq:Deligne-system} is given by
$
             \hat Y^j = Y_{(\hat F_j,W^j)}
$.
Since $N_1,\dots,N_j$ are $(-1,-1)$-morphisms of 
$(\hat F_j,W^j)$, it follows that 
\beq
             [N_k,\hat H_j] = 0             \label{eq:ds-comm-1}
\eeq
for $j>k$ where as in \eqref{eq:Deligne-system}, 
$\hat H_j = \hat Y^j - \hat Y^{j-1}$.

\begin{lemma} Let $(N_1,\dots,N_r;\hat F_r,W)$ define an admissible
nilpotent orbit with limiting mixed Hodge structure $(\hat F,M)$ split 
over $\R$.  Then,
$$
     \hat Y^0 = Y(N_1,Y(N_2,\dots,Y(N_r,Y_{(\hat F,M)})))
$$
preserves $\hat F$.
\end{lemma}
\begin{proof} To begin, we recall  that
$
    (\hat N_1,\hat H_1),\dots,(\hat N_r,\hat H_r)
$
form a commuting system of $\ssl_2$-representations~\cite{D,Schwarz}.  
Consequently, 
\beq
      [\hat Y^j,\hat N_k] = 0                     \label{eq:commuting}
\eeq
for $j<k$.  Indeed, this is true by definition for $j=k-1$.
Suppose that $j\leq k-2$.  Then,
\begin{eqnarray*}
      [\hat Y^j,\hat N_k] 
      &=& -[(\hat Y^{j+1}- \hat Y^j) + \cdots + (\hat Y^{k-1}-\hat Y^{k-2}),\hat N_k] \\
      &=& -[\hat H^{j+1} + \cdots + \hat H^{k-1},\hat N_k] = 0.
\end{eqnarray*}

\par By the prior paragraphs, 
$
       \theta(z) = (e^{zN_1}.\hat F_1,W)
$
is an admissible nilpotent orbit with limit mixed Hodge structure
split over $\R$, and hence by Lemma \eqref{lemma:ds-2}, 
$$
       \hat Y^0(\hat F_1^p)\subseteq \hat F_1^p.
$$
Using the identity $\hat F_r = e^{\sum_{j>1} i\hat N_j}.\hat F$ and the fact that
$\hat Y^0$ commutes with all $\hat N_j$, it then follows from the 
previous equation that $\hat Y^0$ preserves $\hat F$.
\end{proof}                 

\par To pass from admissible nilpotent orbits with limit mixed Hodge structure
split over $\R$ to the general case, we now use the following sequence of 
lemmata.

\begin{lemma}\label{lemma:weight-filt-preserve} Let $(N_1,\dots,N_r;F,W)$ 
generate an admissible nilpotent orbit with $\ssl_2$-splitting 
$(\hat F,W^r) = (e^{-\xi}.\hat F,W^r)$.  Then, $\xi$ preserves each associated 
weight filtration $W^j$.
\end{lemma}
\begin{proof} The $\ssl_2$-splitting is functorial and each $W^j$ is a 
filtration of the limit mixed Hodge structure $(\hat F,W^r)$ by subobjects.  
It follows easily that $\xi$ preserves each $W^j$.
\end{proof}

We now prove the result stated in Lemma~\ref{l.FiltPreserveZero}.

\begin{lemma}\label{l.xiy} Let $(N_1,\dots,N_r;F,W)$ define an admissible 
nilpotent orbit with $\ssl_2$-splitting $(\hat F,W^r) = (e^{-\xi}.F,W^r)$.  
Then, 
$$
       Y(N_1,Y(N_2,\dots,Y(N_r,Y_{(\hat F,W^r)})))
       = e^{-\xi}.Y(N_1,Y(N_2,\dots,Y(N_r,Y_{(F,W^r})))
$$
\end{lemma} 
\begin{proof}  The operator $\xi$ commutes with $N_1,\dots,N_r$ since 
$\xi$ is a universal Lie polynomial in the Hodge components of
Deligne's $\delta$-splitting $(e^{-i\delta}.F,W^r)$ of $(F,W^r)$
and $\delta$ commutes with all $(-1,-1)$-morphisms of $(F,W^r)$,
and hence in particular with $N_1,\dots,N_r$.  Furthermore,
since
$$
       Y_{(e^{-\xi}.F,W^r)} = e^{-\xi}.Y_{(F,W^r)}
$$
and $\xi$ preserves $W^{r-1}$ and commutes with $N_r$, we have
(by the properties of Deligne's construction~\cite{KP})
$$
        Y(N_r,Y_{(\hat F,W^r)}) = e^{-\xi}.Y(N_r,Y_{(F,W^r)}).
$$
Iterating this process, we obtain,
$$
     Y(N_1,Y(N_2,\dots,Y(N_r,Y_{(\hat F,W^r})))
      = e^{-\xi}.Y(N_1,Y(N_2,\dots,Y(N_r,Y_{(F,W^r)}))).
$$
\end{proof}

\begin{corollary} Let $(N_1,\dots,N_r;F,W)$ define an admissible nilpotent 
orbit.  Then,
$$
      Y=Y(N_1,Y(N_2,\dots,Y(N_r,Y_{(F,W^r)})))
$$
preserves the Deligne $I^{p,q}$'s of $(F,W^r)$.
\end{corollary}
\begin{proof} Let $\hat Y$ denote the analog of $Y$ obtained by 
replacing $(F,W^r)$ by the $\ssl_2$-splitting $(\hat F,W^r)$.  Then, 
$\hat Y$ is real and preserves both $\hat F$ and $W^r$.  Therefore,
$\hat Y$ preserves
$$
         I^{p,q}_{(\hat F,W^r)} = \hat F^p\cap\overline{\hat F^q}\cap W^r_{p+q}
$$
By the previous lemma, it then follows that $Y$ preserves $I^{p,q}_{(F,W^r)}$
since $F=e^{\xi}.\hat F$.
\end{proof}

\par For future reference, we now record the following three results:

\begin{lemma}\label{lemma:associated-filtration} Let $(N_1,\dots,N_r;F,W)$
be an admissible nilpotent orbit with limit mixed Hodge structure split
over $\mathbb R$.  Let $\sigma:\mathbb R^{r-i}\to\mathbb R^r$ denote the 
embedding $y\mapsto (0,y)$, and $\tilde z(m)$ be an $\ssl_2$-sequence in 
$\mathbb R^{r-i}$, with associated vectors $\theta^j$ as in \eqref{e.Flag}.
Then,
\beq
        (N_1,\dots,N_i,N(\sigma(\theta^1)),\dots,N(\sigma(\theta^d));F,W)
                                     \label{eq:iter-3}
\eeq
is an admissible nilpotent orbit with the same limit MHS as the original
nilpotent orbit.  In particular, if $\hat F_o,\hat F_1,\dots,$ are the 
associated sequence \eqref{eq:iter-1} of filtrations associated the nilpotent 
orbit \eqref{eq:iter-3} then
$$
      Y(N(\sigma(\theta^1)),\dots,Y(N(\sigma(\theta^d)),Y^r)) 
      = Y_{(\hat F_i,W^i)}
$$
\end{lemma}
\begin{proof} This is a minor variation on \eqref{eq:new-nilp} and the
previous remarks. 
\end{proof}

\begin{lemma}\label{corollary:composition-series} Let $(N',N'';F,W)$ 
generate an admissible nilpotent orbit with associated weight filtrations
$W' = M(N',W)$ and $W'' = M(N'',W')$.  Let $Y'' = Y_{(F,W'')}$.
Then, the pairs $(N'+N'';F,W')$ and $(N'';F,W')$ generate admissible 
nilpotent orbits which give the same associated gradings 
$Y(N'+N'',Y'')$ and $Y(N'',Y'')$ of $W'$.
\end{lemma}
\begin{proof} Since $(e^{z' N' + z'' N''}.F,W)$ is a nilpotent orbit it
follows that $N'$ is a $(-1,-1)$-morphism of $(e^{z'' N''}.F,W')$
whenever the later is a MHS, and hence both $(N'',F,W')$ and 
$(N' + N'',F,W')$ generate admissible nilpotent orbits with the same
limit mixed Hodge structure.  Moreover,
by Lemma~\eqref{l.FiltPreserveZero} we can without loss of generality
assume that $(F,W'')$ is split over $\mathbb R$.  As such, by 
Lemma~\eqref{lemma:ds-2} and Lemma~\eqref{lemma:morphisms-and-splittings}
$$   
      Y(N' + N'',Y'') = \hat Y_{(e^{iN' + iN''}.F,W')}
                       = \hat Y_{(e^{iN''}.F,W')} = Y(N'',Y'')
$$
since $N'$ is a $(-1,-1)$-morphism of $(e^{iN''},F,W')$.
\end{proof}

\begin{lemma}\label{lemma:commutator-lemma} Let $W$ be an increasing
filtration of a finite dimensional vector space $V$ over a field of 
characteristic zero.  Let $Y$ be a grading of $W$ and $N$ be a 
nilpotent endomorphism of $V$ such that $[Y,N] = -2N$.  Let 
$\beta\in W_{-1}(\gl(V))$ and suppose that 
$[e^{\beta}.Y,N] = -2N$. Then, $\beta\in\ker(\ad N)$.
\end{lemma}
\begin{proof} Observe that under the above hypothesis, $[e^{\beta}.Y - Y,N] = 0$.
Let $\beta = \sum_{j<0}\, \beta_{-j}$ with respect to the eigenvalues of $\ad Y$.
If $\beta\neq 0$ then there is a smallest integer $k>0$ such that 
$\beta_{-k}\neq 0$, and hence
$$
         e^{\beta}.Y - Y = k\beta_{-k} \mod W_{-k-1}(\gl(V)).
$$
Applying $\ad N$ to both sides it then follows from the fact that 
$\ad N$ lowers the eigenspaces of $\ad Y$ by 2 that $\beta_{-k} = 0$.
\end{proof}

\section{Deligne Systems II} In~\cite{KNU}, Kato, Nakayama and Usui 
attach to any admissible nilpotent orbit with data $(N_1,\dots,N_r;F,W)$ 
an associated semisimple endomorphism $t(y)$.  For later use, we now 
derive a formula for $t(y)$ in terms of the gradings $\hat Y^j$ constructed 
above.  To this end, let us assume for the moment that $(N_1,\dots,N_r;F,W)$ 
underlies a nilpotent orbit of pure Hodge structure of weight $k$.  Let 
$(\hat F_r,W^r)$ denote the $\ssl_2$-splitting of $(F,W^r)$, and recall
that $W^r$ in this case is the monodromy weight filtration 
$W(N)[-k]$ for any element $N$ in the cone of positive linear
combinations of $N_1,\dots,N_r$.  In particular, since any such
$N$ is a $(-1,-1)$-morphism of $(\hat F_r,W^r)$ it follows that
the pair $(N,\hat Y_{(r)})$ where 
$$
           \hat Y_{(r)} = \hat Y^r - k
$$
defines an $\ssl_2$-pair.  As above, we can iteratively define
$\hat Y_{(j)} = \hat Y^j - k$ using the nilpotent
orbit $(N_1,\dots,N_j;\hat F_j)$.  Define,
$$
      \tilde t(y) = \prod_{j=1}^r\, t_j^{\half\hat Y_{(j)}}
                  = (\prod_{j=1}^r\, t_j^{-\half k})
                     (\prod_{j=1}^r\, t_j^{\half \hat Y^j})
$$
where $t_j = y_{j+1}/y_j$, and hence $t_1 \dots t_r = y_{r+1}/y_1 = 1/y_1$.
Accordingly,
$$
      \tilde t(y) = y_1^{(\half k)}\prod_{j=1}^r\, t_j^{\half \hat Y^j}.
$$

\par By Theorem $(0.5)$ of \cite{KNU}, the mixed version of $t(y)$ is to be
constructed as follows:  If $(N_1,\dots,N_r;F,W)$ defines an admissible
nilpotent orbit then 
$$
       \hat Y_{(e^{\sum\, i y_j N_j}.F,W)} \to \hat Y^0
$$
provided that $t_j\to 0$ for all $j$.   Let $t_k(y)$ denote the semisimple 
endomorphism $\tilde t(y)$ attached by the previous paragraph to the induced 
nilpotent orbit of pure Hodge structure of weight $k$ on $Gr^W_k$.  Then, 
$t(y)$ is constructed by multiplying each $t_k(y)$ by $y_1^{-\half k}$ and 
then lifting the resulting semisimple element to the ambient vector space 
via the grading $\hat Y^0$.  Accordingly, since the gradings 
$\hat Y^0,\dots,\hat Y^r$ are mutually commuting, it follows that
\beq
             t(y) = \prod_{j=1}^r\, t_j^{\half \hat Y^j}
                                                      \label{eq:knu-t}
\eeq

\begin{remark} For a nilpotent orbit of pure Hodge structure, the elements
$\hat Y_{(j)}$ are infinitesimal isometries of the polarization. Consequently,
although $t(y)$ is not an element of $G_{\C}$ since it is the twist of an
automorphism of the graded-polarizations by $y_1^{-\half\hat Y^0}$, the action
of $\Ad(t^{-1}(y))$ preserves $G_{\C}$.
\end{remark}

\par The following result appears in Proposition $(10.4)$ of \cite{KNU} 
with slightly different notation:

\begin{lemma}\label{lemma:twist} Let $(N_1,\dots,N_r;F,W)$ define
an admissible nilpotent orbit.  Then,
$$
          \Ad(t^{-1}(y))e^{\sum_j\, iy_j N_j} = e^{P}
$$
where $P$ is a polynomial in non-negative half integral powers of 
$t_1,\dots,t_r$ with constant term 
$
       iN_1 + i\sum_{j>1}\, \hat N_j
$.
\end{lemma}
\begin{proof} By \eqref{eq:knu-t}, 
$$
        \Ad(t^{-1}(y)) y_k N_k
        = (\prod_{j\leq k-1}\, t_j^{-\half\hat Y^j})
          (\prod_{j\geq k} t_j^{-\half\hat Y^j})
          y_k N_k
$$
where $N_k$ is $(-1,-1)$-morphism of $(\hat F_j,W^j)$ for $j=k,\dots,r$,
and hence $[N_k,\hat Y^j] = -2N_k$. Consequently,
$$
          (\prod_{j\geq k} t_j^{-\half\hat Y^j}) y_k N_k
          = t_k\dots t_r y_k N_k = N_k
$$
On the other hand, $N_k$ preserves $W^j$ for $j<k$.  Therefore,
$$
         (\prod_{j<k}\, t_j^{-\half\hat Y^j})N_k
$$
is a polynomial in non-negative, half-integral powers of $t_j$ for
$j<k$.  Taking the limit as $t_1,\dots,t_r\to 0$ it then follows
that the constant term of $P$ is 
$
          i\sum_k\, N_k^{\sharp}
$
where $N_k^{\sharp}$ is the projection of $N_k$ to 
$\cap_{0<j<k}\, \ker(\ad\hat Y^j)$ with respect to the mutually
commuting gradings $\hat Y^j$.  Accordingly, $N_1^{\sharp} = N_1$,
whereas for $k>1$, we can first project onto $\ker(\ad N_{k-1})$
to obtain $\hat N_k$.  By \eqref{eq:commuting}, $\hat N_k$
commutes with $\hat Y^j$ for $j<k$, and hence $N_k^{\sharp} = \hat N_k$.
\end{proof}

\begin{remark} For nilpotent orbits of pure Hodge structure, this
statement appears in Lemma $(4.5)$ of \cite{Luminy};
note however that in \cite{Luminy}, $t_j$ is defined to be $y_j/y_{j+1}$
which is reciprocal to our convention.
\end{remark}

\section{Relative Compactness}

\par Let $I$ and $I'$ be subsets of $U^r$ as in
\eqref{eq:vertical-strip} and \eqref{eq:finite-translates}.  Let
$F:U\to\mathcal M$ be the period map of an admissible variation of
mixed Hodge structure over $\Delta^{*r}$ with local normal form $F(z)
= e^{N(z)}e^{\Gamma(s)}.F_{\infty}$ as in \eqref{eq:lnf}.  Let $t(y)$
be the associated family of semisimple endomorphisms \eqref{eq:knu-t}.
In this section, we will prove the following result, which is due to
Cattani and Kaplan in the pure case~\cite[Theorem 4.7]{Luminy}.

\begin{lemma}\label{lemma:rel-compact} The image of the set $I'$ under
the map
$$
          \tilde F(z_1,\dots,z_r) 
           =  t^{-1}(y)e^{-\sum_j\, x_j N_j}.F(z_1,\dots,z_r)
$$
is a relatively compact subset of the classifying space $\mathcal M$.
\end{lemma}

\begin{remark} In Theorem (4.7)~\cite{Luminy}, Cattani and Kaplan define 
$t_j = y_j/y_{j+1}$, which is reciprocal to our convention.
\end{remark}

\par For each index $j=1,\dots,r$ let 
$$
        \Gamma_j(s) = \Gamma(0,\dots,0,s_{j+1},\dots,s_r)
$$
Then, for each $j$ we have an associated partial period map
\beq
       F_j(z_1,\dots,z_r) = e^{\sum_j\, z_j N_j}e^{\Gamma_j(s)}.F_{\infty}
                                       \label{eq:partial-period-map}
\eeq
which takes values in $\mathcal M$ for $\Im(z)$ sufficiently large.
Indeed, \eqref{eq:partial-period-map} is the nilpotent orbit obtained
from $F(z)$ by degenerating $z_1,\dots,z_j$. 

\begin{remark} As in Proposition $(2.6)$ of \cite{Luminy}, it follows via 
equation $(6.10)$ of \cite{P2} that 
\beq
             \Gamma_j\in\ker(\ad N_1)\cap\cdots\cap\ker(\ad N_j)
                      \label{eq:gamma-adN}
\eeq
\end{remark}

\par To compactify future notation, we define $F_0(z) = F(z)$ and 
set
\beq
        \tilde F_j(z_1,\dots,z_r) 
        = t^{-1}(y)e^{-\sum_j\, x_j N_j}.F_j(z_1,\dots,z_r)
        \label{eq:twisted-partial-period-map}
\eeq
for $j=0,\dots,r$.
  
\begin{defn} Let $z(m)\in U^r$ be an $\ssl_2$-sequence, and suppose that 
$t_j(m)\to 0$.  Then, we say that $z(m)$ has non-polynomial 
growth with respect to $y_j$ (or $z_j$) if there exists a subsequence $z(m')$ 
of $z(m)$ such that 
\beq
           \lim_{m'\to\infty}\,\frac{y_{j+1}^d(m')}{y_j(m')} = 0
           \label{eq:non-polynomial-growth}
\eeq
for every $d>0$.  In particular, unless a sequence of points $z(m)\in I'$ is 
bounded, there exists a smallest index $\iota$ such that $z(m)$ has 
non-polynomial growth with respect to $y_{\iota}$ (since we formally
define $y_{r+1}(m)=1$).  If $z(m)\in I'$ is 
bounded, we define $\iota=0$.
\end{defn}

\par To employ the notion of non-polynomial growth in aid of the proof
of Theorem~\eqref{theorem:conj-1} we recall the following elementary
observation about convergent sequences:

\begin{lemma}\label{lemma:sequences} Let $\Sigma$ be a topological space.
Then, a sequence $\sigma_m$ in $\Sigma$ converges to $\sigma$ if 
and only if for every subsequence $\sigma'_m$ of $\sigma_m$ there 
exists a subsequence $\sigma''_m$ of $\sigma'_m$ such that 
$\sigma''_m\to\sigma$.
\end{lemma}

Given an $\ssl_2$-sequence $z(m)$ and a grading 
$Y_{\lim}$ of $W$, in order to show
that 
$$
        \hat Y_{(F(z(m)),W)} \to Y_{\lim},
$$
it is sufficient to show that, for every subsequence $z'(m)$ of $z(m)$,
we can find a subsequence $z''(m)$ such that 
$$
        \hat Y_{(F(z''(m)),W)}\to Y_{\lim}
$$
In particular, since each $z'(m)$ is an $\ssl_2$-sequence, it has a
corresponding smallest index $\iota$ with respect to which it has
non-polynomial growth, and hence we can pass to a subsequence
$z''(m)$ of $z'(m)$ for which equation \eqref{eq:non-polynomial-growth}
holds for $y_{\iota}$.

\par As such, it is sufficient to prove Theorem~\eqref{theorem:conj-1}
for $\ssl_2$-sequences $z(m)$ satisfying~\eqref{eq:non-polynomial-growth} 
since the right hand side of \eqref{eq:main-limit} only depends on the 
original sequence $z(m)$ and the associated nilpotent orbit.  Moreover, 
we may pass to a subsequence of $z(m)$ as necessary.

\begin{theorem}\label{theorem:first-reduction}  Let 
$z(m) = x(m) + iy(m)\in I'$ be an $\ssl_2$-sequence.  Let $\iota$ be 
the smallest index with respect to which $z(m)$ has non-polynomial growth with 
respect to $y_{\iota}$.  Assume that \eqref{eq:non-polynomial-growth} holds
on $z(m)$.  Then,
\beq
     \lim_{m\to\infty}\, 
      \hat Y_{(F(z(m)),W)} - \hat Y_{(F_{\iota}(z(m)),W)}  \to 0 
      \label{eq:first-reduction}
\eeq
upon passage to a suitable subsequence.
\end{theorem}
\begin{proof} Assume Lemma~\eqref{lemma:rel-compact}.  
% Note that by  Lemma~\eqref{lemma:Theta-limits} it is sufficient 
% to establish \eqref{eq:first-reduction} along a subsequence.  
If $\iota=0$, then \eqref{eq:first-reduction} is a tautology because
$F_{\iota}=F$.

\par Assume therefore that $\iota>0$ and observe that both $F(z)$ and 
$F_{\iota}(z)$ arise from period maps with the same nilpotent orbit, 
and hence the same  associated family \eqref{eq:knu-t} of semisimple 
endomorphisms $t(y)$.  Therefore,
\begin{eqnarray*}
        e^{-N(x)}.F(z) 
        &=& e^{i N(y)}e^{\Gamma(s)}.F_{\infty} \\
        &=& e^{i N(y)}e^{\Gamma(s)}e^{-\Gamma_{\iota}(s)}
            e^{\Gamma_{\iota}(s)}.F_{\infty}             \\
        &=& \Ad(e^{i N(y)})(e^{\Gamma(s)}e^{-\Gamma_{\iota}(s)})
            e^{i N(y)}e^{\Gamma_{\iota}(s)}.F_{\infty} \\
        &=& \Ad(e^{i N(y)})(e^{\Gamma(s)}e^{-\Gamma_{\iota}(s)})
            e^{-N(x)}.F_{\iota}(z_1,\dots,z_r)
\end{eqnarray*}
Consequently, 
\begin{eqnarray*}
         \tilde F(z) &=& t^{-1}(y) e^{-N(x)}.F(z)             \\ 
         &=& t^{-1}(y) 
            \Ad(e^{i N(y)})(e^{\Gamma(s)}e^{-\Gamma_{\iota}(s)}) 
             t(y)t^{-1}(y)e^{-N(x)}.F_{\iota}(z_1,\dots,z_r)   \\
         &=& \Ad(t^{-1}(y))(\Ad(e^{iN(y)})
                  (e^{\Gamma(s)}e^{-\Gamma_{\iota}(s)}))
                  .\tilde F_{\iota}(z)                        \\  
         &=& e^{B(z)}.\tilde F_{\iota}(z)
\end{eqnarray*}
where
\begin{eqnarray*}
         e^{B(z)} &=& \Ad(t^{-1}(y))(\Ad(e^{iN(y)})
                  (e^{\Gamma(s)}e^{-\Gamma_{\iota}(s)}))                \\ 
         &=& \Ad(\Ad(t^{-1}(y))e^{iN(y)})
              \Ad(t^{-1}(y))(e^{\Gamma(s)}e^{-\Gamma_{\iota}(s)})
\end{eqnarray*}
By Lemma \eqref{lemma:twist}, we know that
\beq
        \Ad(t^{-1}(y))(e^{\sum_j\, i y_j N_j}) = e^{P(t)} 
        \label{eq:poly-twist}
\eeq
where $P(t)$ is a polynomial in non-negative half-integral powers of 
$t_1,\dots,t_r$ (with constant term).  Accordingly,
$$
     e^{B(z)} = \Ad(e^{P(t)})\Ad(t^{-1}(y))(e^{\Gamma(s)}e^{-\Gamma_{\iota}(s)}).
$$

\par To analyze the asymptotic behavior of $e^{B(z)}$, define
$\tilde\Gamma(s)\in\mathfrak{q}$ to be the unique nilpotent operator satisfying
$$
     e^{\tilde\Gamma(s)} = e^{\Gamma(s)}e^{-\Gamma_{\iota}(s)}.   
$$
Then, since $\tilde\Gamma(0) = 0$ if $s_1,\dots,s_{\iota}= 0$,
it follows that there exist $\mathfrak q$-valued holomorphic
functions $f_1,\dots,f_{\iota}$ on $\Delta^r$ such that
$$
      \tilde\Gamma(s) = \sum_{j=1}^{\iota}\, s_j f_j.
$$
Moreover, the identity $|s_j| = e^{-2\pi y_j}$
coupled with the order structure 
\beq
    y_1\geq y_2 \geq\cdots\geq y_r\geq 1        \label{eq:order-structure}
\eeq
on $I'$ implies that $|s_j|\leq |s_{\iota}|$ for $j=1,\dots,\iota$.  
Shrinking $\Delta^r$ if necessary, we can then find a constant $K$ such that
$$
      |\tilde\Gamma(s)| < K|s_{\iota}|.
$$

Consider now a sequence $z(m)\in I'$ such that $\iota$ is the smallest
index such that $z(m)$ has non-polynomial growth with respect to 
$z_{\iota}(m)$.  Then, by construction, the quantities 
$y_1(m),\dots,y_{\iota}(m)$ must satisfy some set of mutually polynomial 
bounds, otherwise we contradict the definition  of $\iota$.
Therefore, since $t(y)$ acts semi-simply by multiplication by monomials in 
half-integral powers of $t_1,\dots,t_r$ on its eigenspaces, it follows that 
(after increasing $K$ if necessary) 
\beq
            |\Ad(t^{-1}(y))\tilde\Gamma(s(m))| 
            < K y^d_{\iota}(m)|s_{\iota}(m)|   \label{eq:first-bound}
\eeq
for some half-integer $d$. 

\par Combining the above remarks, it then follows that  
\beq
     || e^{B(z(m))} -1 ||
     < K y^d_{\iota}(m)|s_{\iota}(m)|          \label{eq:second-bound}
\eeq
for a suitable constant $K$.  
   
\par To continue, observe that the operator norm of 
$\Ad(e^{\sum_j\, x_j N_j})$ is bounded on $I'$ since $x\in [0,1]^r$ is 
compact.  Accordingly, (for any fixed norm)
$$
     || \hat Y_{(F(z),W)} - \hat Y_{(F_{\iota}(z),W)} ||  
     \leq K' || \hat Y_{(e^{-N(x)}.F(z),W)} 
                 - \hat Y_{(e^{-N(x)}.F_{\iota}(z),W)} ||
$$
for some suitable constant $K'$.  Accordingly, \eqref{eq:first-reduction}
is equivalent to 
$$
      \lim_{m\to\infty}\, \hat Y_{(e^{-N(x(m))}.F(z(m)),W)} 
                 - \hat Y_{(e^{-N(x(m))}.F_{\iota}(z(m)),W)} \to 0
$$
In particular, since $t(y)$ is a real automorphism which preserves $W$, 
\begin{eqnarray*}
       \hat Y_{(e^{-N(x(m))}.F(z(m)),W)} 
        &-& \hat Y_{(e^{-N(x(m))}.F_{\iota}(z(m)),W)}   \\
        &=& 
          t(y(m)).(\hat Y_{(\tilde F(z(m)),W)} 
           - \hat Y_{(\tilde F_{\iota}(z(m)),W)})  \\
        &=& 
          t(y(m)).(\hat Y_{(e^{B(z(m))}.\tilde F_{\iota}(z(m)),W)} - 
          \hat Y_{(\tilde F_{\iota}(z(m)),W)}) 
\end{eqnarray*}
Moreover, by Lemma \eqref{lemma:rel-compact}, after passage to a subsequence, 
we can assume that $\tilde F_{\iota}(z(m))$ converges to some point in 
$\mathcal M$.  By the real-analyticity of the map
$(F,W)\mapsto \hat Y_{(F,W)}$ it then follows that 
\beq
       \hat Y_{(e^{B(z(m))}.\tilde F_{\iota}(z(m)),W)}
       = e^{C(z(m))}.\hat Y_{(\tilde F_{\iota}(z(m)),W)}
                                               \label{eq:third-bound}
\eeq
where $e^{C(z(m))}-1$ satisfies a bound of the same form \eqref{eq:second-bound}
as $e^{B(z(m))}-1$.  Therefore,
\begin{multline*}
      \hat Y_{(e^{-N(x(m))}.F(z(m)),W)} 
                 - \hat Y_{(e^{-N(x(m))}.F_{\iota}(z(m)),W)}   \\
      =  t(y(m)).((\Ad(e^{C(z(m))})-1)
           \hat Y_{(\tilde F_{\iota}(z(m)),W)})\to 0
\end{multline*}
on account of the fact that $t(y)$ acts by half-integral powers of 
$t_1,\dots,t_r$ and $y^{\ell}_j(m)(e^{C(z(m))}-1)\to 0$ for $j=1,\dots,r$ 
and every half-integer $\ell$.
\end{proof}      

\par To continue, we now prove Theorem~\eqref{theorem:conj-1} in the
case where $F(z)$ is a nilpotent orbit:

\begin{lemma}\label{lemma:main-limit} 
Theorem~\eqref{theorem:conj-1} is true for admissible nilpotent orbits.
\end{lemma} 
\begin{proof} Let $z(m)=x(m)+iy(m)\in I'$ be an $\ssl_2$-sequence.  Then,
$$
          y(m) = T(v(m)) + b(m)
$$
as in definition \eqref{defn:sl2-sequence}.  Accordingly,
$$
     e^{-N(x(m))}.\hat Y_{(e^{N(z(m))}.F_{\infty},W)}
     = \hat Y_{(e^{\sum_j\, iv_j(m) N(\theta^j)}e^{iN(b(m))}.F_{\infty},W)}
$$
with $\theta^1,\dots,\theta^d$ as described in \eqref{e.Flag}.  

\par Now, for any fixed element $b\in\R^r$, the data
$(N(\theta^1),\dots,N(\theta^d),e^{iN(b)}.F_{\infty},W)$ defines an
admissible nilpotent orbit with limit mixed Hodge structure 
$(e^{iN(b)}.F_{\infty},W)$.
By Lemma~\eqref{lemma:morphisms-and-splittings},
$$
       (\widehat{e^{iN(b)}.F_{\infty}},W^r) = (\hat F_{\infty},W^r)
$$
and hence by \eqref{eq:knu-main} it follows that 
$\hat Y_{(e^{\sum_j\, iv_j(m) N(\theta^j)}e^{iN(b)}.F_{\infty},W)}$
converges to the grading
$$
      \tilde Y 
       = Y(N(\theta^1),Y(N(\theta^2),\dots,Y_{(\hat F_{\infty},W^r)}))
$$
independent of $b$.  By Theorem 10.8 of \cite{KNU} it then follows that
for variable $b$ confined to the interior of a compact set that there
is a constant $c$ such that if $\tau_j=v_{j+1}/v_j<c$ then
$$
      \hat Y_{(e^{\sum_j\, iv_j N(\theta^j)}e^{iN(b)}.F_{\infty},W)}
       = \exp(u(\tau;b)).\tilde Y
$$
where $u(\tau,b)$ is a real-analytic function of $\tau=(\tau_1,\dots,\tau_r)$ 
and $b$ with $u(0,b) = 0$.  Accordingly,
$$
       \hat Y_{(e^{\sum_j\, iv_j(m) N(\theta^j)}e^{iN(b(m))}.F_{\infty},W)}
       \to\tilde Y 
$$
\end{proof}

\begin{corollary}\label{corollary:extreme-iota} Let $z(m)$ be an 
$\ssl_2$-sequence for which equation \eqref{eq:non-polynomial-growth} holds 
for either $\iota=0$ or $\iota=r$.  Then, equation \eqref{eq:main-limit}
holds along a subsequence of $z(m)$.
\end{corollary}
\begin{proof} For $\iota=0$ the sequence is bounded, and the statement
follows from the continuity of the $\ssl_2$-splitting.  For
$\iota=r$ we first use Theorem \eqref{theorem:first-reduction} to reduce
the computation of the limit \eqref{eq:main-limit} to the corresponding 
nilpotent orbit and then use the previous Lemma.
\end{proof}

\begin{corollary}\label{corollary:induction-base}
Theorem~\eqref{theorem:conj-1} is true for variations over $\Delta^*$.
\end{corollary}
\begin{proof} This follows from the previous Corollary since over
$\Delta^*$ any $\ssl_2$-sequence $z(m)$ must have either $\iota=0$
or $\iota=1$.
\end{proof}

\par It remains to verify Lemma \eqref{lemma:rel-compact} for
variations of mixed Hodge structure.  For this, we will modify 
Corollary $(12.8)$ of \cite{KNU} which asserts that if $z(m)$
is a strict $\ssl_2$-sequence then the limit
\beq
         F_{\flat} = \lim_{m\to\infty}\, t^{-1}(y(m)).F(z(m)) 
                   = \lim_{m\to\infty}\, t^{-1}(y(m)).F_r(z(m))
         \label{eq:knu-limit}
\eeq
exists, and belongs to the classifying space $\mathcal M$. 
We begin with the following result:

\subsection*{} A sequence of points $s(m)\in\Delta^{*r}$ is a strict 
$\ssl_2$-sequence if $s(m) = \pi(z(m))$ for some strict $\ssl_2$-sequence 
$z(m)\in U^r$ where $\pi:U^r\to\Delta^{*r}$ is the covering map defined by 
$s_j = e^{2\pi iz_j}$ for $j=1,\dots,r$.

\begin{lemma}\label{lemma:vanishing} If $f:\Delta^r\to\C$ be a holomorphic
function which vanishes along every strict $\ssl_2$-sequence 
$s(m)\in\Delta^{*r}$ then $f\equiv 0$.
\end{lemma}
\begin{proof} Strict $\ssl_2$-sequences $z(m)=x(m) + iy(m)\in U^r$ are 
equivalent to pairs of sequences $(x(m),t(m))$ such that $x(m)$ is a 
convergent sequence in $[0,1]^r$ and $t(m)$ is a sequence in $(0,1]^r$ 
which converges to zero.  Indeed, given a strict $\ssl_2$-sequence 
$z(m) = x(m) + iy(m)$ we define 
$t(m) = (t_1(m),\dots,t_r(m))$ via the usual rule $t_j(m) = y_{j+1}(m)/y_j(m)$. 
Conversely, given $(x(m),t(m))$ as above, the sequence $z(m) = x(m) + i y(m)$ 
obtained by setting $y_j(m) = (t_j(m)\cdots t_r(m))^{-1}$ is a strict 
$\ssl_2$-sequence.  In particular, since the differential of the 
map 
$$
      t=(t_1,\dots,t_r)\in (0,1]^r \mapsto (y_1,\dots,y_r)\in\R_{>0}^r
$$
defined by $y_j = (t_j\dots t_r)^{-1}$ is an isomorphism at each point
$t\in (0,1]^r$, it follows that given any point $s_o\in\Delta^{*r}$
on a strict $\ssl_2$-sequence $s(m)$, there exists strict $\ssl_2$-sequences
passing through every point on a neighborhood of $s_o$.  In particular,
if $f$ vanishes on every strict $\ssl_2$-sequence then $f\equiv 0$.
\end{proof}

\par To continue, we now prove Lemma~\eqref{lemma:rel-compact} in
the case where $F(z)$ is an admissible nilpotent orbit generated
by $(N_1,\dots,N_r;F,W)$.
\begin{proof}[Proof of Lemma~\eqref{lemma:rel-compact} for admissible nilpotent
orbits] Let $(\hat F,W^r) = (e^{-\xi}.F,W^r)$
denote the $\ssl_2$-splitting of $(F,W^r)$.  Let $t(y)$ be the
associated family of semisimple endomorphisms.   Then,
\beq
     t^{-1}(y)e^{\sum_j\, i y_j N_j}.F
     = e^{P(t)}\Ad(t^{-1}(y))(e^{\xi}).\hat F         \label{eq:nilp-limit-1}
\eeq
because $t^{-1}(y)$ fixes $\hat F$.

\par Given $A\in\End(V)$, let 
\beq
           A = \sum_{b\in\Z^r}\, A^b                  \label{eq:eigendecomp}
\eeq
where $A^b$ is the eigencomponent of $A$ on which 
$\Ad(t^{-1}(y))A^b = t_1^{\half b_1}\cdots t_r^{\half b_r} A^b$.
Then, by Lemma \eqref{lemma:weight-filt-preserve}
\beq
       \xi = \sum_{b\in \Z^r_{\geq 0}}\, \xi^b         \label{eq:nilp-limit-4}
\eeq
i.e. $\xi^b = 0$ unless $b$ is a vector with non-negative coordinates.
Accordingly, 
$$
      e^{\xi(t)} = \Ad(t^{-1}(y))\xi
$$
is a polynomial in non-negative, half-integral powers of $t_1,\dots,t_r$.
% With these conventions, it then follows from the existence of the limit
% \eqref{eq:nilp-limit-3} that each eigencomponent $\xi_b$ appearing in
% \eqref{eq:nilp-limit-4} must have $b=(b_1,\dots,b_r)$ with all $b_j\geq 0$
% [because $\xi\in\mathfrak q$ and $\mathfrak q$ is a complement to the isotopy algebra 
% of $\hat F$].  
Therefore, the image of any sequence $z(m)$ in $I'$ under the map
\beq
        z\mapsto t^{-1}(y)e^{iN(y)}.F = e^{P(t)}e^{\xi(t)}.\hat F
                                                       \label{eq:twisted-nilp}
\eeq
has a convergent subsequence in the compact dual $\check\M$.  Now, a point 
in $\check\M$ belongs to $\M$ if and only if it induces polarized Hodge
structures on $Gr^W$.  By Lemma \eqref{lemma:rel-compact} for variations
of pure Hodge structure, the image of $I'$ in $Gr^W$ via the map 
\eqref{eq:twisted-nilp} is a relatively compact subset of the sum of the 
corresponding classifying spaces of pure Hodge structure.  Therefore, the 
image of $I'$ under the map \eqref{eq:twisted-nilp} is a relatively compact 
subset of $\M$.
\end{proof}  

\par Let $b=(b_1,\dots,b_r)\in\Z^r$.  Define a partial order on 
the group $\Z^r$ by declaring that $b\geq 0$ if $b_j\geq 0$ for
all $j$, and  $b<0$ otherwise.  If $b<0$ then $w(b) = \min\{\, j \mid b_j<0\,\}$.

\begin{lemma}\label{lemma:gamma-eigenvalues} Let 
\beq
           \Gamma(s) = \sum_{b\in\Z^r}\, \Gamma^b(s) \label{eq:gamma-decomp}
\eeq
be the decomposition of $\Gamma$ with respect to the action of 
$\Ad(t^{-1}(y))$ as above.  If $\Gamma^b\neq 0$ and $b<0$ then
$\Gamma^b_{w(b)} = 0$.
\end{lemma}
\begin{proof}  The proof of Lemma~\eqref{lemma:twist} shows that
$$
      \Ad(t^{-1}(y))(e^{N(x)}) = e^{Q(x,t)}
$$
where $Q(x,t)$ is a polynomial in half-integral powers of $t_1,\dots,t_r$
with constant term 0 and coefficients which are polynomials in $x_1,\dots,x_r$.
By the above conventions (with $P(t)$ and $\xi(t)$ constructed using the
nilpotent orbit attached to $F(z)$),
$$
       t^{-1}(y).F(z) = e^{Q(x,t)}e^{P(t)}
                        \Ad(t^{-1}(y))(e^{\Gamma(s)})
                        e^{\xi(t)}.\hat F_{\infty}
$$
Consequently, since $Q(x,t)$, $P(t)$, $\Gamma(s)$ and $\xi(t)$ all 
take values in $\mathfrak q$ (see \eqref{eq:decomp}) and converge
along strict $\ssl_2$-sequences, 
it follows that equation \eqref{eq:knu-limit} holds along 
strict $\ssl_2$-sequences if and only if 
$$
         \Ad(t^{-1}(y))(\Gamma^b(s))\to 0
$$ 
along strict $\ssl_2$-sequences for each index $b$.  

% Indeed, $t^{-1}(y).F_r(z) = t^{-1}(y)e^{N(z)}.F_{\infty}$ is 
% obtained from the previous equation by setting $\Gamma=0$.  Consequently,
% $$
%        F_{\flat} = e^{P(0)}e^{\xi(0)}.\hat F_{\infty}
% $$

\par Suppose now that $b<0$ and $\Gamma^b(s)\neq 0$.  Let $w=w(b)$.
Then,
\beq
       t_1^{\half b_1}\dots t_r^{\half b_r}(\Gamma^b(s) - \Gamma^b_w(s))\to 0
                                     \label{eq:weight-b-convergence}
\eeq
along any strict $\ssl_2$-sequence since 
\beq
       \Gamma^b(s) - \Gamma^b_w(s)  = \sum_{k=1}^w\, s_k g_k
                                     \label{eq:weight-b}
\eeq
where $g_1,\dots,g_w$ are $\mathfrak q$-valued holomorphic functions
and $b_1,\dots,b_{w-1}\geq 0$.

\par Therefore, $\Ad(t^{-1}(y))(\Gamma^b(s))\to 0$ along strict 
$\ssl_2$-sequences if and only if 
$$
        t_1^{\half b_1}\dots t_r^{\half b_r}\Gamma^b_w(s)\to 0         
$$
along strict $\ssl_2$-sequences.  In particular, via a choice of basis,
$\Gamma^b_w(s)$ is represented by a matrix of holomorphic functions
in the variables $s_{w+1},\dots,s_r$.  Let $f$ denote a typical matrix
entry of $\Gamma^b_w(s)$.  Then, the previous equation holds if and only
if
\beq
        t_1^{\half b_1}\dots t_r^{\half b_r} f(s)\to 0    
                                      \label{eq:weight-b-convergence-2}     
\eeq
along strict $\ssl_2$-sequences.  If $f\neq 0$ then by 
Lemma~\eqref{lemma:vanishing} it follows that there exists a strict
$\ssl_2$-sequence $s(m)=\pi(z(m))$ on which $f$ is non-vanishing.
Furthermore, since $f$ depends only on $s_{w+1}(m),\dots,s_r(m)$, 
we have the freedom to select $z_1(m),\dots,z_w(m)$ in such a way that 
$t_1^{\half b_1}\dots t_r^{\half b_r}$ becomes unbounded 
(since $b_w<0$).  But this contradicts \eqref{eq:weight-b-convergence-2}
since $f$ does not vanish along $s(m)$.
\end{proof}

\begin{corollary}\label{corollary:gamma-eigenvalues} For any weight 
$b=(b_1,\dots,b_r)\in\Z^r$, the function $\Ad(t^{-1}(y))\Gamma^b(s)$ converges 
along any $\ssl_2$-sequence $z(m)\in I'$.
\end{corollary}
\begin{proof} Suppose that $z(m)$ is bounded.  Then, 
$\Ad(t^{-1}(y))\Gamma^b(s)$ converges since both $\Ad(t^{-1}(y))$ and 
$\Gamma^b(s)$ converge.  Likewise, if $b\geq 0$ then 
$\Ad(t^{-1}(y))\Gamma^b(s)$ converges along every $\ssl_2$-sequence since
both $\Gamma^b(s)$ and $t_1^{\half b_1}\cdots t_r^{\half b_r}$ converge. 

\par Assume therefore that $z(m)$ is unbounded and $b<0$.  Let $w=w(b)$.
Then, by the previous Lemma, $\Gamma^b_w(s) = 0$ and hence by equation
\eqref{eq:weight-b}
\beq
       \Gamma^b(s) = \Gamma^b(s) - \Gamma^b_w(s) = \sum_{k=1}^w\, s_k g_k
                                               \label{eq:weight-b-2}
\eeq
where $g_1,\dots,g_w$ are $\mathfrak q$-valued holomorphic functions.
By the order structure~\eqref{eq:order-structure} on $I'$ it follows
that $y_{\ell}\geq y_w$ for $\ell = 1,\dots,w$ and hence
\beq
      s_{\ell}t_w^{\half b_w}\cdots t_r^{\half b_r}\to 0 
                                                \label{eq:weight-b-3}
\eeq
along any $\ssl_2$-sequence since $|s_{\ell}| = e^{-2\pi y_{\ell}}$.
Combining equations~\eqref{eq:weight-b-2} and \eqref{eq:weight-b-3}, we
obtain the convergence of $\Ad(t^{-1}(y))\Gamma^b(s)$ along 
$\ssl_2$-sequences since the remaining factor 
$t_1^{\half b_1}\cdots t_{w-1}^{\half b_{w-1}}$ converges since
$b_1,\dots, b_{w-1}\geq 0$.
\end{proof}

\begin{corollary}\label{corollary:twisted-limit} Let $z(m)$ be an
$\ssl_2$-sequence.  Then, there exists $F_{\sharp}\in\mathcal{M}$ such that 
$$
         \tilde F(z(m))\to F_{\sharp}.
$$ 
Furthermore, the value of $F_{\sharp}$ depends only on the limiting
values of the sequences $t_j(m)$.
\end{corollary}
\begin{proof}  By the above remarks, 
$$
      \tilde F(z(m)) = e^{P(t)}(\Ad(t^{-1}(y))e^{\Gamma(s)})e^{\xi(t)}.
                      \hat F_{\infty}
$$
Moreover, along any $\ssl_2$-sequence, $P(t)$, $\Ad(t^{-1}(y))\Gamma(s)$
and $\xi(t)$ converge to limiting values in $\mathfrak g_{\C}$ which only 
depend on 
%$\Theta$ and 
the limiting values of $t_j(m)$.  This forces 
$\tilde F(z(m))$ to converge to a point $F_{\sharp}$ of the 
\lq\lq compact dual\rq\rq{} $\check{\mathcal M}$ of $\mathcal M$.  In 
particular, since elements of $\mathfrak g_{\C}$ preserve $W$, it follows
from Lemma~\eqref{lemma:rel-compact} for variations of pure Hodge structure
that $F_{\sharp}$ is an element of $\M$.
\end{proof}

\begin{proof}[End of Proof of~\eqref{lemma:rel-compact}]
To complete the proof of Lemma~\eqref{lemma:rel-compact} for variations
of mixed Hodge structure, observe that every sequence $z(m)\in I'$ contains
an $\ssl_2$-sequence $z'(m)$.  Applying the previous corollary to
$F(z'(m))$ we see that the image of $I'$ by $\tilde F$ is a relatively
compact subset of $\mathcal M$.
\end{proof}

\section{Polarized Mixed Hodge Structures}
\label{s.PolarizedMixed}

\par In this section we prove two technical results about deformations
of admissible nilpotent orbits using the theory of polarized mixed
Hodge structure outlined in~\cite{Luminy}.

\begin{lemma}\label{lemma:pmhs-1} Let $(N_1,\dots,N_k;F,W)$ generate an
admissible nilpotent orbit.  Then, there exists a neighborhood $\mathcal A$ of
$0\in\mathfrak g_{\C}\cap\ker(\ad N_1)\cap\cdots\cap\ker(\ad N_k)$
such that for all $\a\in\mathcal A$, $(N_1,\dots,N_k;e^{\a}.F,W)$ generates
an admissible nilpotent orbit.
\end{lemma}
\begin{proof} The map
$$
         (z_1,\dots,z_k) \mapsto e^{\sum_j\, z_j N_j}e^{\a}.F
$$
is horizontal since $N_j(F^p)\subseteq F^{p-1}$.  Therefore, 
$N_j(e^{\a}.F^p)\subseteq e^{\a}.F^{p-1}$.  Likewise, since 
$(N_1,\dots,N_k;F,W)$ is admissible, all of the required
relative weight filtrations exist.  It remains therefore
to show that there exists a constant $c$ such that 
$$
    (e^{\sum_j\, z_j N_j}e^{\alpha}.F,W)
$$
is a graded-polarized mixed Hodge structure for $\Im(z_1),\dots,\Im(z_k)>c$.
\par Since $N_1,\dots,N_k$ and $\alpha$ preserve $W$, we can assume without
loss of generality that $W$ is pure of weight $\ell$.  Via a Tate-twist,
we can assume that $(e^{\sum_j\, z_j N_j}e^{\alpha}.F,W)$ is pure and
effective of weight $\ell$.  By Theorem $(2.3)$ of~\cite{Luminy}, it is 
therefore sufficient to show that $e^{zN}e^{\alpha}.F$ is a nilpotent orbit of 
pure Hodge structure of weight $\ell$, where $N =\sum_j\, N_j$.  

\par To continue, we note that since $\alpha$ commutes with $N_1,\dots,N_k$
it follows that $\alpha$ preserves the monodromy weight filtration $W(N)$.
Since we already know that $e^{zN}.F$ is a nilpotent orbit of
weight $\ell$ (polarized by some bilinear form $Q$), it then follows from 
the theory of polarized mixed Hodge structures (see:~\cite{Luminy}) that it 
is sufficient to show that $e^{\alpha}.F$ induces a pure Hodge structure of 
weight $j+\ell$ on the primitive part (with respect to $N$) of 
$Gr^{W(N)[-\ell]}_{j+\ell}$ which is polarized by $Q_{\ell}(*,*) =
Q(*,N^{\ell}*)$.  But, this is an open condition on $\alpha\in\mathcal A$ which
is true for $\alpha=0$ since $e^{zN}.F$ is a nilpotent orbit.
\end{proof}

\par Given a nilpotent orbit generated by $(N_1,\dots,N_r;F,W)$ let
$t(y)$ be the associated semisimple operator \eqref{eq:knu-t}.  Then,
the corresponding semisimple operator 
$$
        t_{\iota}(y) = \prod_{j>\iota}\, t_j^{\half\hat Y_j}
$$ 
attached to the nilpotent orbit generated by
$(N_{\iota+1},\dots,N_r;F,W^{\iota})$ is obtained from $t(y)$ by setting 
$t_1,\dots,t_{\iota}=1$.

\begin{lemma}\label{lemma:comm-with-t} If $k\leq \ell$ and 
$\a\in\ker(\ad N_k)$ then each eigencomponent of $\a$ with
respect to $\ad\hat Y^{\ell}$ belongs to $\ker(\ad N_k)$.
\end{lemma}
\begin{proof} By the Jacobi identity,
$$
     [N_k,[\hat Y^{\ell},\a]]
     = [[N_k,\hat Y^{\ell}],\a] + [\hat Y^{\ell},[N_k,\a]]         
     = [2N_k,\a] = 0
$$
since $[N_k,\hat Y^{\ell}] = 2N_k$.  Consequently, each
eigencomponent of $\a$ must also belong to $\ker(\ad N_k)$
since $\ad N_k$ decreases eigenvalues with respect to
$\ad\hat Y^{\ell}$ by 2.
\end{proof}

\begin{corollary}\label{corollary:comm-with-t} If $\a$ commutes with 
$N_1,\dots,N_{\iota}$ then so does $\Ad(t_{\iota}^{-1}(y))\a$.
\end{corollary}
\begin{proof} Decompose $\a$ with respect to 
$\hat Y^{\iota+1},\dots,\hat Y^r$ and apply the previous lemma.
\end{proof}

\begin{lemma}\label{lemma:another-twist} For $k\leq\iota$, 
$
    \Ad(t^{-1}_{\iota}(y)) y_k N_k = y_k/y_{\iota+1} N_k
$
and hence 
$$\Ad(t^{-1}_{\iota}(y))e^{i\sum_{k\leq\iota}\, y_k N_k} 
 = e^{i\sum_{k\leq\iota}\, (y_k/y_{\iota+1}) N_k}.$$
\end{lemma}
\begin{proof} By~\eqref{eq:knu-t},
\beq
      t_{\iota}(y) = \prod_{j>\iota}\, t_j^{\half\hat Y^j}
                   = y_{\iota+1}^{-\half\hat Y^{\iota+1}}
                     \prod_{j>\iota}\, y_{j+1}^{-\half\hat H_{j+1}}
      \label{eq:alt-knu-t}
\eeq
Accordingly, since $[\hat Y^{\iota+1},N_k] = -2N_k$ for $k\leq\iota+1$
whereas $[N_k,\hat H_j]=0$ for $j>k$ by \eqref{eq:ds-comm-1}
it follows by \eqref{eq:alt-knu-t} that (for $k\leq\iota$)
\begin{eqnarray*}
      \Ad(t^{-1}_{\iota}(y))y_k N_k
      &=& \Ad(y_{\iota+1}^{\half\hat Y^{\iota+1}})
                     \prod_{j>\iota}\, \Ad(y_{j+1}^{\half\hat H_{j+1}})y_k N_k   \\   
      &=&  \Ad(y_{\iota+1}^{\half\hat Y^{\iota+1}}) y_k N_k = y_k/y_{\iota+1} N_k 
\end{eqnarray*}
\end{proof}

\par Given a period map $F(z_1,\dots,z_r)$ with local normal form 
\eqref{eq:lnf} let 
\beq
        F_{\infty}(z_{\iota+1},\dots,z_r) = 
        e^{\sum_{j>\iota}\, z_j N_j}e^{\Gamma_{\iota}}.F_{\infty}
        \label{eq:partial-limit}
\eeq
be the limit mixed Hodge structure obtained by degenerating the variables
$z_1,\dots,z_{\iota}$ in $F(z)$.  Then, 
$(F_{\infty}(z_{\iota+1},\dots,z_r),W^{\iota})$ is an admissible variation of 
mixed Hodge structure.  Let 
\beq
      I'_{\iota} = \{ (z_{\iota+1},\dots,z_r)\in U^{r-\iota} \mid 
               x_{\iota+1},\dots,x_r\in [0,1],\quad 
               y_{\iota+1}\geq\cdots\geq y_r\geq 1\,\}
      \label{eq:strip-variant}
\eeq
Then, by Corollary~\eqref{corollary:twisted-limit}, for any $\ssl_2$-sequence
$z(m) = (z_{\iota+1}(m),\dots,z_r(m))\in I'_{\iota}$, 
the filtration 
\beq
   \tilde F_{\infty}(z_{\iota+1}(m),\dots,z_r(m)) 
    = t^{-1}_{\iota}(y) e^{\sum_{j>\iota} i y_j N_j}
      e^{\Gamma_{\iota}(s)}.F_{\infty}
    \label{eq:partial-limit-twist}
\eeq
converges to a filtration $F_{\natural}$ in the corresponding classifying space 
$\mathcal M_{\iota}$.  Furthermore, the filtration $F_{\natural}$ depends
only on the limiting values of $t_{\iota+1}(m),\dots,t_r(m)$.
% and the associated flag $\Theta$.

\begin{lemma}\label{lemma:pmhs-2} Let $(z_{\iota+1}(m),\dots,z_r(m))\in 
I'_{\iota}$ be an $\ssl_2$-sequence.  Let $F_{\natural}\in\mathcal M_{\iota}$ be
the associated limiting filtration constructed above.  Then, the data
$(N_1,\dots,N_{\iota},F_{\natural},W)$ generates an admissible nilpotent
orbit.
\end{lemma}  
\begin{proof} Since $(N_1,\dots,N_r,F_{\infty},W)$ generates an admissible
nilpotent orbit, all of the required relative weight filtrations exist.
Accordingly, we can assume that $F(z_1,\dots,z_r)$ is an effective
variation of pure Hodge structure of weight $\ell$.

\par Let 
\beq
     P_{\iota}(t) = \Ad(t^{-1}_{\iota}(y))(\sum_{j>\iota}\, i y_j N_j)
     \label{eq:alt-poly-twist}.
\eeq
Then, Lemma~\eqref{lemma:twist} applied to the nilpotent orbit generated
by $(N_{\iota+1},\dots,N_r,F_{\infty},W^{\iota})$ asserts that $P_{\iota}(t)$ 
is a polynomial in non-negative, half-integral powers of 
$t_{\iota+1},\dots,t_r$.  
Let $(\hat F_{\infty},W^r) = (e^{-\xi}.F_{\infty},W^r)$ be the
$\ssl_2$-splitting of $(F_{\infty},W^r)$.  Then, by 
Lemma~\eqref{lemma:weight-filt-preserve}, 
\beq
     \xi_{\iota}(t) = \Ad(t^{-1}_{\iota}(y)\xi       \label{eq:twisted-xi}
\eeq
is a polynomial in non-negative, half-integral powers of 
$t_{\iota+1},\dots,t_r$.  Accordingly,
\beq
       \tilde F_{\infty}(z_{\iota+1},\dots,z_r)
        = e^{P_{\iota}(t)}
          (\Ad(t^{-1}_{\iota}(y))e^{\Gamma_{\iota}(s)})e^{\xi_{\iota}(t)}.
          \hat F_{\infty}   \label{eq:twisted-lnf}
\eeq
where $\Ad(t^{-1}_{\iota}(y))e^{\Gamma_{\iota}(s)}$ converges along any
$\ssl_2$-sequence [apply Corollary~\eqref{corollary:gamma-eigenvalues}
to the variation $(F_{\infty}(z_{\iota+1},\dots,z_r),W^{\iota})$].  
Consequently,
\beq
       F_{\natural} = e^{\gamma_{\natural}}.\hat F_{\infty}   
            \label{eq:twisted-lnf-limit-1}
\eeq
where 
\beq
       e^{\gamma_{\natural}} = \lim_{m\to\infty}\,e^{P_{\iota}(t)}
          (\Ad(t^{-1}_{\iota}(y))e^{\Gamma_{\iota}(s)})e^{\xi_{\iota}(t)} 
          \label{eq:twisted-lnf-limit-2}
\eeq
Moreover (see equation~\eqref{eq:gamma-adN}), since $N_{\iota+1},\dots,N_r$, 
$\Gamma_{\iota}$ and $\xi$ all belong to the
subalgebra $\mathfrak q\cap(\cap_{k=1}^{\iota}\, \ker(\ad N_k))$, 
it follows from Corollary~\eqref{corollary:comm-with-t} that
\beq
       \gamma_{\natural}\in
       \mathfrak q\cap\ker(\ad N_1)\cap\cdots\cap\ker(\ad N_{\iota})
        \label{eq:twisted-lnf-limit-3}
\eeq
In particular, by virtue of equations \eqref{eq:twisted-lnf-limit-1}
and \eqref{eq:twisted-lnf-limit-3} it follows that $N_1,\dots,N_{\iota}$
are horizontal with respect to $F_{\natural}$.

\par By Theorem $(2.3)$ of \cite{Luminy}, to complete the proof, it 
suffices to show that 
$$
    z\mapsto e^{z(N_1 + \cdots + N_{\iota})}.F_{\natural}
$$ 
is a nilpotent orbit of pure Hodge structure.  To this end, let $a$ be a 
positive real number and 
$N_{\dag} = a(N_1 + \cdots + N_{\iota})$.  Observe that since
$$
     F_{\iota}(z) = e^{\sum_{j\leq\iota}\, z_j N_j}.
                    F_{\infty}(z_{\iota+1},\dots,z_r)
$$
is an admissible variation of mixed Hodge structure so is
$$
     F_{\dag}(z_{\iota},\dots,z_r) = e^{z_{\iota} N_{\dag}}e^{\sum_{j>\iota} z_j N_j}
                   e^{\Gamma_{\iota}(s)}.F_{\infty}
$$
(here we can use~\cite{Kashiwara} to derive the existence of the required
relative weight filtrations since $N_{\dag}$ belongs to the interior of
the cone generated by $N_1,\dots,N_{\iota}$).  The associated nilpotent
orbit of $F_{\dag}$ is generated by 
$(N_{\dag},N_{\iota+1},\cdots,N_r,F_{\infty},W)$.  Let $t_{\dag}(y)$ be 
the associated semisimple operator.  Then,
$$
      t_{\dag}(y) = t_{\iota}^{\half \hat Y_{\dag}} t_{\iota}(y)
$$
Let 
$
       P_{\dag}(t) = \Ad(t_{\dag}^{-1}(y))
                (iy_{\iota} N_{\dag} + i\sum_{j>\iota}\, y_j N_j)
$
and $\xi_{\dag}(t) = \Ad(t_{\dag}^{-1}(y))\xi$.  Then,
\begin{eqnarray*}
      \tilde F_{\dag}(z_{\iota},\dots,z_r)
      &=& t^{-1}_{\dag}(y) e^{i y_{\iota}N_{\dag}}e^{i\sum_{j>\iota}\, y_j N_j}
          e^{\Gamma_{\iota}}.F_{\infty}                             \\
      &=& e^{P_{\dag}(t)}(\Ad(t^{-1}_{\dag}(y))e^{\Gamma_{\iota}})
          e^{\xi_{\dag}(t)}.\hat F_{\infty}
\end{eqnarray*}
Let $z_{\dag}(m) = (z_{\iota}(m),\dots,z_r(m))$ be the $\ssl_2$-sequence 
obtained from the sequence $(z_{\iota+1}(m),\dots,z_r(m))$ by 
setting $z_{\iota}(m) = z_{\iota+1}(m)$.  Applying
Corollary~\eqref{corollary:twisted-limit} to 
$\tilde F_{\dag}(z_{\dag}(m))$, we then obtain an associated
limit filtration $F_{\sharp}\in\mathcal M$.

\par Regarding the filtration $F_{\sharp}$, we note that since 
$t_{\iota}(m) = y_{\iota+1}(m)/y_{\iota}(m)=1$ along the sequence
$z_{\dag}(m)$ it follows that $t_{\dag}(y) = t_{\iota}(y)$ along
$z_{\dag}(m)$.  Therefore, the limits of
$\Ad(t^{-1}_{\dag}(y))e^{\Gamma_{\iota}}$ and $\xi_{\dag}(t)$ along 
$z_{\dag}(m)$ coincide with the limits of 
$\Ad(t^{-1}_{\iota}(y))e^{\Gamma_{\iota}}$ and $\xi_{\iota}(t)$ along
the original sequence $(z_{\iota+1}(m),\dots,z_r(m))$.  Likewise,
along $z_{\dag}(m)$,
$$
         P_{\dag}(t) = \Ad(t^{-1}_{\iota}(y))
                          (i y_{\iota}N_{\dag} + i\sum_{j>\iota}\, y_j N_j) 
                     = i N_{\dag} + P_{\iota}(t)
$$
since $\Ad(t^{-1}_{\iota}(y)) y_{\iota} N_{\dag} = N_{\dag}$ by 
Lemma~\eqref{lemma:another-twist}.  Therefore, 
$$
           F_{\sharp} = e^{iN_{\dag}}.F_{\natural}\in\mathcal M
$$
In particular, since $N_{\dag} = a(N_1+\cdots+N_{\iota})$ with $a>0$
arbitrary, it follows that $e^{z(N_1+\cdots+N_{\iota})}.F_{\natural}$
is a nilpotent orbit of pure Hodge structure.
\end{proof}

\section{Proof of Theorem \eqref{theorem:conj-1}}
\label{s.proof} In this section we 
prove Theorem \eqref{theorem:conj-1} by induction on dimension $r$ of the base
$\Delta^{*r}$.  For $r=1$, Theorem~\eqref{theorem:conj-1} follows from 
Corollary~\eqref{corollary:induction-base}.  

Accordingly, assume that $r>1$ and let $z(m)$ be an 
$\ssl_2$-sequence.  Let $\iota$ be the smallest index such that $y(m)$ 
has non-polynomial growth with respect to $\iota$.  If $\iota = 0$ or 
$\iota=r$, Theorem~\eqref{theorem:conj-1} along $z(m)$ follows from 
Corollary~\eqref{corollary:extreme-iota}.  Therefore, we can assume that 
$r>1$ and $0<\iota<r$.  Therefore, by Theorem~\eqref{theorem:first-reduction}
it follows that
$$
       \hat Y_{(F(z(m),W)} - \hat Y_{(F_{\iota}(z(m)),W)} \to 0
$$
and hence the proof of Theorem~\eqref{theorem:conj-1} is reduced to the case
of period maps of the special form $F_{\iota}(z)$.

\par To complete the induction, we construct an associated flag 
in the spirit of \eqref{e.Flag} and verify that if 
\beq
       Y(\sum_{j\leq d}\, v_j(m)N(\theta^j),
         \hat Y_{(e^{\sum_{j>\iota}\, i y_j N_j}
                         e^{\Gamma_{\iota}(s)}.F_{\infty},W^{\iota})})\to Y_o
        \label{eq:reduction-5}
\eeq
then $e^{-N(x(m))}.\hat Y_{(F_{\iota}(z(m)),W)}\to Y_o$.  Having done this,
we then compute the limit \eqref{eq:reduction-5} using the induction
hypothesis and the properties of Deligne systems.

\subsection*{} As in equations \eqref{eq:partial-limit} and 
\eqref{eq:partial-limit-twist}, let us define
\begin{eqnarray*}
F_{\infty}(z_{\iota+1},\dots,z_r)
&=& e^{\sum_{j>\iota}\, z_j N_j}e^{\Gamma_{\iota}}.F_{\infty}             \\
\tilde F_{\infty}(z_{\iota+1},\dots,z_r)
&=& t^{-1}_{\iota}(y) e^{\sum_{j>\iota} i y_j N_j}
      e^{\Gamma_{\iota}(s)}.F_{\infty}
\end{eqnarray*}
% as in \eqref{eq:partial-limit} and \eqref{eq:partial-limit-twist}.  
Then, 
\begin{eqnarray*}
      F_{\iota}(z_1,\dots,z_r)
      &=& e^{\sum_j\, z_j N_j}e^{\Gamma_{\iota}}.F_{\infty} \\
      &=& e^{N(x)} e^{\sum_j\, iy_j N_j}e^{\Gamma_{\iota}}.F_{\infty}
                                                            \\
      &=& e^{N(x)}e^{\sum_{j\leq\iota} \, iy_j N_j}
          e^{\sum_{j>\iota}\, iy_j N_j}e^{\Gamma_{\iota}}.F_{\infty} \\
      &=& e^{N(x)}e^{\sum_{j\leq\iota} \, iy_j N_j}t_{\iota}(y).
          \tilde F_{\infty}(z_{\iota+1},\dots,z_r)           \\
      &=&  e^{N(x)}t_{\iota}(y)e^{i\sum_{k\leq\iota}\, (y_k/y_{\iota+1}) N_k}.
      \tilde F_{\infty}(z_{\iota+1},\dots,z_r) 
\end{eqnarray*}
where the last step is justified by Lemma~\eqref{lemma:another-twist}.

\par By equation~\eqref{eq:twisted-lnf}
$$
      \tilde F_{\infty}(z_{\iota+1},\dots,z_r)
        = e^{P_{\iota}(t)}
          (\Ad(t^{-1}_{\iota}(y))e^{\Gamma_{\iota}(s)})e^{\xi_{\iota}(t)}.
          \hat F_{\infty} 
$$
Moreover, by \eqref{eq:twisted-lnf-limit-1}, along the sequence 
$(z_{\iota+1}(m),\dots,z_r(m))$, 
$$
        \tilde F_{\infty}(z_{\iota+1}(m),\dots,z_r(m))\to F_{\natural}
               = e^{\gamma_{\natural}}.\hat F_{\infty}
$$

\par Define
$$
      \mathfrak v = \mathfrak q\cap\ker(\ad N_1)\cap\cdots\cap
                                            \ker(\ad N_{\iota})
$$
Then, by Corollary~\eqref{corollary:comm-with-t} it follows that
the function
$$
    e^{P_{\iota}(t)} (\Ad(t^{-1}_{\iota}(y))e^{\Gamma_{\iota}(s)})
    e^{\xi_{\iota}(t)}
$$
takes valued in $\mathfrak v$, and hence so does its limiting value
$\gamma_{\natural}$ along $(z_{\iota+1}(m),\dots,z_r(m))$.  Let 
\beq
      e^{\gamma(z_{\iota+1},\dots,z_r)}
      = e^{P_{\iota}(t)} (\Ad(t^{-1}_{\iota}(y))e^{\Gamma_{\iota}(s)})
        e^{\xi_{\iota}(t)}e^{-\gamma_{\natural}}   \label{eq:twisted-gamma}
\eeq
Then, $\tilde F_{\infty}(z_{\iota+1},\dots,z_r) 
= e^{\gamma(z_{\iota+1},\dots,z_r)}.F_{\natural}$, with
\beq
      \gamma(z_{\iota+1}(m),\dots,z_r(m))\to 0  \label{eq:twisted-gamma-limit}
\eeq

\par By Lemma~\eqref{lemma:pmhs-2}, the data $(N_1,\dots,N_{\iota},F_{\natural},W)$
defines an admissible nilpotent orbit, and hence by Lemma~\eqref{lemma:pmhs-1} 
there exists a neighborhood $\mathfrak v_o$ of zero in $\mathfrak v$ such
that for every $\nu\in\mathfrak v_o$ the data 
$(N_1,\dots,N_{\iota},e^{\nu}.F_{\natural},W)$ defines an admissible nilpotent 
orbit.  Shrinking $\mathfrak v_o$ as necessary, we can further assume that 
there exists a common constant $c$ such that if 
$\Im(z_1),\dots, \Im(z_{\iota})>c$ then
\beq
      e^{\sum_{j\leq\iota}\, z_j N_j}e^{\nu}.F_{\natural}\in\mathcal M
      \label{eq:nu-nilp-orbit}
\eeq
In particular, by \eqref{eq:twisted-gamma-limit}, there exists index
$m_o$ such that 
$$
      \gamma(z_{\iota+1}(m),\dots,z_r(m))\in\mathfrak v_o
$$
whenever $m>m_o$.

\begin{corollary}
\par Combining the above equations, it follows that along the given
$\ssl_2$-sequence $z(m)$, we can write
\beq
        F_{\iota}(z) =  e^{N(x)}t_{\iota}(y)
                        e^{i\sum_{k\leq\iota}\, (y_k/y_{\iota+1}) N_k}
                        e^{\gamma(z_{\iota+1},\dots,z_r)}.F_{\natural}
                        \label{eq:reduced-form-1}                       
\eeq
with $\gamma(z_{\iota+1}(m),\dots,z_r(m))$ taking values $\mathfrak v_o$ 
for $m>m_o$. 
\end{corollary}

\par Consider now the sequence 
$\tilde y(m) = (\tilde y_1(m),\dots,\tilde y_{\iota}(m))$
obtained by setting 
\beq
      \tilde y_j(m) = y_j(m)/y_{\iota+1}(m), 
                      \label{eq:new-divided-sequence}
\eeq
Then, since $y(m) = (y_1(m),\dots,y_r(m))$ is an $\ssl_2$-sequence 
% and $\iota$ was smallest index with respect to which $y(m)$ has 
% non-polynomial growth,  
it follows that $\tilde y(m)$ is also an $\ssl_2$-sequence in 
$\iota$-variables.  Therefore (see definition \eqref{defn:sl2-sequence})
we have
$$
         \tilde y(m) = T(\tilde v(m)) + \tilde b(m)                 
$$
for some linear map $T:\R^d\to\R^{\iota}$, some strict $\ssl_2$-sequence
$\tilde v(m)\in\R^d$ and a convergent sequence $b(m)\in\R^{\iota}$.  
Accordingly,
$$
        N(\tilde y(m)) = \sum_{j=1}^d\, N(\theta^j)\tilde v_j(m) + N(\tilde b(m))
$$
where $N(\theta^1),\dots,N(\theta^d)$ and $N(\tilde b(m))$ belong to 
the real subalgebra $\mathfrak b_{\R}$ of $\mathfrak v$ generated by 
$N_1,\dots,N_{\iota}$.  

\begin{warning}\label{warning:diff-theta} The system of $\theta$'s constructed 
in this way may differ from the sequence appearing in 
Theorem~\eqref{theorem:conj-1}.  Nonetheless, the result limit gradings will 
be the same.  This will addressed in the final paragraph of the paper.
\end{warning}

% For future reference, we define
% $\tilde\tau(m) = (\tilde\tau_1(m),\dots,\tilde\tau_d(m))$ by setting
% $$
%        \tilde\tau_j(m) = \tilde v_{j+1}(m)/\tilde v_j(m)
% $$
% where as usual we formally define $\tilde v_{d+1}(m)=1$.

% \begin{remark} Eventually, we must compare the elements 
% $N(\theta^1),\dots,N(\theta^d)$ defined above to the corresponding
% elements constructed from \eqref{e.Flag}.  This will be done at the
% end of the proof.  
% \end{remark}

\par Let $N(\tilde b(m))\to N(\tilde b_o)$ and $\mathfrak b_o$ be a 
neighborhood of $N(\tilde b_o)$ in $\mathfrak b_{\R}$.  Then, for any 
$\a\in\mathfrak b_o$ and any $\nu\in\mathfrak v_o$, the data 
$$
       (N(\theta^1),\dots,N(\theta^d);e^{i\a}e^{\nu}.F_{\natural},W)
$$ 
generates an admissible nilpotent orbit (with slightly larger $c$ to 
compensate for $\alpha$, see~\eqref{eq:nu-nilp-orbit}).  We therefore define
$$
    F(v_1,\dots,v_d;\alpha,\nu)
       = e^{\sum_{j\leq d}\, i v_j N(\theta^j)}e^{i\alpha}e^{\nu}.F_{\natural}
$$
As in the proof of Corollary~\eqref{corollary:extreme-iota}, it then
follows via Lemma~\eqref{lemma:morphisms-and-splittings} that for fixed 
$\alpha$ and $\nu$ as above, 
$$
     \hat Y_{(F(v_1,\dots,v_d;\alpha,\nu),W)}
     \to \hat Y(\nu) := Y(N(\theta^1),\dots,Y(N(\theta^d),
           \hat Y_{(e^{\nu}.F_{\natural},W^{\iota})}))
$$
along any strict $\ssl_2$-sequence in the variables $v_1,\dots,v_d$.

\par By Theorem $(0.5)$ of \cite{KNU}, there exists a constant $b$ such
that if  $\tau_1=v_2/v_1,\dots,\tau_d=1/v_d\in (0,b)$ then
\beq
       \hat Y_{(F(v_1,\dots,v_d;\a,\nu),W)}
       = \exp(u(\tau;\a,\nu)).\hat Y(\nu)          \label{eq:def-u}
\eeq
where $u(\tau;\a,\nu)$ has a convergent series expansion
\beq
      u(\tau;\a,\nu) = \sum_m\,  u_m(\a,\nu) \prod_{j=1}^r \tau_j^{m(j)}
                                                   \label{eq:u-series}
\eeq
with constant term 0.  Furthermore, by Theorem $(10.8)$ of \cite{KNU}, 
the coefficients $u_m(\a,\nu)$ are analytic functions of $\a\in\mathfrak b_o$ 
and $v\in\mathfrak v_o$. 

\begin{corollary}\label{corollary:reduction-4} Let 
$\nu(m) = \gamma(z_{\iota+1}(m),\dots,z_r(m))$ and $\a(m) = N(\tilde b(m))$.  
Then, 
\beq
  \hat Y_{(F_{\iota}(z(m)),W)}
    = e^{N(x(m))}t_{\iota}(y(m))
      e^{u(\tau(m);\a(m),\nu(m))}.\hat Y(\nu(m))
      \label{eq:reduction-4}
\eeq
where $\tau_j(m) = v_{j+1}(m)/v_j(m)$ for $v_j(m) = \tilde v_j(m)$. 
\end{corollary}
\begin{proof} Combine equations~\eqref{eq:reduced-form-1}, 
\eqref{eq:new-divided-sequence}, and \eqref{eq:def-u}.
\end{proof}

\par The remainder of the proof of Theorem~\eqref{theorem:conj-1} now divides 
into two parts depending on $d$:  

\begin{remark} The notation introduced in 
Corollary~\eqref{corollary:reduction-4} will remain in effect for the
remainder of this section.
\end{remark}

\subsection*{d=1}

\par In this case, it follows from the definition of non-polynomial growth that 
$$
     \tau_1(m)y_{\iota+1}^e(m)\to 0 
$$
for any half-integral power $e$, and hence the sequence
\beq
       e^{\eta(m)} := \Ad(t_{\iota}(y(m)))e^{u(\tau_1(m);\a(m),\nu(m))} 
       \to 1 \label{eq:eta-converge-1}
\eeq
since the action of $\Ad(t_{\iota}(y))$ on $\gg_{\C}$ is bounded 
by a polynomial in $y_{\iota+1}^{\half}$.  

\par On the other hand, by Lemma~\eqref{lemma:another-twist} and the properties
of Deligne systems it follows that 
\begin{eqnarray*}
      t_{\iota}(y).\hat Y(\nu)  
      &=& Y(\Ad(t_{\iota}(y))N(\theta^1),
            \hat Y_{(t_{\iota}(y)e^{\nu}.F_{\natural},W^{\iota})}))    \\ 
      &=& Y(y_{\iota+1}N(\theta^1),\hat Y_{(t_{\iota}(y)e^{\nu}.F_{\natural},W^{\iota})}))
\end{eqnarray*}

\par Now, if $(N,Y_M,W)$ is a Deligne system then so is 
$(\l N,Y_M,W)$ for any non-zero scalar. Furthermore,
\beq
       Y(\l N,Y_M) = Y(N,Y_M)              \label{eq:freedom-to-rescale}
\eeq
In particular, by the previous paragraph
\beq
     t_{\iota}(y).\hat Y(\nu)  
      = Y(N(\theta^1),
          \hat Y_{(t_{\iota}(y)e^{\nu}.F_{\natural},W^{\iota})})) 
     \label{eq:rescaled}
\eeq
Therefore, by equations \eqref{eq:reduction-4}, 
\eqref{eq:eta-converge-1} and \eqref{eq:rescaled} it follows that
\begin{eqnarray*}
     e^{-N(x(m))}.\hat Y_{(F_{\iota}(z(m)),W)}
      &=& e^{\eta(m)}. Y(N(\theta^1),
          \hat Y_{(t_{\iota}(y(m))e^{\nu(m)}.F_{\natural},W^{\iota})}))     \\ 
      &=& e^{\eta(m)}. Y(N(\theta^1),\hat Y_{(e^{\sum_{j>\iota}\, i y_j(m) N_j}
                   e^{\Gamma_i(s(m))}.F_{\infty},W^{\iota})}))
\end{eqnarray*}
and hence \eqref{eq:reduction-5} hold for $d=1$.  

\par By induction, 
$$
     \hat Y_{(e^{\sum_{j>\iota}\, i y_j(m) N_j}
              e^{\Gamma_{\iota}(s(m))}.F_{\infty},W^{\iota})})
          \to Y(N(\theta^2),\dots,Y(N(\theta^{d'}),\hat Y_{(F_{\infty},W^r)}))
$$
and hence combining the above, we have
$$
      e^{-N(x(m))}.\hat Y_{(F_{\iota}(z(m)),W)}
      \to Y(N(\theta^1),\dots,Y(N(\theta^{d'}),\hat Y_{(F_{\infty},W^r)}))
$$

\begin{remark} To complete the proof for $d=1$ we still need to 
resolve the discrepancy introduced by the different system of 
$\theta$'s as remarked in \eqref{warning:diff-theta}.
\end{remark}

\subsection*{$d>1$}

\begin{lemma}\label{lemma:degen-1} For fixed $\a$ and $\nu$ as 
in~\eqref{eq:def-u} and $\tau_1,\dots,\tau_{d-1}\in (0,b)$
\beq
     \exp(u(\tau_1,\dots,\tau_{d-1},0;\a,\nu).\hat Y(\nu)
     = Y(\sum_{j\leq d}\, \omega_j N(\theta^j),
         \hat Y_{(e^{\nu}.F_{\natural},W^{\iota})})  \label{eq:degeneration-1}
\eeq
for any vector $(\omega_1,\dots,\omega_d)\in\R^d_{>0}$ 
such that $\tau_j=\omega_{j+1}/\omega_j$ for $j\leq d-1$.
\end{lemma}
\begin{proof} Let $v_j = y \omega_j$ for $j=1,\dots,d$.  Then, 
by~\eqref{eq:def-u},
$$ 
 \hat Y_{(F(v_1,\dots,v_d;\a,\nu),W)}
       = \exp(u(\tau_1,\dots,\tau_{d-1},1/(y\omega_d);\a,\nu))
         .\hat Y(\nu)    
$$
On the other hand,
\begin{eqnarray*}
      \hat Y_{(F(v_1,\dots,v_d;\a,\nu),W)}
      &=& \hat Y_{(e^{\sum_{j\leq d}\, i v_j N(\theta^j)}
                 e^{i\a}e^{\nu}.F_{\natural},W)}                \\
      &=& \hat Y_{(e^{iy\sum_{j\leq d}\, \omega_j N(\theta^j)}
                   e^{i\a}e^{\nu}.F_{\natural},W)} 
\end{eqnarray*}
Comparing these two equations, it follows that
$$
    \exp(u(\tau_1,\dots,\tau_{d-1},1/(y\omega_d);\a,\nu))
         .\hat Y(\nu)
    = \hat Y_{(e^{iy\sum_{j\leq d}\, \omega_j N(\theta^j)}
              e^{i\a}e^{\nu}.F_{\natural},W)}
$$
Taking the limit as $y\to\infty$, we then obtain 
equation~\eqref{eq:degeneration-1} using \eqref{eq:knu-main}.
\end{proof}

\begin{remark} \emph{A priori}, $u(\tau;\a,\nu)$ is only defined for 
$\tau_1,\dots,\tau_d\in (0,b)$.  However, via the series expansion
\eqref{eq:u-series}, we can extend $u$ to a real-analytic
function on a neighborhood of $\tau=0$.  By continuity, the
formula for $u(\tau_1,\dots,\tau_{d-1},0)$ given above agrees
with the value of $u(\tau_1,\dots,\tau_{d-1},0)$ determined
by \eqref{eq:u-series}.
\end{remark}

\par For $\a$ and $\nu$ as in~\eqref{eq:def-u}, let 
\begin{eqnarray*}
        u_1(\tau;\a,\nu) &=& u(\tau_1,\dots,\tau_d;\a,\nu)
                           -u(\tau_1,\dots,\tau_{d-1},0;\a,\nu) \\
        u_2(\tau;\a,\nu) &=& u(\tau_1,\dots,\tau_{d-1},0;\a,\nu)
\end{eqnarray*}
Then, $\exp(u(\tau;\a,\nu)) = \exp(u_1 + u_2)$ where $u_1$ is 
divisible by $\tau_d$ in the ring of real-analytic functions of
$\tau_1,\dots,\tau_d$.  Therefore,
\beq
      \Ad(t_{\iota}(y))\exp(u(\tau;\a,\nu))
       = \Ad(t_{\iota}(y))(e^{u_1+u_2}e^{-u_2})
         \Ad(t_{\iota}(y))e^{u_2}   \label{eq:sep-var-2}
\eeq
where $e^{u_1+u_2}e^{-u_2} = e^{u_3}$ with $u_3$ again divisible by 
$\tau_d$ in the ring of real-analytic functions in 
$\tau_1,\dots,\tau_d$.  Consequently, as in~\eqref{eq:eta-converge-1}, 
it follows that if $\eta(m)$ is the sequence defined by the equation
\beq
        e^{\eta} = \Ad(t_{\iota}(y))(e^{u_3})
                                      \label{eq:sep-var-3}
\eeq         
along $z(m)$ then $\eta(m)\to 0$.  

\begin{lemma} 
$$
     e^{u_2(\tau(m);\a(m),\nu(m))}
        .\hat Y(\nu(m)) 
     = Y(\sum_{j\leq d}\, v_j(m)N(\theta^j),
         \hat Y_{(e^{\nu(m)}.F_{\natural},W^{\iota})})
$$
\end{lemma}
\begin{proof} Use Lemma~\eqref{lemma:degen-1} with $\omega_j(m) = v_j(m)$.
\end{proof}

\par Accordingly, by the previous Lemma and 
equations~\eqref{eq:reduction-4}, \eqref{eq:sep-var-2}, 
\eqref{eq:sep-var-3} it follows that  
\begin{eqnarray*}
     e^{-N(x(m))}.\hat Y_{(F_{\iota}(z(m)),W)}
     &=&  e^{\eta(m)}t_{\iota}(y(m))e^{u_2(\tau(m);\a(m),\nu(m))}
         .\hat Y(\nu(m)) \\
     &=&  e^{\eta(m)}t_{\iota}(y(m)).Y(\sum_{j\leq d}\, v_j(m)N(\theta^j),
         \hat Y_{(e^{\nu(m)}.F_{\natural},W^{\iota})})
\end{eqnarray*}
Therefore, using Lemma~\eqref{lemma:another-twist} and our freedom
to rescale $N(\theta^j)\to \l N(\theta^j)$ it follows that 
\begin{multline}
     t_{\iota}(y(m)).Y(\sum_{j\leq d}\, v_j(m)N(\theta^j),
         \hat Y_{(e^{\nu(m)}.F_{\natural},W^{\iota})})  \\
     = Y(\sum_{j\leq d}\, v_j(m)N(\theta^j),
         \hat Y_{(e^{\sum_{j>\iota}\, i y_j N_j}
                         e^{\Gamma_{\iota}(s)}.F_{\infty},W^{\iota})})
     \label{eq:new-deligne-form}
\end{multline}
and hence 
\begin{multline}
    e^{-N(x(m))}.\hat Y_{(F_{\iota}(z(m)),W)} 
    = e^{\eta(m)}.Y(\sum_{j\leq d}\, v_j(m)N(\theta^j),
                 \hat Y_{(e^{\sum_{j>\iota}\, i y_j N_j}
                         e^{\Gamma_{\iota}(s)}.F_{\infty},W^{\iota})})
         \label{eq:factoring-1}
\end{multline}
In particular, since $\eta(m)\to 0$, it follows that \eqref{eq:reduction-5}
holds for $d>1$.

% \begin{corollary}\label{corollary:reduction-5} If 
% \beq
%        Y(\sum_{j\leq d}\, v_j(m)N(\theta^j),
%          \hat Y_{(e^{\sum_{j>\iota}\, i y_j N_j}
%                          e^{\Gamma_{\iota}(s)}.F_{\infty},W^{\iota})})\to Y_o
%         \label{eq:reduction-5}
% \eeq
% then $e^{-N(x(m))}.\hat Y_{(F_{\iota}(z(m)),W)}\to Y_o$.
% \end{corollary}
% \begin{proof} This follows from~\eqref{eq:factoring-1} since $\eta(m)\to 0$. 
% \end{proof}

\par Thus, to complete the proof it remains to compute the limit
\eqref{eq:reduction-5}.
% \begin{remark} Resume editing here.  Need to use the fact the we are
% using the $\hat Y$ to split the limit MHS in the replacement for 
% equation following 8.2.6
% \end{remark}
To this end, observe that by induction, 
\beq
        \hat Y_{(e^{\sum_{j>\iota}\, i y_j N_j}
                         e^{\Gamma_{\iota}(s)}.F_{\infty},W^{\iota})}
        \to  Y^{\dag} = Y(N(\theta^{d+1}),\dots,Y(N(\theta^{d'}),
                    \hat Y_{(F_{\infty},W^r)}))
        \label{eq:induction-hyp-1}
\eeq
Consequently, there exists a unique $W^{\iota}_{-1}{\rm gl}(V)$-valued 
sequence $\b(m)$ which converges to zero such that 
$$
           \hat Y_{(e^{\sum_{j>\iota}\, i y_j N_j}
                         e^{\Gamma_{\iota}(s)}.F_{\infty},W^{\iota})}
            = e^{\b(m)}.Y^{\dag}
$$
along $z(m)$.

\par To continue, note that since
$\hat Y_{(e^{\sum_{j>\iota}\, i y_j N_j}
                         e^{\Gamma_{\iota}(s)}.F_{\infty},W^{\iota})}$ 
arises from the $\ssl_2$-splitting of the limit mixed Hodge structure of the 
nilpotent orbit
$$
    (z_1,\dots,z_{\iota}) \mapsto
    e^{\sum_j\, z_j N_j}.(e^{\sum_{j>\iota}\, i y_j N_j}
                          e^{\Gamma_{\iota}(s)}.F_{\infty})
$$
it follows that 
\begin{itemize}
\item[{(a)}] $[N_j,e^{\beta}.Y^{\dag}] = -2 N_j$ for $j\leq\iota$;
\item[{(b)}] $e^{\beta}.Y^{\dag}$ preserves $W^0,\dots,W^{\iota}$.
\end{itemize}

\par On the other hand, by Lemma~\eqref{lemma:associated-filtration} we know
that $Y^{\dag}= Y_{(\hat F_{\iota},W^{\iota})}$ arises via the limit mixed
Hodge structure of a nilpotent orbit
$$
          (e^{\sum_{j\leq\iota}\, z_j N_j}.\hat F_{\iota},W^0)
$$
(with limit MHS split over $\mathbb R$) we also have 
\begin{itemize}
\item[{(a')}] $[N_j,Y^{\dag}] = -2 N_j$ for $j\leq\iota$;
\item[{(b')}] $Y^{\dag}$ preserves $W^0,\dots,W^{\iota}$.
\end{itemize}
% Indeed, for fixed $(z_1,\dots,z_{\iota})$ with sufficiently large imaginary
% part, the map
% $$
%         (z_{\iota+1},\dots,z_r) \mapsto
%          e^{\sum_j\, z_j N_j}.F_{\infty}
% $$
% is a nilpotent orbit.  Therefore, since $(z_{\iota+1}(m),\dots,z_r(m))$ 
% is an $\ssl_2$-sequence, we obtain an associated nilpotent orbit with
% data
% $$
%    (N_1,\dots,N_{\iota},N(\theta^{d+1}),\dots,N(\theta^{d'}),F_{\infty}, W^0)
% $$
% The point $F_o$ is then obtained by applying the $\SL_2$-orbit theorem
% to the nilpotent orbit generated by 
% $$
%   (N(\theta^{d+1}),\dots,N(\theta^{d'}),F_{\infty},W^{\iota})       
% $$
% to obtain a split mixed Hodge structure $(F_o,W^{\iota})$.

%% Part of an even older draft.
%% In particular, for future reference, note that the correct inductive 
%% definition of $\theta^1,\dots,\theta^{d'}$ is to group the variables in 
%% the sequence $y(m)$ into blocks according to places where non-polynomial 
%% growth occurs, and then applying the procedure described after
%% \eqref{eq:new-divided-sequence} to each block.  

\begin{warning} For the remainder of this paper, $\hat F_{\iota}$ is the 
filtration of the previous paragraph and not the filtration obtained from 
equation
\eqref{eq:iter-1} and the nilpotent orbit $(N_1,\dots,N_r;\hat F_{\infty},W)$.
\end{warning}

\begin{remark} As in the remark to Theorem~\eqref{theorem:conj-1}, we
are implicitly assuming that $y_j(m)\to\infty$ for all $j$.  The case
where some $y_j(m)$ remain bounded is handled by absorbing these factors
into $F_{\infty}$.  The details are left to the reader.
\end{remark}

\par In particular, comparing $(a)$ and $(a')$ it follows from
Lemma~\eqref{lemma:commutator-lemma} that
% \beq
%         [e^{\b(m)}.Y^{\dag} - Y^{\dag},N_j] = 0,\qquad
%         j\leq\iota                   \label{eq:commutator-1}
% \eeq
% Since $Y^{\dag}$ grades $W^{\iota}$ and $\b\in W^{\iota}_{-1}\gl(V)$, we have
% $$
%        \b(m) = \sum_{j<0}\, \b_j(m) 
% $$
% where $\b_j(m)$ belongs to the $j$-eigenspace of $\ad Y^{\dag}$.
% Consequently, it follows from \eqref{eq:commutator-1} and the
% previous equation together with the fact that $\ad N_j$ lowers
% the eigenspaces of $\ad Y^{\dag}$ by $2$ that 
\beq
          [\b(m),N_j] = 0,\qquad j\leq\iota            \label{eq:commutator-2}
\eeq
Likewise, it follows from properties $(b)$ and $(b')$ that
\beq
          \b(m)\hphantom{a}\mbox{preserves}\hphantom{a} W^j,\qquad
          j\leq\iota
                                     \label{eq:preserve-w}
\eeq
By the functoriality of Deligne systems, it follows from
\eqref{eq:commutator-2} and \eqref{eq:preserve-w} that
\begin{multline}
  Y(\sum_{j\leq d}\, v_j(m)N(\theta^j),
         \hat Y_{(e^{\sum_{j>\iota}\, i y_j N_j}
                         e^{\Gamma_{\iota}(s)}.F_{\infty},W^{\iota})})   \\
     = Y(\sum_{j\leq d}\, v_j(m)N(\theta^j),e^{\beta(m)}.Y^{\dag})
     = e^{\b(m)}.Y(\sum_{j\leq d}\, v_j(m)N(\theta^j),Y^{\dag})
     \label{eq:reduction-6}
\end{multline}
% Invoking~\eqref{eq:new-deligne-form} yields
% \beq 
%           t_{\iota}(y(m))
%      e^{u_2(\tau(m);\a(m),\nu(m))}
%         .\hat Y(\nu(m))
%         = e^{\b(m)}.Y(\sum_{j\leq d}\, v_j(m)N(\theta^j),Y^{\dag})
%       \label{eq:new-deligne-form-2}
% \eeq
Moreover, by the properties of Deligne systems,
\beq
  Y(\sum_{j\leq d}\, v_j(m)N(\theta^j),Y^{\dag}) 
  = \hat Y_{(e^{i\sum_{j\leq d}\, v_j(m)N(\theta^j)}.\hat F_{\iota},W^0)}
  \label{eq:deligne-form-3}
\eeq
% and hence 
% \beq 
%                  t_{\iota}(y(m))
%      e^{u_2(\tau(m);\a(m),\nu(m))}
%         .\hat Y(\nu(m))
%      =  e^{\b(m)}.\hat Y_{(e^{i\sum_{j\leq d}\,v_j(m)N(\theta^j)}.
%              F_o,W^0)}
%         \label{eq:deligne-form-4}
% \eeq
Letting $m\to\infty$ and using \eqref{eq:knu-main} to compute the right
hand side shows that
$$
     Y(\sum_{j\leq d}\, v_j(m)N(\theta^j),Y^{\dag})
     \to \hat Y(N(\theta^1),\dots,Y(N(\theta^d), Y_{(\hat F_{\iota},W^{\iota})}))
$$
Therefore, since $\beta(m)\to 0$ and  
$$
     Y_{(\hat F_{\iota},W^{\iota})} = Y^{\dag} 
     = Y(N(\theta^{d+1}),\dots,Y(N(\theta^{d'}),
                    \hat Y_{(F_{\infty},W^r)})).
$$
it follows that from equation~\eqref{eq:reduction-6} that
$$
 Y(\sum_{j\leq d}\, v_j(m)N(\theta^j),
         \hat Y_{(e^{\sum_{j>\iota}\, i y_j N_j}
                         e^{\Gamma_{\iota}(s)}.F_{\infty},W^{\iota})}) 
 \to  Y(N(\theta^1),\dots,Y(N(\theta^{d'}),\hat Y_{(F_{\infty},W^r)}))
$$
Returning to equation~\eqref{eq:reduction-5} it then follows
that 
\beq 
      e^{-N(x(m))}.\hat Y_{(F(z(m)),W)} \to 
       Y(N(\theta^1),\dots,Y(N(\theta^{d'}),\hat Y_{(F_{\infty},W^r)}))
       \label{eq:new-limit}
\eeq
as required.

\par To complete the proof, we note that the elements 
$\theta^1,\dots,\theta^{d'}$ constructed above depend 
on the sequence $z(m)$ and not the period map.  Consequently, it follows from 
\eqref{eq:new-limit} that 
\beq
       e^{-N(x(m))}.\hat Y_{(F(z(m)),W)} - 
       e^{-N(x(m))}.\hat Y_{(e^{N(z(m))}.F_{\infty},W)} = 0
       \label{eq:invariant-limit-form}
\eeq
for every $\ssl_2$-sequence since both $F(z)$ and $e^{N(z)}.F_{\infty}$ have
the same limit Hodge filtration.  On the other hand, by 
Lemma~\eqref{lemma:main-limit}, we know that 
\eqref{eq:main-limit} holds for nilpotent orbits.  Thus, by virtue of
equation~\eqref{eq:invariant-limit-form}, equation~\eqref{eq:main-limit} 
is also true for period maps.

\part{Tannakian Categories of Nilpotent Orbits}\label{p.2}

\section{Central filtrations on Tannakian Categories}

The main purpose of this appendix is to prove
Theorem~\ref{theorem:SL2_Splitting} and Lemma ~\ref{lemma:ds-2} from
the body of the paper.  These results characterize the
$\ssl_2$-splitting as the unique splitting $\epsilon$ given as a
universal Lie polynomial in the $\delta_{p,q}$ such that, if $(N, F, W)$
is a one-variable nilpotent orbit with limit split over $\mathbb{R}$,
and $\xi=\epsilon(e^{iN}.F,W)$, then $Y(e^{\xi}e^{iN}.F,W)$ is a morphism
of the limit mixed Hodge structure $(F,M(W,N))$.
That this
is so was stated first by Deligne in an unpublished letter to Cattani
and Kaplan.   In this appendix, we phrase Deligne's results in the 
language of Tannakian categories in the hope that this will lead to 
clarity by making explicit the relationships between various categories
of Hodge structures and nilpotent orbits.  

\subsection{}\label{s.TannakaFiltered} Let $\mathbf{C}$ be a Tannakian
category over a field $k$
and let $\omega:\mathbf{C}\to\Vect_k$ be a fiber functor (where $\Vect_k$ denotes the category of finite dimensional vector spaces over $k$).  Recall from 
Saavedra-Rivano~\cite[p.~213]{Rivano}, that an \emph{exact filtration} of $\omega$ 
consists of an exhaustive, increasing filtration $W_k$ of the functor $\omega$
satisfying
\begin{enumerate}
\item the associated graded $\Gr^W\omega$ is exact;
\item for every $n\in\mathbb{Z}$ and every pair of objects $X,Y$ in $\mathbf{C}$, we have
$$
W_n\omega(X\otimes Y)=\sum_{p+q=n} W_p\omega X \otimes W_q \omega Y.$$
\end{enumerate}
See \cite{Rivano} for a definition where the field $k$ is replace with
a ring $A$.

Saavedra-Rivano calls a filtration $W$ \emph{central} if it arises by applying
$\omega$ to a filtration (also denoted $W_k$) of the identity functor
on $\mathbf{C}$. (So $W_k\omega X=\omega W_k X$.)  Suppose $W$ is an exact
central filtration.  Then we say that an object $X$ is pure of weight
$k$, if $\Gr^W_j X=0$ for all $j\neq k$.  An object is \emph{split} if
it is a direct sum of pure subobjects.  It follows from the exactness
of $W$, that the full subcategory $S_W\mathbf{C}$ consisting of all
split objects is a Tannakian subcategory of $\mathbf{C}$~\cite[1.7]{Milne2007}.  It also
follows that $\Hom_{\mathbf{C}}(X,Y)=0$ if $X$ and $Y$ are two pure
objects of $\mathbf{C}$ of different weights.

\subsection{} Suppose $\omega:\mathbf{C}\to\Vect_k$ is a neutral
Tannakian category with Galois group $G=\Aut^{\otimes} \omega$.  A \emph{grading}
(or, to be precise, a  
$\mathbb{Z}$-\emph{grading})   of $\omega$ is a functorial decomposition
$\omega=\oplus_{k\in\mathbb{Z}} \omega_k$.  A grading of $\omega$
amounts to the same thing as a group homomorphism $i:\mathbb{G}_m\to
G$ where $\mathbb{G}_m$ denotes the multiplicative group of $k$.  Any
grading, or equivalently a group homomorphism $i:\mathbb{G}_m\to G$,
gives rise to a filtration $W=W(i)$ on $\omega$ via the rule
$W_n\omega X=\oplus_{k\leq n} \omega_k X$.  
A filtration $W$ is said 
to be \emph{splittable} if $W=W(i)$ for some grading $i$.  In this case,
$i$ is said to be a \emph{splitting} of $W$.   

\subsection{}\label{s.PandU}
In~\cite[p.217]{Rivano}, Saavedra-Rivano considers affine group schemes
$P=\Aut_W^{\otimes} (\omega)$ and $U=\Aut^{\otimes !}(\omega)$
associated to a neutral Tannakian category
$\omega:\mathbf{C}\to\Vect_k$ and an exact filtration $W$ of $\omega$.
The group $P$ consists of automorphisms preserving $W\omega X$,
for any $X\in\mathbb{C}$,  while the
group $U$ consists of automorphisms inducing the identity on
$\Gr^W\omega$.  More generally, Saavedra-Rivano considers filters $P$ by the subgroups
$U_{\alpha}$ consisting of elements $g\in P$ acting trivially on 
$(W_p/W_{p+\alpha})\omega X$.  In this notation, $U=U_{-1}$ and we set $P=U_0$.  
To express the the fact that the filtration $U_{i}$ depends on $W$, we write $W_iU:=U_i$.

If $W$ is central, then $P=G=\Aut^{\otimes}(\omega)$.

\begin{proposition}\label{p.ExactSequence}
  Suppose $\omega:\mathbf{C}\to\Vect_k$ is a neutral Tannakian
  category equipped with a central filtration $W$.  Let
  $Q=\Aut^{\otimes}(\omega|S_W\mathbf{C})$.  Then the map $\pi:G\to Q$
  induced by the inclusion of $S_W\mathbf{C}$ in $\mathbf{C}$ is
  faithfully flat with kernel $U$.  Thus we have a short exact
  sequence
$$
1\to U\to G\to Q \to 1
$$
of $k$-groups.   Thus, $G/U$ is isomorphic to $Q$.  
\end{proposition}

\begin{proof}
  By~\cite[Proposition 2.21]{DeligneMilne}, $G\to Q$ is faithfully
  flat, because the inclusion $S_W\mathbf{C}\to \mathbf{C}$ is fully
  faithful and every subobject in $\mathbf{C}$ of a split object is
  split.  It is clear that $U$ is in the kernel of $\pi$.  On the other
hand, ever object in $\mathbb{C}$ is a successive extension of pure objects.
From this observation, it follows that $U$ contains the kernel of $\pi$. 
\end{proof}

\subsection{} Suppose the filtration $W$ in Proposition~\ref{p.ExactSequence}
is splittable by a homomorphism $i:\mathbb{G}_m\to G$.   Let $\Cent(i)$
denote the centralizer of $i$ in $G$.   

\begin{proposition}[Saavedra-Rivano]
  If $W$ is splittable, then $G=U\rtimes\Cent(i)$.   In particular, $\Cent(i)$ 
is isomorphic to $Q$.  
\end{proposition}

\subsection{}  In the context of the Proposition, the 
set of splittings $i:\mathbb{G}_m\to G$ is a pseudo-torsor under $U(k)$ (acting
via conjugation).     Fortunately, all the filtrations which come up 
in Hodge theory are, in fact, splittable.  One way to see this is to use
the following result, a corollary of a recent theorem of Ziegler.

\begin{theorem}\label{t.Splittable}
 Suppose $\omega:\mathbb{C}\to\Vect_k$ is a neutral Tannakian category 
over a field $k$ of characteristic $0$ and let $W$ be a filtration on $\omega$.
Then $W$ is splittable 
\end{theorem}

\begin{proof}
  This follows directly from Theorem 1.3 of~\cite{Ziegler} and the fact 
that, in characteristic $0$, algebraic groups are smooth.
\end{proof}

\section{Mixed Hodge Structures}

\subsection{} The category $\MHS$ of mixed Hodge structures over
$\mathbb{R}$ equipped with the forgetful functor $\omega:\MHS\to
\Vect_{\mathbb{R}}$ sending a Hodge structure to its underlying real
vector space is a neutral Tannakian category.  The weight filtration
$W$ induces a central filtration on $\MHS$ and the category $S_W\MHS$
is simply the category $\HS$ of split mixed Hodge structures.  

By a classical observation of Deligne, the group
$\mathbb{S}:=\Aut^{\otimes}(\omega| \HS)$ is the Weil restriction
of scalars from $\mathbb{C}$ to $\mathbb{R}$ of the group
$\mathbb{G}_m$.    This group sits in exact sequences
$$\xymatrix{
  1 \ar[r] &  S^1\ar[r]^{s} & \mathbb{S}\ar[r]^{t} & \mathbb{G}_m \ar[r] & 1\\
  1 \ar[r] & \mathbb{G}_m\ar[r]^{w} & \mathbb{S} \ar[r] & S^1 \ar[r] &
  1 }$$ where $S^1$ is the unique (up to isomorphism) non-split real
form of $\mathbb{G}_m$ and the action of $\mathbb{G}_m$ via $w$ on a
split mixed Hodge structure induces the splitting of the weight
filtration.  The homomorphism $t:\mathbb{S}\to\mathbb{G}_m$ induces
a fully faithful functor from the category of graded vector spaces
to the category of split mixed Hodge structures whose essential image consists
of the Tate mixed Hodge structures.  The map $s\times w: S^1\times\mathbb{G}_m\to \mathbb{S}$ presents $\mathbb{S}$ as the quotient of $S^1\times\mathbb{G}_m$ by 
the diagonally embedded copy of $\mathbb{Z}/2$.

\subsection{} By Theorem~\ref{t.Splittable}, the central filtration $W$ on $\MHS$ is
splittable. So $\mathfrak{M}:=\Aut^{\otimes}(\omega)$ is isomorphic to
a semi-direct product $\mathfrak{U}\rtimes \mathbb{S}$ where
$\mathfrak{U}=\Aut^{\otimes!}(\omega)$.  In fact, Deligne  determined the
structure of $\mathfrak{U}$, which, for the convenience of the reader,
we explain Deligne's language.

Let $\mathcal{L}_{\mathbb{C}}$
denote the free Lie algebra over $\mathbb{C}$ on generators $D^{i,j}$
where $i$ and $j$ are negative integers.   The $\mathbb{C}$-vector
space underlying $\mathcal{L}_{\mathbb{C}}$ carries a unique bigrading 
$\mathfrak{L}_{\mathbb{C}}=\oplus\mathfrak{L}_{\mathbb{C}}(i,j)$ for which 
$D^{i,j}$ is in bidegree $(i,j)$.   Let 
$$W_n\mathcal{L}_{\mathbb{C}}=\oplus_{i+j\leq n} \mathcal{L}_{\mathbb{C}}(i,j).$$
Then $W_n\mathcal{L}_{\mathbb{C}}$ is an ideal in $\mathcal{L}_{\mathbb{C}}$, and it is easy to see that the Lie
algebra $\mathcal{L}_{\mathbb{C}}/W_n\mathcal{L}_{\mathbb{C}}$ is nilpotent and finite
dimensional as a $\mathbb{C}$ vector space.  There is a unique real
structure on the Lie algebra $\mathcal{L}_{\mathbb{C}}$ for which
$\bar{D^{i,j}}=-D^{j,i}$.  Let $\mathcal{L}$ denote the corresponding
real Lie algebra.   Since the real structure respects the filtration $W$ on
$\mathcal{L}_{\mathbb{C}}$, $W$ descends to a filtration on $\mathcal{L}$.
Set $\mathfrak{U}(n):=\exp(\mathcal{L}/W_n\mathcal{L})$, a real algebraic
unipotent group.   

The bigrading on $\mathcal{L}_{\mathbb{C}}$ induces an action of
$\mathbb{G}_m^2$ on $\mathcal{L}_{\mathbb{C}}$ and, thus, an action on
$\mathfrak{U}(n)\otimes\mathbb{C}$.  Under this action
$(s,t)D^{i,j}=s^it^jD^{i,j}$.  This action descends to an action of
$\mathbb{S}$ on $\mathcal{L}$ for which
$tD^{i,j}=z^i(t)\bar{z}^j(t)D^{i,j}$ where $z,\bar{z}$ denote
conjugate generators of the character group of $\mathbb{S}$.  The
action of $\mathbb{S}$ on $\mathcal{L}$ then induces an action on the
real algebraic group $\mathfrak{U}(n)$.  Let $\varprojlim \mathfrak{U}(n)$ denote the
inverse limit of the real algebraic groups $\mathfrak{U}(n)$, a pro-unipotent real
affine group scheme, and write $W_k\varprojlim\mathfrak{U}(n)$ for the kernel
of the canonical homomorphism $\varprojlim\mathfrak{U}(n)\to \mathfrak{U}(k)$. 

Suppose $V=(V,F,W)$ is a real mixed Hodge structure.  Let
$\delta=\delta_{F, W}\in \End(V)$ and write
$\delta=\sum_{i,j<0}\delta_{i,j}\in \End(V)_{\mathbb{C}}$.  Since $\delta$ is
real $\bar{\delta}_{i,j}=\delta_{j,i}$  This induces an action of $\varprojlim
\mathfrak{U}(n)$ on $V$ by letting $D^{i,j}$ act via $\sqrt{-1}\delta_{i,j}$.
This induces a homomorphism $\varprojlim \mathfrak{U}(n)\to \mathfrak{M}$, and, 
since $\mathfrak{U}(n)$ is unipotent for each $n$, the homomorphism must
factor through $\mathfrak{U}$.   So we have a map $\varprojlim\mathfrak{U}(n)\to \mathfrak{U}$.

\begin{theorem}[Deligne]\cite{DeligneSeattle}
We have $\mathfrak{U}=\varprojlim \mathfrak{U}(n)$.   
\end{theorem}

\begin{remark}
If $(V,F, W)$ is a real mixed Hodge structure, then $(V,e^{-i\delta}.F, W)$
is split over $\mathbb{R}$.   The splitting $Y_{(e^{-i\delta}.F,W)}$ gives a 
canonical $\mathbb{R}$ grading of $W$, and thus a homomorphism $i:\mathbb{G}_m\to \mathfrak{M}$, which induces a splitting of the sequence 
$$
1\to \mathfrak{U}\to \mathfrak{M}\to \mathbb{S}\to 1.
$$
If we set $\mathfrak{M}(n):=\mathfrak{M}/W_n\mathfrak{U}$, then $\mathfrak{M}(n)$ is a real algebraic group and $\mathfrak{M}=\varprojlim\mathfrak{M}(n)$.      
\end{remark}

\section{Nilpotent Orbits}

\subsection{} Suppose $V$ is a finite dimensional real vector space.  Recall,
from the body of the paper or from~\cite[p.~405]{KNU}, that the data of 
an admissible nilpotent orbit (or mixed nilpotent orbit in the terminology
of~\cite{KNU}) consists of a quadruple $(V,N,F,W)$ where 
\begin{enumerate}
\item $N$ is a nilpotent endomorphism of $V$,
\item $F$ is a decreasing filtration of $V_{\mathbb{C}}$,
\item $W$ is an increasing filtration of $V$.
\end{enumerate}
These data are assumed to satisfy several conditions spelled out
in~\cite{KNU}.  The class of all admissible nilpotent orbits
$(V,N,F,W)$ forms a category in an obvious way (morphisms are vector
space homomorphisms preserving $N$, $F$ and $W$).  Moreover, by a
result of Kashiwara~\cite[Proposition 5.2.6]{Kashiwara} this category,
which we will call $\Nil_1$, is abelian.  In fact, $\Nil_1$ is a
neutral Tannakian category: the tensor products are defined as
in~\cite[4.4.3]{Kashiwara} in an obvious way, the functor $\omega_1$
from $\Nil_1$ to the category $\Vect_{\mathbb{R}}$ of finite
dimensional real vector spaces is the one that forgets everything but
$V$. From this, it is easy to check that $\Nil_1$ is a neutral
Tannakian category using, for example, ~\cite[Proposition
1.20]{DeligneMilne}.  The filtration $W$ equips $\omega_1$ with a central
filtration.  

\subsection{} Unfortunately, we do not have a simple description of
$\Nil_1$.  However, $\Nil_1$ has a Tannakian subcategory whose
fundamental group is more tractable: the full subcategory 
$\Split_1$ consisting
of all object $(V,N,F, W)$ with limit $(V,F, M(N,W))$ split
over $\mathbb{R}$.  Let $\mathfrak{M}_1$ denote the affine group
scheme $\Aut^{\otimes}(\omega_1|\Split_1)$.   The filtration $W$  
induces a central filtration on the neutral Tannakian category
$\omega_1:\Split_1\to\Vect_{\mathbb{R}}$.  We write $\mathbf{S}_1$ for the
category $S_W\Split_1$ of all split objects in $\Split_1$.   Objects
in $\mathbf{S}_1$ are called \emph{$\SL_2$-orbits}.   Write $\mathbb{S}_1=
\Aut^{\otimes}(\omega_1|\mathbf{S}_1)$.    Then we have a (splittable) 
faithfully flat homomorphism $\pi_1:\mathfrak{M}_1\to \mathbb{S}_1$, and, if 
we write $\mathfrak{U}_1$ for the kernel of $\pi_1$, we obtain a (splittable)
exact sequence
$$
1\to \mathfrak{U}_1\to \mathfrak{M}_1\to \mathbb{S}_1\to 1.
$$
In fact, $\mathfrak{M}_1$ inherits a filtration $W_k\mathfrak{M}_1$ from $W$
as in~\eqref{s.PandU}, and $\mathfrak{U}_1=W_1\mathfrak{M}_1$.
 
For each non-negative integer $n$, set
$\mathfrak{M}_1(n):= \mathfrak{M}_1/W_n\mathfrak{U}_1$.  Then
$\mathfrak{M}_1(n)$ is a real algebraic group with unipotent radical
$\mathfrak{U}_1(n)$.  The homomorphism $\pi_1(n):\mathfrak{M}_1(n)\to
\mathbb{S}_1$ induced by $\pi_1$ sends $\mathfrak{M}_1(n)$ onto its
largest reductive quotient.  We have $\mathfrak{M}=\varprojlim
\mathfrak{M}_1(n)$.

Since the limit $(V,F,M(N,W))$ of an object $(V,N,F,W)$ in $\Split_1$ is, by
definition, split over $\mathbb{R}$, we obtain morphism 
$i_M:\mathbb{G}_m\to \mathfrak{M}_1$ inducing a grading of $M$.  
We write $\bar i_M:=\pi_1\circ i_M:\mathbb{G}_m\to\mathbb{S}_1$.   
Moreover, write $H:=i_M(\mathbb{G}_m)$ and $\bar H=\bar i_M(\mathbb{G}_m)$.

Note that, any split real mixed Hodge structure, $(V,F,W)$ gives rise
to a split nilpotent orbit $(V,0,F,W)$ in a trivial way.  In other
words, each split real mixed Hodge structure gives rise to a constant
split nilpotent orbit.  This association gives rise to a tensor
functor $\Const:\HS\to\mathbb{S}_1$ and, thus, to a homomorphism
$\const:\mathbb{S}_1\to\mathbb{S}$.  

The following proposition is proved in~\cite[Lemma 3.12]{CKS}.

\begin{proposition}\label{p.FunctorFromMixedToSmooth}
  Suppose $(V,N,F,W)$ is an object in $\Split_1$ and $z$ is a complex
  number in the upper-half plane.  Then $(V,e^{zN}.F,W)$ is a real
  mixed Hodge structure.
\end{proposition}

For every $z$ in the upper-half plane, 
the proposition gives a functor $\spec_z:\Split_1\to \MHS$ compatible
with $\otimes$, the filtration $W$, and the forgetful functors to
$\Vect_{\mathbb{R}}$.  (We call it $\spec_z$ for ``specialization''
because it corresponds to specializing the nilpotent orbit to the
point $i$ in the upper half plane.) 
The functor $\spec_z$ induces, in
turn, a group homomorphism $\spec_z:\mathfrak{M}\to\mathfrak{M}_1$
compatible with the $W$ filtration of both groups.  
For most purposes, it suffices to consider the case $z=i$. 

The following is
the main theorem of the appendix.

\begin{theorem}\label{t.AppMain}  Concerning $\mathfrak{M}_1$ and its relationship to $\mathfrak{M}$, we have the following results.
\begin{enumerate}
\item The group $\mathbb{S}_1$ is isomorphic to $(\mathbb{G}_m\times S^1\times \SL_2)/\mu_2$ where $\mu_2$ is embedded in diagonally in the product and $S^1$
denotes the unique non-split real form of $\mathbb{G}_m$.
\item  The homomorphism $\spec_i:\mathfrak{M}\to \mathfrak{M}_1$ is
injective.
\item The restriction of $\spec_i$ to $\mathfrak{U}$ induces an
  isomorphism of $\mathfrak{U}$ with $\mathfrak{U}_1$, thus a
  commutative diagram
$$\xymatrix{
1\ar[r]& \mathfrak{U}\ar[r]\ar[d]^{\cong}& \mathfrak{M}\ar[r]^{\pi}\ar[d]^{\spec_i} &
\mathbb{S}\ar[r]\ar[d]^{\overline{\spec}_i} & 1\\
1\ar[r] & \mathfrak{U}_1\ar[r] & \mathfrak{M}_1\ar[r]^{\pi_1} &
\mathbb{S}_1\ar[r] & 1.\\
}$$ 
\item 
$\Cent(H)\cap\mathfrak{U}_1=\{1\}$.  On the other hand,
$\overline{\spec}_i(w(\mathbb{G}_m))$ is central in $\mathbb{S}_1$.  
\end{enumerate}  
\end{theorem}

\begin{corollary}\label{c.Why}
There is a 
unique splitting $\sigma:\mathbb{S}\to\mathfrak{M}$ such that 
$\spec_i\circ\sigma(w(\mathbb{G}_m))$ commutes with $H$.
\end{corollary}

\begin{proof}
Since $\Cent(H)\cap\mathfrak{U}_1=\{1\}$, the restriction of $\pi_1$
to $\Cent(H)$ gives an injection $\Cent(H)\to\mathbb{S}_1$.
Let $\bar{H}$ denote the image of $H$ in $\mathbb{S}_1$.
We then
obtain a commutative diagram of group homomorphisms
$$\xymatrix{
\Cent(H)\cap\mathfrak{M}\ar[r]^{\pi}\ar[d] & \Cent{\bar{H}}\cap\mathbb{S}\ar[d]\\
\Cent(H)\ar[r]^{\pi_1}                       & \Cent{\bar{H}}
}$$
where all arrows are injective and we regard $\mathfrak{M}$ as being 
contained $\mathfrak{M}_1$ via $\spec_i$.

It follows from~\cite[Exp. XV Lemma 7.2]{SGA3.2} that the horizontal
arrows are surjective (and, therefore, isomorphisms).  Thus, since
$w(\mathbb{G}_m)$ is contained in $\Cent(\bar{H})\cap\mathbb{S}$,
there exists a unique section of $\pi$ over $w(\mathbb{G}_m)$
commuting with $H$.  This gives a grading 
$i:\mathbb{G}_m\to\mathfrak{M}$ of $W$.  The centralizer of $i$ in $\mathfrak{M}$ gives a
splitting $\sigma:\mathbb{S}\to\mathfrak{M}$ of $\pi$.  It is clearly
the unique splitting $\sigma$ such that $\sigma(w(\mathbb{G}_m)$
commutes with $H$.

\end{proof}

% \begin{proof} 
% Using the filtrations $W_n$ on $\mathfrak{M}$ and $\mathfrak{M}_1(n)$, we obtain
% commutative diagrams
% $$\xymatrix{
% 1\ar[r]& \mathfrak{U}(n)\ar[r]\ar[d]^{\cong}& \mathfrak{M}(n)\ar[r]^{\pi(n)}\ar[d]^{\spec_i(n)} &
% \mathbb{S}\ar[r]\ar[d]^{\overline{\spec}_i} & 1\\
% 1\ar[r] & \mathfrak{U}_1(n)\ar[r] & \mathfrak{M}_1(n)\ar[r]^{\pi_1(n)} &
% \mathbb{S}_1\ar[r] & 1.\\
% }$$
% Let $H(n)$ denote the image of $H$ in $\mathfrak{M}_1(n)$. Then it suffices
% to show that $\pi(n)$ maps  $\spec_i(n)^{-1}\Cent(H(n))$ isomorphically  onto $\mathbb{S}$.

% It follows from \cite[Exp. XV Lemma 7.2]{SGA3.2} that $\Cent(H(n))$ is
% mapped surjectively onto $\Cent(\bar H)$ via $\pi_1(n)$.  On the other hand, since
% $\Cent(H)\cap\mathfrak{U}_1=\{1\}$, the restriction of $\pi_1$ to
% $\Cent(H)$ is injective.

% Since $\overline{\spec}_i(\mathbb{S})\subset \Cent(\bar{H})$, it follows 
% that $\spec_i^{-1}(\Cent(H))$ maps isomorphically onto $\mathbb{S}$.  
% \end{proof}

\begin{lemma}\label{l.KNU-Epsilon}
Suppose $(V,N,F,W)$ is a one-variable admissible nilpotent orbit with 
limit split over $\mathbb{R}$.   Set $F_{(0)}=e^{iN}.F$.   Then 
$Y(\hat F_{(0)},W)$ is a morphism of the limit mixed Hodge structure
$(V,F,M(N,W))$.  
\end{lemma}

\begin{proof}
This follows from~\cite[10.1.3]{KNU}.
\end{proof}

The following is a restatement of Theorem~\ref{theorem:SL2_Splitting}.

\begin{corollary}\label{c.SL2_Splitting}
The $\ssl_2$-splitting is the unique functorial
splitting on the category of real mixed Hodge structures given by  a
universal Lie polynomials $\epsilon$ in the $\delta_{p,q}$ satisfying 
the following property.   If $(N,F,W)$ is a one-variable nilpotent
orbit with limit split over $\mathbb{R}$ and we set 
$\xi=\epsilon(e^{iN}.F,W)$, then $Y(e^{-\xi}e^{iN}.F,W)$ is a morphism
of $(F,M(N,W))$. 
\end{corollary}

\begin{proof}
  Universal Lie polynomials in the the $\delta_{p,q}$ are in one-one
  correspondence with elements of the completed Lie algebra
  $\hat{\mathcal{L}}:= \varprojlim\mathcal{L}/W_n\mathcal{L}$.  By the
  Baker-Campbell-Hausdorff theorem, the map $\zeta\mapsto e^{\zeta}$ gives an
  isomorphism from $\hat{\mathcal{L}}$ to $\mathfrak{U}_1(\mathbb{R})$.
  
  Gradings $i:\mathbb{G}_m\to\mathfrak{M}$ of the weight filtration
  form a torsor under $\mathfrak{U}(\mathbb{R})$.  If we let
  $i_{\delta}$ denote the grading $Y(e^{-i\delta}.F,W)$, then any other
  grading $i_{\zeta}:\mathbb{G}_m\to\mathfrak{M}$ of the weight filtration
  is given by $Y(e^{\zeta}e^{-i\delta}.F,W)$ where $\zeta\in\hat{\mathfrak{L}}$
  is a universal Lie polynomial in the $\delta_{p,q}$.  The grading is
  defined over $\mathbb{R}$ iff the universal Lie polynomial $\zeta$ is.

  Note that there is a one-one correspondence between splittings
  $\sigma:\mathbb{S}\to\mathfrak{M}$ and gradings
  $i:\mathbb{G}_m\to\mathfrak{M}$.  This correspondence sends the
  splitting $\sigma$ to its restriction to
  $w(\mathbb{G}_m)\subset\mathbb{S}$.  The inverse of the
  correspondence sends the grading $i$ to its centralizer (in
  $\mathfrak{M}$), which is isomorphic to $\mathbb{S}$.  Under this
  correspondence, the $\mathbb{S}$ representation given by the split
  mixed Hodge structure $(e^{\zeta}e^{-i\delta}.F,W)$ is sent to the
    splitting $Y(e^{\zeta}e^{-i\delta}.F,W)$.

    Corollary~\ref{c.Why} shows that there is a unique splitting
    $\sigma:\mathbb{G}_m\to\mathfrak{M}$ such that $\spec_i(\sigma)$
    commutes with $H$.  Thus, there can be only real one universal Lie
    polynomial $\zeta$ in the $\delta_{p,q}$ such that, when
    $\delta=\delta(e^{iN}.F,W)$, $Y(e^{\zeta}e^{-i\delta}e^{iN}.F,W)$
    commutes with the grading $Y(F,M)$ of the limit mixed Hodge
    structure.  Now, any $\zeta$ such that
    $Y(e^{\zeta}e^{-i\delta}e^{iN}.F,W)$ is a morphism of the limit
    mixed Hodge structure commutes with the limit grading $Y(F,M)$.
    The result now follows immediately from Lemma~\ref{l.KNU-Epsilon}.

\end{proof}

\section{Deligne's Splittings}

\subsection{} Suppose $V$ is a vector space over a field $F$ of
characteristic $0$, $N$ is a nilpotent operator on $V$ and $W$ is an
exhaustive, separated increasing filtration of $V$ (indexed by
integers) such that $NW_i\subset W_{i-1}$.  The triple $(V,N,W)$ is
called \emph{admissible} if the relative weight filtration $M=M(N,W)$
exists.  (By~\cite{Weil2}, $M$ is unique if it does exist.)

A grading $Y_M$ of $M$ is said to be \emph{compatible} with $N$ and $W$
if 
\begin{enumerate}
\item $[Y_M,N]=-2N$;
\item for all $i$, $Y_M(W_i)\subset W_i$.
\end{enumerate}

If $Y_M$ is such a grading, let $G(W,Y_M)$ denote the set of gradings
of $W$ which commute with $Y_M$.  It is not hard to see that $G(W,Y_M)$ is
non-empty.  In fact, if we let $P$ denote the subgroup of $\GL(V)$
consisting of $g$ such that $gY_Mg^{-1}=Y_M$ and $(g-1)W_i\subset
W_{i-1}$, then $G(W,Y_M)$ is a principle homogeneous space for $P$.

Suppose $Y_W\in G(W,Y_M)$.  Then set $H=H(Y_W)=Y_M-Y_W$.  Let $N_i$
denote the $i$-th eigencomponent for $N$ under the action of $\ad
Y_W$.  Since $N$ preserves $W$, $N=\sum_{i\in\mathbb{Z}_{\leq 0}}
N_i$.  It follows from the definition of the relative weight
filtration, that the pair $(N_0,H)$ is an $\ssl_2$-pair.  Let
$N_+=N_+(Y_W)$ denote the unique element of $\End V$ such that
$(N_0,H,N^+)$ is an $\ssl_2$-triple.

\begin{theorem}[Deligne]\label{t:DeligneSplittings}
  Suppose that $(V,N,W)$ is an admissible triple and $Y_M$ is grading of $M$
  compatible with $N$ and $W$.  Then there exists a unique grading
  $Y_W=Y(N,Y_M)\in G(W,Y_M)$ with respect to which,  for each $i<0$, $[N^+,N_i]=0$.
\end{theorem}

\begin{proof} The result was proved originally in a letter from
  Deligne to Cattani and Kaplan.  For a published proof,
  see~\cite{PearlsteinDegenerationsDuke}.
\end{proof}

\subsection{}   Suppose $(V,N,F,W)$ is an object in $\Nil_1$. 
Set $Y_M=Y_{(F,M)}$.  Then $Y_M$ is a morphism of $(V,F,M)$ and 
the weight filtration $W$ is a filtration of $(V,F, M)$ by 
subobjects.   Therefore, $Y_M$ preserves $W$, and, as $N$ is a $(-1,-1)$ 
morphism of the limit mixed Hodge structure $(V,F,M)$, $[Y_M,N]=-2N$.   So, by 
Theorem~\ref{t:DeligneSplittings} $(V,N,F, M)$, 
gives rise to grading $Y_W=Y(N,Y_{(F,M)})$ 
of $W$ and an $\ssl_2$-triple $(N_0,H,N_+)$.

Because of the uniqueness of $Y_W$, the construction of the $\ssl_2$-triple is
functorial.  In other words, Theorem~\ref{t:DeligneSplittings} gives rise to a 
functor $\Nil_1\leadsto \Rep\, \ssl_2$ associating to every admissible
nilpotent orbit $(V,N,F,W)$ the representation $\rho_V:\ssl_2\to \End
V$ determined by the $\ssl_2$-triple $(N_0,H,N_+)$.  It follows that
there is a homomorphism $\SL_2\to \Aut^{\otimes}(\omega_1)$ where
$\omega_1:\Nil_1\to\Vect_{\mathbb{R}}$ is the forgetful functor.  By
restriction, we obtain a homomorphism
$h_{\SL_2}:\SL_2\to\mathfrak{M}_1=\Aut^{\otimes}(\omega_1|\Split_1)$.

On the other hand, if $(V,N,F,W)$ is in $\Split_1$, then, by
definition, the limit mixed Hodge structure $(V,F, M(N,W))$ is split.
Thus there is an action of Deligne's group $\mathbb{S}$ on $V$.  From
this, it follows that there is a a group homomorphism
$h_{\lim}:\mathbb{S}\to\mathfrak{M}_1$.

\subsection{}\label{s.MHS-T}
Let $T$ denote the split mixed Hodge structure $\mathbb{R}\oplus\mathbb{R}(1)$.
So $T_{\mathbb{R}}=\mathbb{C}$ with real basis $e:=1$ in $T_{\mathbb{C}}^{(0,0)}$ 
and $f:=2\pi i\in T_{\mathbb{C}}^{(-1,-1)}$.  
The action of
$\mathbb{S}$ on $T$ induces an embedding $i_{T}:\mathbb{S}\to\GL(T)$.
This in turn induces an action of $\mathbb{S}$ on $\SL(T)$ given by
$$
a(s)(\gamma):=i_T(s)\gamma i_T(s)^{-1}.
$$

With respect to the ordered basis $(e,f)$, the Lie algebra $\ssl(T)$ is, of course,
identified with $\ssl_2$.    It has real basis
$$
n_0:=\begin{pmatrix}
0 & 0\\
1 & 0\\
\end{pmatrix},
h:=\begin{pmatrix}
1 & 0\\
0 & -1\\
\end{pmatrix},
n_{+}:=\begin{pmatrix}
0 & 1\\
0 & 0\\
\end{pmatrix}.
$$
Moreover, from the Hodge structure on $T$, $\ssl(T)$ inherits a split mixed Hodge structure, 
and, thus, an action of $\mathbb{S}$.  It is easy to see that this action on $\ssl_2(T)$
is simply the Lie derivative on the above action on $a$ on $\SL_2(T)$.

\begin{lemma}\label{l.YWtypes}
  Let $(V,N,F,W)$ be an object in $\Split_1$.  Then
  \begin{enumerate}
  \item $Y_W=Y(N,Y_{(F,M)})$ is a morphism of the limit mixed Hodge structure
$V_{(F,M)}$;
\item $N_i\in \End V^{(-1,-1)}_{(F,M)}$ for all $i$ and each $N_i$ is real;
\item $N^{+}$ is a real element of $\End V^{(1,1)}_{(F,M)}$.  
  \end{enumerate}
\end{lemma}

\begin{proof}
  (i) Consider the action $\rho_{\lim}:\mathbb{S}\to \Aut V$ of
  $\mathbb{S}$ on $V$ induced by the split mixed Hodge structure
  $(F,M)$.  The map $Y_M$ is a morphism of the Hodge structure
  $(V,F,M)$.  Therefore $\mathbb{S}$ fixes $Y_M$.  On the other hand,
  $N$ induces a morphism $V\to V(1)$ relative the the Hodge structure
  $(V,F,M)$.  Therefore, $N$ is real of type $(-1,-1)$ relative to the
  mixed Hodge structure $(V,F,M)$.  Consequently, $N$ is fixed by
  $s(S^1)\subset \mathbb{S}$.  Now, it follows from the uniqueness of
  $M=M(N,W)$ and of $Y(N,Y_{(F,M)})$ that, if $g$ is any endomorphism
  of $V_{\mathbb{C}}$ fixing $W$, $Y_{(F,M)}$ and $N$, we have
  $gY_Wg^{-1}=Y(gNg^{-1},gY_{(F,M)}g^{-1})=Y_W$.  The filtration $W$
  on $V$ is by sub-mixed-Hodge structures of $(V,F,M)$.  Therefore,
  $W$ is fixed by $\mathbb{S}$.  So $Y_W$ is fixed by all $g\in
  s(S^1)$.  Consequently, $Y_W\in\oplus \End V^{(p,p)}_{(F,M)}$.
  Moreover, since $N$ and $Y_{(F,M)}$ are real, the uniqueness of
  $Y(N,Y_{(F,M)})$ shows that $Y_W$ is also real.  Since $Y_W$
  commutes with $Y_M$, we have $Y_W\in \End V^{(0,0)}_{(F,M)}$ as
  desired.   It follows that $Y_W$ is a morphism of $(V,F,M)$.

(ii) Let $\rho_W:\mathbb{G}_m\to\Aut V$ denote the action of $\mathbb{G}_m$ induced by the grading $Y_W$.
Then, by (i), we see that $\rho_W$ commutes with the action of $\mathbb{S}$ induced by the 
Hodge structure $(F,M)$.   Therefore, the group $\mathbb{G}_m\times\mathbb{S}$ acts on $V$ 
via $\rho_W\times\rho_{\lim}$.  We have already noted that $N\in\End V^{(-1,-1)}_{(F,M)}$ in (i).  It follows
from the fact that $\rho_W$ commutes with $\rho_{\lim}$ that $N_i\in\End V^{(-1,-1)}_{(F,M)}$ as well.  The uniqueness
of the $N_i$ shows that they are real.

(iii) Since $Y_M$ and $Y_W$ are both morphisms of the split mixed Hodge structure $(V,F,M)$, $H=Y_M-Y_W$ 
is as well.  Since $N_+$ is uniquely determined by the pair $(N_0, H)$, and $N_0$ and $H$ are both fixed
by the action $\rho_{\lim}\circ s:S^1\to\Aut V$, $N_+$ is also fixed by this $S^1$ action.  Moreover, since $N_0$ and 
$H$ are real $N_+$ is also real.   Therefore $N_+$ is a real element of $\oplus\End V^{(p,p)}_{(F,M)}$.  
On the other hand, since $[H,Y_W]=[N_0,Y_W]=0$, it follows that $[N_+,Y_W]=0$.  Therefore
$[Y_M,N_+]=[H+Y_W,N_+]=[H,N_+]+[Y_W,N_+]=[H,N_+]=2N_+$.  This shows that $N_+\in \End V^{(1,1)}_{(F,M)}$.  
\end{proof}

\begin{remark}\label{r.ds-2}
  Part
  (i) of Lemma~\ref{l.YWtypes} together with
  Corollary~\ref{c.SL2_Splitting} implies that $Y(N,Y_{(F,M)})$ and
  $\hat Y(e^{iN}.F,W)$ coincide (because both splittings are morphisms
  of the limit mixed Hodge structure).  This proves 
  Lemma~\ref{lemma:ds-2} (once Theorem~\ref{t.AppMain} is established). 
\end{remark}

\begin{corollary}\label{c.PreSemiDirect}
In $\mathfrak{M}_1$, we have 
\begin{equation}\label{e.HowSActs}
h_{\lim}(s)h_{\SL_2}(\gamma)h_{\lim}(s)^{-1}=
h_{\lim}(a(s)(\gamma)).
\end{equation}
\end{corollary}

\begin{proof}
  If $(V,N,F,W)$ is in $\Split_1$, we obtain a group
  homomorphism $\rho_V:\mathfrak{M}_1\to \Aut(V)$.  To prove the
  proposition, it suffices to show that, for any such $V$, the
  equation \eqref{e.HowSActs} holds when $\rho_V$ is applied to both
  sides.  This follows from Lemma~\ref{l.YWtypes} because the action $a$ of 
$\mathbb{S}$ on $\SL_2$ is exactly the one which give $n_0,h$ and $n_+$ the Hodge types
of the Lemma.
\end{proof}

\begin{corollary} Let $\mathbb{S}$ act on $\SL_2$ via $a$.  Then
there is a 
unique morphism 
$$
\rho_1:\SL_2\rtimes\mathbb{S}\to \mathfrak{M}_1
$$ 
such that  $h_{\SL_2}(\gamma)=\rho_1(\gamma,1)$ for $\gamma\in\SL_2$ 
and $h_{\lim}(\alpha)=\rho_1(1,\alpha)$ for $\alpha\in\mathbb{S}$.  
\end{corollary}

\begin{proof}
  This follows from Corollary~\ref{c.PreSemiDirect} and the definition
of the semi-direct product.
\end{proof}

\begin{proposition}\label{p.iW}
Consider the homomorphism $i_W:\mathbb{G}_m\to\SL_2\rtimes\mathbb{S}$
given by 
$$
i_W(\alpha)=(
\begin{pmatrix}
  \alpha^{-1} & 0\\
  0 & \alpha
\end{pmatrix}, w(\alpha)).
$$
Set $P_W:=i_W(\mathbb{G}_m)$.
Then $P_W$ is central in  
$\SL_2\rtimes\mathbb{S}$.
\end{proposition}

\begin{proof} 
We first show that $P_W$ centralizes $\SL_2$.  This follows
from the fact that the action of $w(\alpha)$ on $\SL_2$ coincides with the
adjoint action of the matrix
$$
\begin{pmatrix}
  \alpha & 0\\
  0 & \alpha^{-1}
\end{pmatrix}.
$$
On the other hand, $\mathbb{S}$ centralizes the subgroup $D$ of diagonal
matrices in $\SL_2$.   From this is follows that $P_W$ centralizes $\mathbb{S}$.
Therefore $P_W$ centralizes $\SL_2\rtimes\mathbb{S}$. 
\end{proof}

\begin{corollary}
Let $S$ denote the subgroup of $\SL_2\rtimes\mathbb{S}$ consisting
of elements of the form $(1,s(\beta))$ for $\beta\in S^1$.    Then $S$
is isomorphic to $S^1$ and central in $\SL_2\rtimes\mathbb{S}$.  Moreover the center
of $\SL_2\rtimes\mathbb{S}$ is the product $P_WS=P_W\times S$.   
\end{corollary}
\begin{proof}
Clearly $S$ is isomorphic to $S^1$, and 
$S$ is central in $\SL_2\rtimes\mathbb{S}$ because
$S$ acts trivially on $\SL_2$.   It is also clear that $S\cap P_W=\{1\}$.
Therefore $SP_W=S\times P_W$ is a subgroup of $\SL_2\rtimes\mathbb{S}$.  
To see that $S\times P_W$ is the connected component of the center, note
that $\SL_2\rtimes\mathbb{S}$ is reductive (as it is an extension of 
reductive groups).  It is of complex rank $3$ since ranks are additive
in extensions, and it has semi-simple rank $1$ since $\SL_2$ is its 
derived subgroup.  Therefore the connected component of the center must 
be a rank $2$ torus.   So it must
be $P_WS$. 
\end{proof}

\begin{corollary}
  Consider the morphism $\pi:\SL_2\times P_W\times S\to
  \SL_2\rtimes\mathbb{S}$ given by $(\gamma,\alpha,\tau)\mapsto
  \gamma\alpha\tau$.  It is a surjective homomorphism of
  algebraic groups with kernel a diagonally embedded copy of $\mu_2$.
  In other words, the homomorphism $\pi$ sets up an isomorphism
$$
\SL_2\rtimes\mathbb{S}\cong (\SL_2\times \mathbb{G}_m\times S^1)/\mu_2.
$$
\end{corollary}

\begin{proof}
  Suppose $G$ is a (connected) reductive group with derived subgroup
  $G^{\der}$ and with $Z(G)^0$ the connected component of the center.
  Then it is well-known from the theory of reductive groups that the
  morphism $G^{\der}\times Z(G)^0\to G$ given by $(g,z)\mapsto gz$ is
  a surjective and faithfully flat homomorphism with kernel isomorphic
  to $Z(G)^0\cap Z(G^{\der})$.   (The kernel consists of pairs $(g,g^{-1})$ 
where $g\in Z(G)^0\cap Z(G^{\der})$.)   

So, set $G=\SL_2\rtimes\mathbb{S}$.  Then $G^{\der}=\SL_2$, $Z(G)^0=P_WS$ and 
the intersection $Z(G)^0\cap Z(G^{\der})$ is simply $Z(G^{\der})\cong\mu_2$.  It is
embedded diagonally in the product $\SL_2\times P_W\times S$.  
\end{proof}

\subsection{} 
Set $G=\SL_2\rtimes\mathbb{S}$, and let $\tilde G=\SL_2\times
\mathbb{G}_m\times S^1$ so that $G=\tilde G/\mu_2$ where $\mu_2$ is
acting diagonally.  If $B_{\SL_2}$ denotes the upper triangular
matrices in $\SL_2$, then $\tilde B:=B_{\SL_2}\times
\mathbb{G}_m\times S^1$ is a Borel subgroup of $\tilde G$.  Let
$\tilde T$ denote the maximal torus in $\tilde G$ generated by the
diagonal matrices in the $\SL_2$ factor and the connected component of
the center $Z(\tilde G)^0=\mathbb{G}_m\times S^1$.  Write 
$\diag:\mathbb{G}_m\to\SL_2$ for the map 
$
z\mapsto 
\begin{pmatrix}
  z & 0\\
  0 & z^{-1}
\end{pmatrix}.
$
If we fix an isomorphism
$\spl_{S^1}:\mathbb{G}_m\otimes\mathbb{C}\to S^1_{\mathbb{C}}$, then we get an isomorphism
$
\spl_{\tilde T}=\diag\times \id\times \spl_{S^1}:\mathbb{G}_{m\mathbb{C}}^3 \to \tilde T_{\mathbb{C}}.
$
This gives an identification of $\mathbb{Z}^3$ with $X^*(\tilde T_{\mathbb{C}})$
such that, for $(a,b,c)\in\mathbb{Z}^3$, we have $(a,b,c)(r(z_1,z_2,z_3)=
z_1^az_2^bz_3^b$.  

Let $T$ denotes the image of $\tilde T$ in $G$, then the canonical
map $X^*(T)\to X^*(\tilde T)$ identifies $X^*(T)$ with the triples
$(a,b,c):2|a+b+c$.  Let $B$ denote the image of $\tilde B$ in $G$. The
action of $\Gal(\mathbb{C}/\mathbb{R})$ on $X^*(\tilde T)$ sends
$(a,b,c)$ to $(a,b,-c)$.  It follows easily from the theory of
representation of reductive groups that irreducible representations of
$G_{\mathbb{C}}$ are classified by weights which are positive with
respect to $B$ : that is complex representations of $G$ are classified
by triples $(a,b,c)$ such that $a\geq 0$ and $2|a+b+c$.   Write $V(a,b,c)$ 
for the representation associated to the triple $(a,b,c)$.   If $c=0$, then
$V(a,b,c)$ is defined over $\mathbb{R}$.  Otherwise, it follows from the 
theory of representations of real reductive groups (see~\cite{Tits}) that
there is a real representation $E(a,b,c)$ of $G$ such that $E(a,b,c)\otimes\mathbb{C}=V(a,b,c)\oplus V(a,b,-c)$.   Thus, if we set $E(a,b,0)=V(a,b,0)$, we 
find that the representations of $G$ are classified by triples $(a,b,c)$ 
such that $2|a+b+c$ and $a,c\geq 0$.   To sum up, we obtain the following:

\begin{theorem}
For each triple $(a,b,c)\in\mathbb{Z}^3$ with $2|a+b+c$ and $a,c\geq 0$,
there is a unique irreducible representation $E(a,b,c)$ of $G$ such that
$E(a,b,c)\otimes\mathbb{C}$ has a component of highest weight $(a,b,c)$ under
the above identification of $\mathbb{Z}^3$ with $X^*(\tilde T)$.  
\end{theorem}

\subsection{} Composing $\rho_1:G\to\mathfrak{M}_1$ with
$\pi_1:\mathfrak{M}_1\to \mathbb{S}_1$, we obtain a morphism
$\bar{\rho_1}:G\to\mathbb{S}_1$.  Thus we obtain a functor
$\bar{\rho_1}^*:\mathbf{S}_1\to\Rep G$.  We want to describe the image
of certain $\SL_2$ orbits under this functor.   To do this, consider
first the standard $\SL_2$-orbit
$\Standard:=(V,N,F,W)$ where
\begin{align*}
  V&=\mathbb{C}^2\, \text{with basis $e=(1,0)$, $f=(0,1)$};\\
  N&=
  \begin{pmatrix}
    0 & 0\\
    1 & 0\\
  \end{pmatrix};\\
   W_iV&=
   \begin{cases}
     0, & i<1,\\
     V, & i\geq 1;
   \end{cases}\\
    F^pV_{\mathbb{C}}&=
    \begin{cases}
      V_{\mathbb{C}}, & i<1,\\
        \mathbb{C}e, & i=1,\\
        0, & i>1.        
    \end{cases}
\end{align*}
Then the limit mixed Hodge structure associated to $\Standard$ is $\mathbb{R}\oplus \mathbb{R}(-1)$ and the $\ssl_2$ action is the standard representation. 
It follows that the action of $S\subset G$ on $V$ is trivial, and the 
relative weight filtration $M$ is given by 
$$
M_iV=\begin{cases}
0,  & i<0,\\
\mathbb{R}f, & i\in [0,1],\\
V,  &  i\geq 2.
\end{cases}
$$
Therefore, $Y_M(e)=2, Y_M(f)=0$.  So, $i_W(\alpha)\in G$ acts on $V$ by
$$\begin{pmatrix}
  \alpha^{-1} & 0\\
           0  & \alpha
\end{pmatrix}
\begin{pmatrix}
  \alpha^2 & 0\\
         0 & 1\\
\end{pmatrix}
=
\begin{pmatrix}
  \alpha & 0 \\
       0 & \alpha
\end{pmatrix}.
$$
It follows that $\bar\rho_1^*\Standard$ is isomorphic to $E(1,1,0)$.

For each non-negative integer $a$, let $S(a)$ denote the symmetric 
product $\Sym^a\Standard$.   Then $\bar\rho_1^*S(a)=E(a,a,0)$. 

The representations of $G$ associated to the constant variations are
very easy to describe.  The homomorphism $G\to\mathbb{S}$ obtained by
composition $\bar\rho_1$ with $\const:\mathbb{S}_1\to\mathbb{S}$ is
simply the canonical homomorphism
$G=\SL_1\rtimes\mathbb{S}\to\mathbb{S}$.  It follows that the Tate
Hodge structure $\mathbb{R}(k)$ corresponds to the constant
$\SL_2$-orbit $E(0,2k,0)$. 

 Suppose $p,q$ are integers with $p>q$.
Let $E(p,q)$ denote an irreducible
 real Hodge structure of weight $p+q$ with 
underlying vector space $H$ and with $\dim H^{p,q}=1$.  Then $S\subset G$ 
acts on $H_{\mathbb{C}}$ with weights $\pm p-q$ and $P_W\subset G$ acts
on $V$ via the character $z\mapsto z^{p+q}$.    Thus $\bar\rho_1^*E(p,q)\cong E(0,p+q,p-q)$.

\begin{corollary}\label{c.EssentiallySurjective}
  The tensor functor $\mathbf{S}_1\to \Rep G$ is essentially surjective.
\end{corollary}

\begin{proof}
Since $G$ is a real reductive group, all representations of $G$ are decomposable
into irreducibles.   Thus, it suffices to show that every irreducible representation of $G$ is in the essential image of $\mathbf{S}_1\to \Rep G$.

To do this, it suffices to note that, when $2|a+b+c$ and $a,c\geq 0$ then, 
\begin{equation}
  E(a,b,c)=
 \begin{dcases}
  \bar\rho_1^*((S(a)\otimes \mathbb{R}(\frac{b-a}{2})), & c=0,\\
   \bar\rho_1^*(S(a)\otimes E(\frac{b+c-a}{2},\frac{b-c-a}{2})), & c>0.
  \end{dcases}
\end{equation}
\end{proof}

\begin{theorem}
  The homomorphism $G\to\mathbb{S}_1$ is an isomorphism.
\end{theorem}

\begin{proof}
  It is equivalent to show that the tensor functor
  $\bar\rho_1^*:\mathbf{S}_1\to\Rep G$ is an equivalence of
  categories.  To do this, we need to show that $\bar\rho_1^*$ is
  fully faithful and essentially surjective.  That $\bar\rho_1^*$ is
  faithful is clear (because the functor
  $\mathbf{S}_1\to\Vect_{\mathbb{R}}$ if faithful.    That $\bar\rho_1^*$
is essentially surjective is the content of Corollary~\ref{c.EssentiallySurjective}.   

To show that $\bar\rho_1^*$ is full, it suffices to show that 
$\Hom_{\mathbb{S}_1}(V_1,V_2)\to\Hom_{G}(V_1,V_2)$ is surjective
when $\mathbf{V}_i=(V_i,N_i,W_i,F_i)$ are two pure objects in $\mathbf{S}_1$
of the same weight.   But this is clear because the $G=\SL_2\rtimes\mathbb{S}$ action 
on $V_i$ determines both $N_i$ and $F_i$: the action of $\SL_2$ 
determines $N_i$ and the action of $\mathbb{S}$ determines $F$. 
\end{proof}

\subsection{} From now on we use the theorem to identify
$\mathbb{S}_1$ with $G=\SL_2\rtimes\mathbb{S}$.  We want to explain
explicitly how to get an $\SL_2$-orbit from a representation
$\rho:\mathbb{S}_1\to \Aut(V)$.  As we said above, $N$ is determined
by the action of $\SL_2$.  It is the image of $n_0\in\ssl_2$ in $\End
V$.  The Hodge filtration $F$ is determined by the $\mathbb{S}$ action
as is the relative weight filtration $M$.  The weight filtration $W$
is the filtration of $V$ corresponding to the central co-character
$i_W$ of Proposition~\ref{p.iW}.  To check this, it suffices check
that it is true on the constant variations and on the standard
$\SL_2$-orbit $\Standard$.   If $V$ is an $\SL_2$-orbit, both $M$ and 
$W$ are canonically split by $Y_M$ and $Y_W=Y(N,Y_M)$.  Moreover $Y_W$ is a 
morphism of $\SL_2$-orbits.   We say that an element $v\in V$ is 
pure of weight $m$ for $M$ and $w$ for $W$ if $Y_M(v)=mv$ and $Y_W(v)=wv$.

\section{The Main Theorem}

\subsection{}\label{s.FMSGen} For each $z$ in the upper half-plane, $\mathfrak{h}$, Proposition~\ref{p.FunctorFromMixedToSmooth}
gives us a homomorphism $\spec_z:\mathfrak{M}\to\mathfrak{M}_1$.  On
the other hand, if $(V,N,F,W)$ is an object in $\mathbf{S}_1$ then
$(V,e^{zN}F,W)$ is split.  Thus we have a homomorphism
$\bar\spec_z:\mathbb{S}\to\mathbb{S}_1$ making the diagram
$$
\xymatrix{
\mathfrak{M}\ar[r]^{\pi}\ar[d]^{\spec_z}  & \mathbb{S}\ar[d]^{\bar\spec_z} \\
\mathfrak{M}_1\ar[r]^{\pi_1}      & \mathbb{S}_1.
}   
$$
commute.   To describe $\bar\spec_z$ explicitly, consider for each $z\in\mathfrak{h}$, let $s_i:S^1\to\SL_2$ denote the homomorphism
of real algebraic groups given by 
$$
(x,y)\mapsto 
\begin{pmatrix}
  x & -y\\
  -x & y
\end{pmatrix}$$
where we regard $S^1$ as $\Spec \mathbb{R}[x,y]/(x^2+y^2-1)$.  
For $z=u+iv\in\mathfrak{h}$, set 
$$
A_z=
\begin{pmatrix}
  1 & 0\\
  u & v\\
\end{pmatrix}, s_z(\sigma)=A_z\sigma A_z^{-1}.$$
We then have a homomorphism $S_z:S^1\to\mathbb{S}_1=\SL_2\rtimes\mathbb{S}$
given by $S_z(\sigma)=(s_z(\sigma), s(\sigma))$.
If we represent $\mathbb{S}$ via the quotient map
$w\times s:\mathbb{G}_m\times S^1\to \mathbb{S}$ with kernel $\mu_2$, 
we have
$\bar\spec_z(\lambda,\sigma)=i_W(\lambda)S_z(\sigma)$.
Note that $\bar\spec_z(-1,-1)=1$, so the map factors through to $\mathbb{S}$.

\subsection{}\label{s.E} For each integer $k\leq -2$, let $E_k$ denote the
$\mathbb{S}_1$ representation $E(-k-2,k,0)$.  So $E_{-2}$ corresponds
to the constant nilpotent orbit $\mathbb{R}(1)$ and $E_{-3}$ is a
$2$-dimensional non-constant nilpotent orbit of weight $-3$.  In
general, $E_k=(\Sym^{-k-2} E(1,-1,0))\otimes\mathbb{R}(1)$.  Set $E=\oplus
E_k$.  Since $E_k$ corresponds to a pure nilpotent orbit of weight
$k$, $E$ is a graded vector space.  In fact, the action of
$\mathbb{G}_m$ on $E$ via $i_W:\mathbb{G}_m\to\mathbb{S}_1$ induces the
grading.  

\subsection{}\label{s.E2} Here is one way to view $E$ as a representation of
$\mathbb{S}_1$.  Recall from \eqref{s.MHS-T} the mixed Hodge structure
$T\cong \mathbb{R}\oplus \mathbb{R}(1)$ with basis $e$ in the
$\mathbb{R}$ factor and $f=(2\pi i)$ in the $\mathbb{R}(1)$ factor.
Then $\Sym^*T$ can be viewed as the polynomial ring
$\mathbb{C}[e,f]$.  It inherits an $\ssl_2$ action described in \eqref{s.MHS-T},
which, together with the $\mathbb{S}$ action, induces a representation of
$\mathbb{S}_1=\SL_2\rtimes\mathbb{S}$.  It is isomorphic to the $\SL_2$-orbit
$E(1,-1,0)$.  

Then, if we let $\mathbb{R}(1)$ denote the constant $\SL_2$-orbit,
$E=(\Sym^*T)\otimes\mathbb{R}(1)$.  For each integer, $k\leq -2$,
write $n_k$ for $e^{-k+2}\otimes (2\pi i)\in E$.  Since $n_+ e=0$,
$n_+n_k=0$ for all $k\leq -2$.  The element $e\in T$ is in
$T_{(F,M)}^{(0,0)}$, so $n_k\in E_{(F,M)}^{(-1,-1)}$.  Similarly $f\in T_{(F,M)}^{(-1,-1)}$, so 
$E_k=M_{-1}E_k$ for all $k$.   Thus $E=M_{-1}E$.  Similarly, since $e$ has
weight $-1$ for $W$, $n_k$ has weight $-k+2-2=-k$ for $W$.

\subsection{}  Let $\mathfrak{L}_1$ denote the free Lie algebra on the graded
vector space $E$.  From $E$, $\mathfrak{L}_1$ inherits a grading:
$\mathfrak{L}_1=\oplus \mathfrak{L}_1(i)$.   In fact, the group
$\mathbb{S}_1$ acts on $\mathfrak{L}_1$.  
For each integer $m$, we
let $W_m\mathfrak{L}_1:=\oplus_{i\leq m} \mathfrak{L}_1(i)$.  Then it is 
rather easy to see that $W_n\mathfrak{L}_1$ is an ideal in $\mathfrak{L}_1$
stable under the action of $\mathbb{S}_1$. 
Moreover, the Lie algebras $\mathfrak{L}_1/W_m\mathfrak{L}_1$ are finite
dimensional, nilpotent Lie algebras equipped with compatible actions of 
$\mathbb{S}_1$.  
Let $K_1(n)$ denote the unipotent Lie group associated with 
$\mathfrak{L}_1/W_n\mathfrak{L}_1$.   Then $\mathbb{S}$ acts on 
$K_1(n)$ and, thus, on the affine group scheme $K_1
:=\varprojlim K_1(n)$.   
Set $G_1=K_1\rtimes\mathbb{S}_1$.

\begin{lemma}\label{l.KReps}
Let $\mathbf{U}_1$ denote the category of finite dimensional real
vector spaces $V$ equipped with a linear map $\rho_E:E\to\End V$ such that 
\begin{enumerate}
\item $\rho_E(u)$ is nilpotent for all $u\in E$;
\item $\rho_E(W_k E)=0$ for $k\ll 0$.
\end{enumerate}
Then the functor $\rho\mapsto \rho_{|E}$ is an equivalence from the 
category $\Rep K_1$ to the category $\mathbf{U}_1$.
\end{lemma}

\begin{proof}
We leave this exercise in unraveling the definition of $K_1$ to the 
reader.  
\end{proof}

\subsection{}\label{s.NsFromG1} Suppose $\rho:G_1\to \Aut V$ is a representation of $G_1$
on a finite dimensional vector space $V$.  The restriction of $\rho$
to $K_1$ induces a representation of the Lie algebra $\mathfrak{L}_1$
which is trivial on $W_n\mathfrak{L}_1$ for some $n$.  For $k\leq -2$,
let $N_k\in\End V$ denote the image of $n_k\in E\subset
\mathfrak{L}_1$.  The restriction of $\rho$ to the semi-direct factor
$\mathbb{S}_1$ of $G_1$ induces a representation of the Lie algebra
$\ssl_2$.  Let $N_0, H$ and $N_+$ denote the images of the elements
$n_0, h$ and $n_+$ respectively.  The $\mathbb{S}_1$-action on $V$
gives $V$ the structure of an $\SL_2$-orbit $V=(V,N_0,F,W)$.  Setting
$M=M(N_0,W)$ and using the definition of $G_1$ as a semi-direct
product, we find that $N_k\in\End V_{(F,M)}^{(-1,-1)}$ and
$Y_W(N_k)=-kN_k$.  Since $N_+(n_k)=0$, $[N_+,N_k]=0$ as well.

\subsection{}\label{s.CDef}
Consider the category $\mathbf{C}_1$ consisting of $\SL_2$-orbits
$V$ equipped with a family of nilpotent operators $N_k$ ($k\leq -2)$ 
such that $N_k\in \End V_{(F,M)}^{(-1,-1)}$, $Y_W(N_k)=-2N_k$ and $[N_+,N_k]=0$.
The morphisms in $\mathbf{C}_1$ are simply morphisms of $\SL_2$-orbits
respecting the $N_k$.  The category $\mathbf{C}_1$ has the structure
of a neutral Tannakian category in an obvious way.  
 By \eqref{s.NsFromG1}, we have a functor 
$\Res:\Rep G_1\to \mathbf{C}_1$
of Tannakian categories.

\begin{proposition}\label{p.GtoC}
  The functor $\Res:\Rep G_1\to\mathbf{C}_1$ is an equivalence.
\end{proposition}

\begin{proof}
  Since $G_1=K_1\rtimes\mathbb{S}_1$, to  give a representation $R$ of
  $G_1$ is the  same thing as to give  representations $R_K$ and $R_S$
  of   $K_1$   and   $\mathbb{S}_1$   respectively  such   that,   for
  $s\in\mathbb{S}_1$             and            $k\in            K_1$,
  $R_S(s)R_K(k)R_S(s)^{-1}=R_K(s.k)$  (where $s.k$ denotes  the effect
  of $s$ acting on $k$).   Since representations of $K$ are determined
  by  their restrictions to  $E$ (by  Lemma~\ref{l.KReps}), to  give a
  representation of $G_1$ on a  finite dimensional vector space $V$ is
  actually the same thing as giving representations $R_E:E\to \End V$
satisfying the conditions of the lemma along with the condition that
\begin{equation}\label{e.SemiOut}
R_S(s)R_E(e)R_S(s)^{-1}=R_E(s.e)\text{ for all $e\in E,
  s\in\mathbb{S}_1$.}
\end{equation}   

Suppose then that $V$ in $\mathbf{C}_1$ is given.  Then by definition
we have a representation $R_S:\mathbb{S}_1\to \Aut V$ which determines
operators $N_0$ and $N_+$ on $V$ along with nilpotent operators $N_k$
($k\leq -2$) satisfying $[N_+,N_k]=0$ for all $k$.  Using the fact that
$n_+.n_k$ and $[N_+,N_k]$ are both $0$ for $k\leq -2$, it is not hard
to see that $R_E(n_0^i.n_k):=(\ad N_0)^iN_k$ for $k\leq -2$
unambiguously defines a map $R_E:E\to \End V$
satisfying~\ref{e.SemiOut}.  Thus, from an object $V$ in
$\mathbf{C}_1$, we have a representation $R$ of $G_1$.  It is now
simple to check that $\Res (R)=V$.  Thus, we have an equivalence of
categories.
\end{proof}

\subsection{}  By Deligne's Theorem (Theorem~\ref{t:DeligneSplittings})
along with Lemma~\ref{l.YWtypes}), we have a Tannakian functor
$h_1^*:\Split_1\to \mathbf{C}_1$ sending a nilpotent orbit $(V,N,F,W)$ with 
limit split over $\mathbb{R}$ to the associated $\SL_2$-orbit 
$(V,N_0,F,W)$ along with the data of the $N_k$ for $k\leq -2$.  
By Proposition~\ref{p.GtoC}, this produces a homomorphism 
$h_1:G_1\to\mathfrak{M}_1$ of affine group schemes.

\begin{theorem}\label{t.h1}
  The homomorphism $h_1:G_1\to\mathfrak{M}_1$ is an isomorphism.
\end{theorem}

\begin{proof}
We need to show that the functor $\Split_1\to\mathbf{C}_1$ is fully
faithful and essentially surjective.   It is obvious that the functor
is faithful, and full is also easy.   So suppose $V$ is an object in 
$\mathbf{C}_1$.  Explicitly, $V$ consists of the data $(V,N_0,F,W)$
of an $\SL_2$-orbit together with operators $N_k$ for $k\leq -2$ 
satisfying the conditions in~\eqref{s.CDef}.  Set 
$N=N_0+\sum_{k\leq -2} N_k$.  Then we claim that $(V,N,F,W)$ is an
object in $\Split_1$.   

It follows directly from the definition of the relative weight
filtration, that the relative weight filtration $M(N,W)$ exists and is
equal to $M(N_0,W)$.  Since $\Gr^W N=\Gr^W N_0$, $(\Gr^W,
e^{izN}.F)=(\Gr^W,e^{izN_0}.F)$ is a split mixed Hodge structure for
all $z$ in the upper-half plane.  Thus, $(V,e^{izN}.F,W)$ is a mixed
Hodge structure for all such $z$.

To see that $(V,N,F,W)$ is an admissible nilpotent orbit, it remains
to check that $N(F^P)\subset F^{p-1}$ for all $p$.   This follows from the
fact that all the $N_i$ (including $N_0$) are in $\End V_{(F,M)}^{(-1,-1)}$.    

Since $M(N,W)=M(N_0,W)$, the object $(V,N,F,W)$ has limit split over
$\mathbb{R}$ and is, thus, in $\Split_1$.   It is  easy to see 
that $h_1^*(V,N,F,W)$ is the original object $V$ in $\mathbf{C}_1$ that 
we started out with.   So, $h_1^*:\Split_1\to G_1$ is an equivalence.
\end{proof}

\subsection{} Suppose $z$ is a complex number in the upper half plane.
Then we have $\spec_z:\mathfrak{M}\to\mathfrak{M}_1$ and,
by~\eqref{s.FMSGen}, $\spec_z(\mathfrak{U})\subset \mathfrak{U}_1$.  We
thus have a commutative diagram
$$\xymatrix{
  1\ar[r]& \mathfrak{U}\ar[r]\ar[d]^{\spec_z}&
  \mathfrak{M}\ar[r]^{\pi}\ar[d]^{\spec_z} &
  \mathbb{S}\ar[r]\ar[d]^{\overline{\spec}_z} & 1\\
  1\ar[r] & \mathfrak{U}_1\ar[r] & \mathfrak{M}_1\ar[r]^{\pi_1} &
  \mathbb{S}_1\ar[r] & 1.\\
}$$ 

We want to prove that the restriction of $\spec_z$ to $\mathfrak{U}$ is
an isomorphism onto $\mathfrak{U}_1$.   To begin, note the following.

\begin{lemma}
The affine group schemes $\mathfrak{U}$ and $\mathfrak{U}_1$ are abstractly
isomorphic by an isomorphism preserving the filtration $W$.
\end{lemma}

\begin{proof}
  It suffices to see that the Lie algebras $\mathfrak{L}$ and
  $\mathfrak{L}_1$ are isomorphic by an isomorphism preserving the
  filtration.  Now, $\mathfrak{L}_1$ is the free Lie algebra on the
  graded vector space $E$ while $\mathfrak{L}_{\mathbb{C}}$ is the
  free Lie algebra on the symbols $D^{i,j}$ with $i,j<0$.  It is
  easily seen that the real form $\mathfrak{L}$ of
  $\mathfrak{L}_{\mathbb{C}}$ is nothing but the free Lie algebra on
  the graded vector space $V=\oplus V_k$ where $\dim V_k=\{(i,j):
  i+j=k, i<0, j<0\}$.   Moreover, the grading on $\mathfrak{L}$ is the
one induced from the grading on $V$.   But $V$ is isomorphic to $E$ 
as a graded vector space, since $\dim E_k=\dim \Sym^{-k-2}\mathbb{R}^2=
-k-1=\dim V_k$. 
\end{proof}

Since $\mathfrak{U}$ and $\mathfrak{U}_1$ are pro-unipotent, it is
plausible to show that $\spec_z$ induces an isomorphism by showing
that the map on the abelianizations induced by $\spec_z$ is an
isomorphism.  

\begin{lemma}\label{l.JacobsonEx}
  Suppose $\mathfrak{g}$ is a nilpotent Lie algebra over a field $k$
  and $\mathfrak{h}$ is a subalgebra.  Suppose $\mathfrak{h}$ surjects
  onto $\mathfrak{g}/[\mathfrak{g},\mathfrak{g}]$.  Then
  $\mathfrak{h}=\mathfrak{g}$.
\end{lemma}

\begin{proof}
This is essentially Exercise 11 on page 29 of~\cite{JacobsonLieAlg}.
\end{proof}

Now, since $\mathfrak{U}/W_k\mathfrak{U}$ and
$\mathfrak{U}_1/W_k\mathfrak{U}_1$ are both nilpotent 
algebraic groups over $\mathbb{R}$, to show that $\spec_z$
induces an isomorphism, it suffices to show that $\spec_z$ is surjective.
By Lemma~\ref{l.JacobsonEx}, it suffices to show that $\spec_z$ 
is surjective on the abelianizations.   This is, in fact, what we
are going to do.

\begin{lemma}
  Suppose $\mathbf{H}$ is an $\SL_2$-orbit, and let 
$\mathbb{R}$ denote the constant $\SL_2$ orbit of weight $0$. Then 
$
\Ext^1_{\Split_1}(\mathbb{R},\mathbf{H})=\Hom(E,\mathbf{H})^{\mathbb{S}_1}$.
If $\mathbf{H}$ is irreducible and pure of weight $-k-2$, then we have
$$
\Ext^1_{\Split_1}(\mathbb{R},\mathbf{H})=
\begin{cases}
  \mathbb{R}, & \mathbf{H}=E_k; \\
           0, & \text{else}.
\end{cases}
$$
\end{lemma}
\begin{proof}
First note that, we can assume that $\mathbf{H}$ is irreducible of weight
$n$ for some integer $n$.   If $V$ is an extension of $\mathbb{R}$
by $\mathbf{H}$, then $W_{n-1}\mathfrak{M}_1$ acts trivially on 
$V$.    So, 
$\Ext^1_{\Split_1}(\mathbb{R},\mathbf{H})
=\Ext^1_{\mathfrak{M}_1(n-1)}(\mathbb{R},\mathbf{H})$.    We can compute
the later extension group by means of the inflation-restriction
sequence for the short exact sequence of algebraic groups
$$
1\to \mathfrak{U}_1(n-1)\to \mathfrak{M}_1(n-1)\to \mathbb{S}_1\to 1.
$$ 
We see that the extension group is $H^1(\mathfrak{M}_1(n-1),\mathbf{H})
=H^1(\mathfrak{U}_1(n-1),\mathbf{H})^{\mathbb{S}_1}=
\Hom(\mathfrak{U}_1(n-1)^{\mathrm{ab}},\mathbf{H})^{\mathbb{S}_1}
=\Hom(E,\mathbf{H})^{\mathbb{S}_1}$.  Here we use the fact that 
the abelianization of $\mathfrak{U}_1(n)$ is simply $\oplus_{k\geq n} E_k$.

The $E_k$ are all irreducible as $\mathbb{S}_1$ representations
even when base-changed to $\mathbb{C}$.   So $\Hom(E_k,E_k)^{\mathbb{S}_1}=\mathbb{R}$.   Since 
$E_k$ has weight $-k-2$, the formula for the extension group 
in the case that $\mathbf{H}$ is irreducible of weight $-k-2$ follows.  
\end{proof}

\begin{proposition}\label{p.dpq}
  Let $\mathbb{R}$ denote the constant $\SL_2$-orbit of weight $0$,
  and suppose $z$ is a number in the upper half-plane.  Let $V$ be a
  non-zero element in the extension group
  $\Ext^1_{\Split_1}(\mathbb{R},E_k)$.  Then, on $V$, 
each $\delta_{p,q}$ with $p+q=-k$, $p,q<0$ is non-zero.
\end{proposition}

\begin{proof}
  We have $N=N_0+N_k$ and $F(z)=e^{zN}.F$.  Since $V$ is split over
  $\mathbb{R}$, there is a unique element $u$ of $V_{(F,M)}^{(0,0)}$
  projecting onto $1$ in $\mathbb{R}$.  By multiplying the class of
  $V$ in the extension group
  $\Ext^1_{\Split_1}(\mathbb{R},E_k)\cong\mathbb{R}$, we can assume
  that $N_k(u)=e^{-k-2}\otimes (2\pi i)\in E_k$.

  We have $e\in T_{(F,M)}^{(0,0)}$ and $f\in T_{(F,M)}^{-1,-1}$.  
It follows that
$$
\mathbb{C}(e+zf)^a(e+\bar{z}f)^b\otimes (2\pi i)=
(E_k)_{F(z)}^{(-b-1,-a-1)}.
$$
Now, we have
$$
e=\frac{z(e+\bar{z}f)-\bar{z}(e+zf)}{z-\bar{z}}.
$$
So $$e^{-k-2}\otimes (2\pi i)=(z-\bar{z})^{-k-2}\sum_{a+b=-k-2} 
\binom{-k-2}{a}z^a(-\bar{z})^b
(e+\bar{z}f)^a(e+zf)^b.$$

We could use this to compute the $\delta_{p,q}$ directly.  
However, we only need to show that each $\delta_{p,q}$ is non-zero.   For this,
pick $p,q$ with $p,q<0$ and $p+q=-k$.    Let $\mathbf{H}$ denote the
irreducible Hodge substructure of $(E_k)_{F(z)}$ with a $\mathbf{H}_{F(z)}^{(p,q)}\neq 0$. 
Thus $\mathbf{H}$ is generated as a $\mathbb{C}$ vector space by $(e+\bar{z}f)^{-q-1}(e+zf)^{-p-1}\otimes (2\pi i)$
and $(e+\bar{z}f)^{-p-1}(e+zf)^{-q-1}\otimes (2\pi i)$.  
Note that $\Ext^1_{\MHS}(\mathbb{R},\mathbf{H})=\mathbf{H}_{\mathbb{C}}/(F^0(z)\mathbf{H}+\mathbf{H}_{\mathbb{R}})$.  
But $F^0(z)\mathbf{H}=0$, so $\Ext^1_{\MHS}(\mathbb{R},\mathbf{H})=\mathbf{H}_{\mathbb{C}}/\mathbf{H}_{\mathbb{R}}$.
The image of the extension class of $V$ in $\Ext^1_{\MHS}(\mathbb{R},\mathbf{H})$ is 
then the projection of $zN_k(u)=ze^{-k-2}\otimes (2\pi i)$ onto $\mathbf{H}_{\mathbb{C}}/\mathbf{H}_{\mathbb{R}}$. 
But, since each component in the above expression for $e^{-k-2}\otimes (2\pi i)$ 
is non-zero and $e^{-k-2}\otimes (2\pi i)$ is real, the 
image of $V$ in $\Ext^1_{\MHS}(\mathbb{R},\mathbf{H})$ is easily seen to be non-zero.
Therefore $\delta_{p,q}\neq 0$.
\end{proof}

\begin{corollary}
The map $\spec_z:\mathfrak{U}\to\mathfrak{U}_1$ is an isomorphism.
\end{corollary}

\begin{proof}
It suffices to show that, for each integer $n\leq 0$, the map of unipotent
real algebraic groups 
$\spec_z:\mathfrak{U}(n)\to \mathfrak{U}_1(n)$ is an isomorphism.   
For this it suffices to show that the maps  on the Lie algebras
are isomorphisms.  Since the Lie algebras are nilpotent and of the
same dimension over $\mathbb{R}$, it suffices to show that the
map induced on the abelianizations is surjective.  So, set $\mathfrak{L}(n):=\mathfrak{L}/W_n\mathfrak{L}$
and $\mathfrak{L}_1(n):=\mathfrak{L}_1/W_n\mathfrak{L}_1$.   It suffices to show that 
the map
$
s_z:\mathfrak{L}^{\mathrm{ab}}(n)\to \mathfrak{L}_1^{\mathrm{ab}}(n)$
induced by $\spec_z$ is an isomorphism from each $n$.   This map is $\mathbb{S}$-equivariant, 
where $\mathbb{S}$ is acting on $\mathfrak{L}_1(n)$ via conjugation.
For $p\leq q<0$ write $H(p,q)$ for the unique irreducible real pure Hodge structure
with $H(p,q)^{(p,q)}\neq 0$. Then, with the given $\mathbb{S}$-action,
$\mathfrak{L}^{\mathrm{ab}}(n)$ is a direct sum of 
the $H(p,q)$ such that $p+q>-n$ with each irreducible factor appearing 
exactly once.  Suppose the map $s_z$ has a kernel.  Then 
there is some $(p,q)$ with $H(p,q)$ in the kernel.  Set $p+q=k>n$. 
Then, if $V$ den totes a non-zero extension in $\Ext^1_{\Split_1}(\mathbb{R},E_k)$,
we will have $\delta_{p,q}=0$ for the Hodge structure $V_{(F(z),W}$.  This contradicts
Proposition~\ref{p.dpq}.
\end{proof}

\begin{proof}[Proof of Theorem~\ref{t.AppMain}] The only thing left to
  prove is part (iv) of the theorem.  Recall that
  $i_M:\mathbb{G}_m\to\mathfrak{M}_1$ denotes the homomorphism
inducing the splitting of the relative weight filtration $M$ and $H$ denotes
the image of $i_M$.  
Now all elements of $E_k$ are in $M_{-1}$ (see~\S\ref{s.E2}).   So the grading induced by $M$
on the free Lie algebra $\mathfrak{L}_1$ on $E$ puts all elements of $\mathfrak{L}_1$
in negative weight. 
Therefore, for each $n$, the intersection of $H$ with $\mathfrak{U}_1$ 
trivial.  

The fact that $\overline{\spec}_i(w(\mathbb{G}_m))$ is central in $\mathbb{S}_1$
was proved in Proposition~\ref{p.iW}. 
\end{proof}

\bibliographystyle{alpha}
%% input <zero-locus.bbl>
%% \bibliography{z2}
%%%%% BEGIN zero-locus.bbl%%%%%%%%%%%%%%%

%%%%% END  zero-locus.bbl%%%%%%%%%%%%%%%
\end{document}